\documentclass[oneside,english]{amsart}

\usepackage[T1]{fontenc}
\usepackage[latin9]{inputenc}
\usepackage{geometry}
\usepackage{color}
\usepackage{verbatim}
\usepackage{mathrsfs}
\usepackage{amsthm}
\usepackage{amssymb}
\usepackage{bbm}
\usepackage{tikz}
\usepackage{pgfplots}
\usepackage{babel}

\makeatletter

\numberwithin{equation}{section}

\newtheorem{theorem}{Theorem}[section]
\newtheorem{lemma}[theorem]{Lemma}

\newtheorem{corollary}[theorem]{Corollary}
\newtheorem{remark}[theorem]{Remark}
\newtheorem{definition}[theorem]{Definition}

\newcommand \N {\mathbb{N}} 
\newcommand \R {\mathbb{R}} 
\newcommand \ER {\mathrm{ER}}

\begin{document}

\title{Glauber dynamics on the Erd\H{o}s-R\'enyi random graph}

\author{F.\ den Hollander}
\address{Mathematical Institute, Leiden University, P.O.\ Box 9512,
2300 RA Leiden, The Netherlands.}
\email{denholla@math.leidenuniv.nl}

\author{O.\ Jovanovski}
\address{Bank of Montreal, 100 King Street West, Toronto, 
ON M5X 1A9 Canada.}
\email{oliver.jovanovski@bmo.com}

\begin{abstract}
We investigate the effect of disorder on the Curie-Weiss model with Glauber dynamics. 
In particular, we study metastability for spin-flip dynamics on the Erd\H{o}s-R\'enyi random 
graph $\ER_n(p)$ with $n$ vertices and with edge retention probability $p \in (0,1)$. Each 
vertex carries an Ising spin that can take the values $-1$ or $+1$. Single spins interact with 
an external magnetic field $h \in (0,\infty)$, while pairs of spins at vertices connected by an 
edge interact with each other with ferromagnetic interaction strength $1/n$. Spins flip according 
to a Metropolis dynamics at inverse temperature $\beta$. The standard Curie-Weiss model 
corresponds to the case $p=1$, because $\ER_n(1) = K_n$ is the complete graph on $n$ 
vertices. For $\beta>\beta_c$ and $h \in (0,p \chi(\beta p))$ the system exhibits \emph{metastable 
behaviour} in the limit as $n\to\infty$, where $\beta_c=1/p$ is the \emph{critical inverse temperature} 
and $\chi$ is a certain \emph{threshold function} satisfying $\lim_{\lambda\to\infty} \chi(\lambda) =1$ 
and $\lim_{\lambda \downarrow 1} \chi(\lambda)=0$. We compute the average crossover time 
from the \emph{metastable set} (with magnetization corresponding to the `minus-phase') to 
the \emph{stable set} (with magnetization corresponding to the `plus-phase'). We show that 
the average crossover time grows exponentially fast with $n$, with an exponent that is the 
same as for the Curie-Weiss model with external magnetic field $h$ and with ferromagnetic 
interaction strength $p/n$. We show that the correction term to the exponential asymptotics 
is a multiplicative error term that is \emph{at most polynomial} in $n$. For the complete 
graph $K_n$ the correction term is known to be a multiplicative constant. Thus, apparently, 
$\ER_n(p)$ is so homogeneous for large $n$ that the effect of the fluctuations in the disorder 
is small, in the sense that the metastable behaviour is controlled by the average of the disorder. 
Our model is the first example of a metastable dynamics on a random graph where the correction 
term is estimated to high precision. 
\end{abstract}

\thanks{The research in this paper was supported by NWO Gravitation Grant 024.002.003-NETWORKS}

\date{June 29, 2020}

\maketitle

\small
\tableofcontents
\normalsize

%%%%%%%%%%%%%%%%%%%%%%%%%%%%%%%%%%%%%%%%%%%%%%%%

\section{Introduction and main results}

In Section~\ref{ss:background} we provide some background on metastability. In 
Section~\ref{ss:model} we define our model: spin-flip dynamics on the 
Erd\H{o}s-R\'enyi random graph $\ER_n(p)$. In Section~\ref{ss:metpair} we identify 
the metastable pair for the dynamics, corresponding to the `minus-phase' and the 
`plus-phase', respectively. In Section~\ref{ss:perturbedCW} we recall the definition 
of spin-flip dynamics on the complete graph $K_n$, which serves as a comparison 
object, and recall what is known about the average metastable crossover time for 
spin-flip dynamics on $K_n$ (Theorem~\ref{thm: class CW result} below). In 
Section~\ref{ss:metatheorems} we transfer the sharp asymptotics for $K_n$ to a 
somewhat rougher asymptotics for $\ER_n(p)$ (Theorem~\ref{thm:metER} below). 
In Section~\ref{ss:disc} we close by placing our results in the proper context and 
giving an outline of the rest of the paper.

%%%

\subsection{Background}
\label{ss:background}

Interacting particle systems, evolving according to a \emph{Metropolis dynamics} 
associated with an energy functional called the \emph{Hamiltonian}, may end up 
being trapped for a long time near a state that is a local minimum but not a global 
minimum. The deepest local minima are called \emph{metastable states}, the 
global minimum is called the \emph{stable state}. The transition from a metastable 
state to the stable state marks the relaxation of the system to equilibrium. To describe 
this relaxation, it is of interest to compute the crossover time and to identify the set 
of critical configurations the system has to visit in order to achieve the transition. 
The critical configurations represent the saddle points in the energy landscape:
the set of mini-max configurations that must be hit by any path that achieves the 
crossover.

Metastability for interacting particle systems on \emph{lattices} has been 
studied intensively in the past three decades. Various different approaches 
have been proposed, which are summarised in the monographs by Olivieri 
and Vares~\cite{OV05}, Bovier and den Hollander~\cite{BdH15}. Recently, 
there has been interest in metastability for interacting particle systems on 
\emph{random graphs}, which is much more challenging because the crossover 
time typically depends in a delicate manner on the realisation of the graph.

In the present paper we are interested in metastability for spin-flip dynamics 
on the \emph{Erd\H{o}s-R\'enyi random graph}. Our main result is an estimate 
of the average crossover time from the `minus-phase'  to the `plus-phase' 
when the spins feel an external magnetic field at the vertices in the graph 
as well as a ferromagnetic interaction along the edges in the graph. Our paper
is part of a larger enterprise in which the goal is to understand metastability 
on graphs. Jovanovski~\cite{J17} analysed the case of the \emph{hypercube}, 
Dommers~\cite{D17} the case of the \emph{random regular graph}, Dommers, 
den Hollander, Jovanovski and Nardi~\cite{DdHJN17} the case of the 
\emph{configuration model}, and den Hollander and Jovanovski~\cite{dHJ17} 
the case of the \emph{hierarchical lattice}. Each case requires carrying out a 
detailed combinatorial analysis that is model-specific, even though the metastable 
behaviour is ultimately universal. For lattices like the hypercube and the hierarchical 
lattice a full identification of the relevant quantities is possible, while for random 
graphs like the random regular graph and the configuration model so far only 
the communication height is well understood, while the set of critical configurations 
and the prefactor remain somewhat elusive.    

The equilibrium behaviour of the Ising model on random graphs is well understood.
See e.g.\ Dommers, Giardin\`a and van der Hofstad~\cite{DGvdH10}, \cite{DGvdH14}.

%%%

\subsection{Spin-flip dynamics on $\ER_n(p)$}
\label{ss:model}

Let $\ER_n(p)=\left(V,E\right)$ be a realisation of the Erd\H{o}s-R\'enyi random graph 
on $\left|V\right|=n\in\N$ vertices with edge retention probability $p\in(0,1)$, i.e., each 
edge in the complete graph $K_{n}$ is present with probability $p$ and absent with 
probability $1-p$, independently of other edges (see Fig.~\ref{fig:ERRG}). We write
$\mathbb{P}_{\ER_n(p)}$ to denote the law of $\ER_n(p)$. For typical properties of $\ER_n(p)$, 
see van der Hofstad~\cite[Chapters 4--5]{vdH17}. 

%%%%%%%%%%%%%%%%%%%%%%%%%
\begin{figure}[htbp]
\includegraphics[width=.2\textwidth]{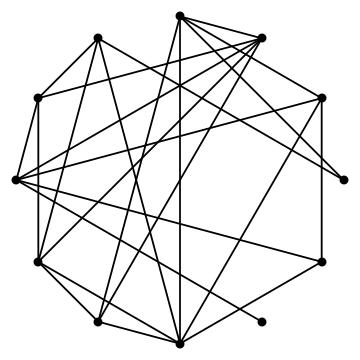}
\caption{A realization of the Erd\H{o}s-R\'enyi random graph with $n=12$ and $p=\tfrac13$.}
\label{fig:ERRG}
\end{figure}
%%%%%%%%%%%%%%%%%%%%%%%%%

\noindent
Each vertex carries an Ising spin that can take the values $-1$ or $+1$. Let $S_{n}
=\left\{-1,+1\right\}^{V}$ denote the set of spin configurations on $V$, and let $H_{n}$ 
be the \emph{Hamiltonian} on $S_{n}$ defined by
\begin{equation}
H_{n}\left(\sigma\right)=-\frac{1}{n}\sum_{\left(v,w\right)\in E}
\sigma(v)\sigma(w) -h\sum_{v\in V}\sigma(v),\quad\sigma\in S_{n}.
\label{eq:Hn}
\end{equation}
In other words, single spins interact with an \emph{external magnetic field} $h \in (0,\infty)$, while 
pairs of neighbouring spins interact with each other with a \emph{ferromagnetic coupling strength} 
$1/n$. 

Let $\boxminus=\left\{-1\right\} ^{V}$ and $\boxplus=\left\{+1\right\} ^{V}$ denote the 
configuration where all spins are $-1$, respectively, $+1$. Since 
\begin{equation}
H_{n}\left(\boxminus\right)=-\frac{\left|E\right|}{n}+hn,
\label{eq:H boxminus}
\end{equation}
we have the geometric representation
\begin{equation}
H_{n}\left(\sigma\right)=H_{n}\left(\boxminus\right)
+\frac{2}{n}\left|\partial_{E}\sigma\right|
-2h\left|\sigma\right|,\quad\sigma\in S_{n},
\label{eq:Hn as boundary size}
\end{equation}
where 
\begin{equation}
\partial_{E}\sigma=\left\{ \left(v,w\right)\in E\colon \sigma (v) = -\sigma (w) = +1\right\} 
\end{equation}
is the \emph{edge-boundary} of $\sigma$ and 
\begin{equation}
\left|\sigma\right|=\left\{ v\in \ER_n(p)\colon\,\sigma(v)=+1\right\} 
\end{equation}
is the \emph{vertex-volume} of $\sigma$. 

In the present paper we consider a spin-flip dynamics on $S_{n}$ commonly referred to as \emph{Glauber 
dynamics}, defined as the continuous-time Markov process with transition rates 
\begin{equation}
r\left(\sigma,\sigma\right)=
\begin{cases}
e^{-\beta [H_{n}(\sigma')-H_{n}(\sigma)]_{+}}, 
& \mbox{if }\left\Vert \sigma-\sigma'\right\Vert =2,\\
0, & \mbox{if }\left\Vert \sigma-\sigma'\right\Vert >2,
\end{cases}
\quad \sigma,\sigma' \in S_{n},
\label{eq:rate r}
\end{equation}
where $\|\cdot\|$ is the $\ell_1$-norm on $S_{n}$. This dynamics has as \emph{reversible} 
stationary distribution the Gibbs measure 
\begin{equation}
\label{eq:mu}
\mu_n\left(\sigma\right)=\frac{1}{Z_n} e^{-\beta H_{n}(\sigma)}, 
\quad \sigma \in S_{n},
\end{equation}
where $\beta \in (0,\infty)$ is the \emph{inverse temperature} and $Z_n$ is the 
normalizing partition sum. We write
\begin{equation}
\label{eq:ERdyn}
\left\{\xi_{t}\right\} _{t\geq0}
\end{equation} 
to denote the path of the random dynamics and $\mathbb{P}_\xi$ to denote its law given $\xi_0=\xi$. 
For $\chi \subset S_n$, we write
\begin{equation}
\tau_\chi = \inf\{t \geq 0\colon\, \xi_t \in \chi,\,\xi_{t-} \notin \chi\}.
\end{equation} 
to denote the \emph{first hitting/return time} of $\chi$.

We define the \emph{magnetization} of $\sigma$ by 
\begin{equation}
m\left(\sigma\right)=\frac{1}{n}\sum_{v\in V}\sigma(v),
\label{eq:magnetization}
\end{equation}
and observe the relation 
\begin{equation}
m\left(\sigma\right)=\frac{2\left|\sigma\right|}{n}-1, \quad \sigma \in S_{n}.
\label{eq:mag-vol-relat}
\end{equation}
We will frequently switch between working with volume and working with magnetization. 
Equation (\ref{eq:mag-vol-relat}) ensures that these are in one-to-one correspondence. 
Accordingly, we will frequently look at the dynamics from the perspective of the induced 
\emph{volume process} and \emph{magnetization process},
\begin{equation}
\label{eq:volmag}
\left\{\left|\xi_{t}\right|\right\}_{t\geq0}, \quad \left\{m(\xi_{t})\right\}_{t\geq0},
\end{equation} 
which are \emph{not} Markov.

%%%

\subsection{Metastable pair}
\label{ss:metpair}

For fixed $n$, the Hamiltonian in \eqref{eq:Hn} achieves a \emph{global minimum} at $\boxplus$ 
and a \emph{local minimum} at $\boxminus$. In fact, $\boxminus$ is the deepest local minimum 
not equal to $\boxplus$ (at least for $h$ small enough). However, in the limit as $n\to\infty$, these 
do \emph{not} form a metastable pair of configurations because \emph{entropy} comes into play. 

\begin{definition}[{\bf Metastable regime}]
\label{def:metreg} 
\rm
The parameters $\beta,h$ are said to be in the \emph{metastable regime} when
\begin{equation}
\label{eq:metreg}
\beta \in (1/p,\infty), \qquad h \in \big(0,p \chi(\beta p)\big)
\end{equation}
with (see Fig.~\ref{fig:threshold})
\begin{equation}
\label{eq:chidef}
\chi(\lambda) = \sqrt{1-\tfrac{1}{\lambda}} 
- \tfrac{1}{2\lambda} \log\left[\lambda\left(1+\sqrt{1-\tfrac{1}{\lambda}}\,\right)^{2}\right],
\qquad \lambda \geq 1.
\end{equation}

%%%%%%%%%%%%%%%%%%%%%%%%%%%%%%%%%%%%%%%%%%%
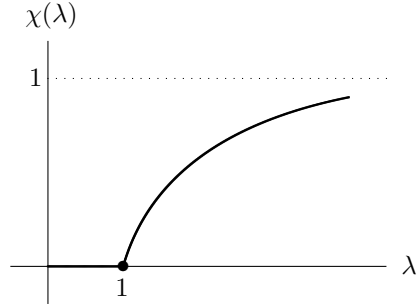
\begin{figure}[htbp]
\begin{center}
\setlength{\unitlength}{0.5cm}
\begin{picture}(10,6)(0,1)
\put(-1,0){\line(10,0){10}}
\put(0,-1){\line(0,7){7}}
\qbezier[40](0,5)(4.5,5)(9,5)
{\thicklines
\qbezier(0,0)(1,0)(2,0)
\qbezier(2,0)(3,3.5)(8,4.5)
}
\put(-.5,4.8){$1$}
\put(1.8,-.8){$1$}
\put(-.6,6.5){$\chi(\lambda)$}
\put(9.4,-.2){$\lambda$}
\put(2,0){\circle*{.3}}
\end{picture}
\end{center}
\vspace{1cm}
\caption{Plot of $\lambda\mapsto\chi(\lambda)$.}
\label{fig:threshold}
\end{figure}
%%%%%%%%%%%%%%%%%%%%%%%%%%%%%%%%%%%%%%%%%%%

\noindent
We have $\lim_{\lambda\to\infty} \chi(\lambda) = 1$ and $\lim_{\lambda \downarrow 1} \chi(\lambda) = 0$ 
(with slope $\tfrac12$). Hence, for $\beta \to \infty$ any $h \in (0,p)$ is metastable, while for $\beta 
\downarrow 1/p$ or $p \downarrow 0$ no $h \in (0,\infty)$ is metastable. The latter explains why we do 
not consider the non-dense Erd\H{o}s-R\'enyi random graph with $p=p_{n} \downarrow 0$ as $n\to\infty$.
\hfill$\blacksquare$
\end{definition}

\noindent
The threshold $\beta_c=1/p$ is the \emph{critical temperature}: the static model has a phase transition
at $h=0$ when $\beta>\beta_c$ and no phase transition when $\beta \leq \beta_c$ (see e.g.\ Dommers,
Giardin\`a and van der Hofstad~\cite{DGvdH14}).  

To define the proper metastable pair of configurations, we need the following definitions. Let 
\begin{equation}
\label{eq:InJndef}
\Gamma_n = \{-1,-1+\tfrac{2}{n},\ldots,1-\tfrac{2}{n},1\}, \quad
I_n(a) = - \frac{1}{n} \log \binom{n}{\tfrac{1+a}{2}n}, \quad
J_n(a) = 2\beta(pa+h) - 2 I'_n(a).
\end{equation} 
Define
\begin{equation}
\begin{aligned}
\mathbf{m}_n &= \min\left\{a \in \Gamma_n\colon\, J_n(a) \leq 0\right\},\\
\mathbf{t}_n &= \min\left\{a\in \Gamma_n\colon\, a>\mathbf{m}_n,\, J_n(a) \geq 0\right\},\\
\mathbf{s}_n &= \min\left\{a\in \Gamma_n\colon\, a>\mathbf{t}_n,\, J_n(a) \leq 0\right\}.
\end{aligned}
\label{eq: mts}
\end{equation} 
The numbers in the left-hand side of \eqref{eq: mts} play the role of magnetizations; it will be made clear below that when in the metastable regime, these quantities are well defined. 
Further define 
\begin{equation}
\label{eq:MTSdef}
\mathbf{M}_n= \frac{n}{2}(\mathbf{m}_n+1), \quad \mathbf{T}_n= \frac{n}{2}(\mathbf{t}_n+1), 
\quad \mathbf{S}_n= \frac{n}{2}(\mathbf{s}_n+1),
\end{equation}  
which are the volumes corresponding to \eqref{eq: mts}, and
\begin{equation}
A_k =\left\{ \sigma\in S_{n}\colon\,\left|\sigma\right|=k\right\},
\qquad k \in \{0,1,\ldots,n-1,n \},
\label{eq:def Ak}
\end{equation}
the set of configurations with volume $k$. Define
\begin{equation}
\label{eq:potential pre}
R_n(a)=-\frac{1}{2}pa^{2}-ha+\frac{1}{\beta}I_n(a)
\end{equation}
and note that
\begin{eqnarray}
\label{eq:diffR=J pre}
R'_n(a) = -pa-h+\frac{1}{\beta}I'_n(a) = - \frac{1}{2\beta} J_n(a).
\end{eqnarray}
The motivation behind the definitions in \eqref{eq:InJndef}, \eqref{eq: mts} and \eqref{eq:potential pre} will 
become clear in Section~\ref{s:preprep}. Via Stirling's formula it follows that
\begin{equation}
\label{eq:Stirling}
J_n(a) = 2\beta(pa+h) +\log\left(\frac{1-a+\frac{1}{n}}{1+a+\frac{1}{n}}\right) + O(n^{-2}), 
\quad a \in \Gamma_n.
\end{equation}  
We will see that, in the limit as $n\to\infty$ when $(\beta,h)$ is in the metastable regime defined by 
\eqref{eq:metreg}, the numbers in \eqref{eq: mts} are well-defined: $A_{\mathbf{M}_n}$ is the 
\emph{metastable set}, $A_{\mathbf{S}_n}$ is the \emph{stable set}, $A_{\mathbf{T}_n}$ is the 
\emph{top set}, i.e., the set of saddle points that lie in between $A_{\mathbf{M}_n}$ and $A_{\mathbf{S}_n}$. 
Our key object of interest will be the \emph{crossover time} from $A_{\mathbf{M}_n}$ to $A_{\mathbf{S}_n}$ 
via $A_{\mathbf{T}_n}$. 

Note that 
\begin{equation}
\label{eq:InJnlim}
\Gamma_n \to [-1,1], \quad I_n(a) \to I(a), \quad J_n(a) \to J_{p,\beta,h}(a), \qquad n\to\infty, 
\end{equation}
with 
\begin{equation}
\label{Jadef}
J_{p,\beta,h}(a) = 2\beta(pa+h) + \log\left(\frac{1-a}{1+a}\right)
\end{equation}
and
\begin{equation}
\label{Idef}
I(a) = \frac{1-a}{2} \log \left(\frac{1-a}{2}\right) + \frac{1+a}{2} \log \left(\frac{1+a}{2}\right).
\end{equation}
Accordingly, 
\begin{equation}
\mathbf{m}_n \to \mathbf{m}, \quad \mathbf{t}_n \to \mathbf{t}, \quad \mathbf{s}_n \to \mathbf{s}, 
\qquad n\to\infty,
\end{equation}
with $\mathbf{m}$, $\mathbf{t}$, $\mathbf{s}$ the three successive zeroes of $J$ (see Fig.~\ref{fig:CWfe} 
and recall \eqref{eq: mts}). Define 
\begin{equation}
\label{eq:potential}
R_{p,\beta,h}(a)=-\frac{1}{2}pa^{2}-ha+\frac{1}{\beta}I(a)
\end{equation}
Note that $R_{p,\beta,h}(a)$ plays the role of \emph{free energy}: $-\frac{1}{2}pa^{2}-ha$ and $\frac{1}{\beta} I(a)$ 
represent the energy, respectively, entropy at magnetisation $a$. Note that $I(a)$ equals the relative 
entropy of the probability measure $\tfrac12(1+a)\delta_{+1} + \tfrac12(1-a)\delta_{-1}$ with respect to 
the counting measure $\delta_{+1} + \delta_{-1}$. Also note that 
\begin{eqnarray}
\label{eq:diffR=J}
R'_{p,\beta,h}(a) = -pa-h+\frac{1}{\beta}I'(a) = - \frac{1}{2\beta} J_{p,\beta,h}(a).
\end{eqnarray}

\begin{remark}
\label{rem:hbounds} \rm 
As shown in Corollary~\ref{cor:pleqh} below, if $h \in (p,\infty)$, then (\ref{eq:rate r}) leads to 
\emph{non-metastable} behaviour where the dynamics `drifts' through a sequence of configurations 
with volume growing from $\mathbf{M}$ to $\mathbf{S}$ within time $O(1)$. \hfill$\blacksquare$
\end{remark}

%%%

\subsection{Metastability on $K_n$}
\label{ss:perturbedCW}

%%%%%%%%%%%%%%%%%%%%%%%%%
\begin{figure}[htbp]
\includegraphics[width=.2\textwidth]{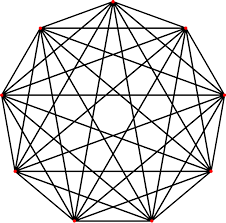}
\caption{The complete graph with $n=9$.}
\label{fig:complete}
\end{figure}
%%%%%%%%%%%%%%%%%%%%%%%%%

Let $K_n$ be the complete graph on $n$ vertices (see Fig.~\ref{fig:complete}). Spin-flip 
dynamics on $K_{n}$, commonly referred to as \emph{Glauber dynamics for the 
Curie-Weiss model}, is defined as in Section~\ref{ss:model} but with the \emph{Curie-Weiss 
Hamiltonian} 
\begin{equation}
\label{eq:CW}
H_{n}(\sigma)=-\frac{1}{2n} \sum_{1\leq i,j\leq n}\sigma(i)\sigma(j)
-h\sum_{1\leq i\leq n}\sigma(i),\quad\sigma\in S_{n}.
\end{equation}
This is the special case of \eqref{eq:Hn} when $p=1$, except for the diagonal term $-\frac{1}{2n}
\sum_{1 \leq i \leq n}\sigma(i)\sigma(i)=-\tfrac12$, which shifts $H_{n}$ by a constant and has no 
effect on the dynamics. The advantage of (\ref{eq:CW}) is that we may write
\begin{equation}
H_{n}(\sigma) = n \big[-\tfrac12 m(\sigma)^2-hm(\sigma)\big],
\end{equation}
which shows that the energy is a function of the magnetization only, i.e., the Curie-Weiss
model is a \emph{mean-field} model. Clearly, this property fails on $\ER_n(p)$.

For the Curie-Weiss model it is known that there is a \emph{critical inverse temperature} $\beta_c 
= 1$ such that, for $\beta > \beta_c$, $h$ small enough and in the limit as $n\to\infty$, the stationary 
distribution $\mu_n$ given by \eqref{eq:mu} and \eqref{eq:CW} has two phases: the `minus-phase', 
where the majority of the spins are $-1$, and the `plus-phase', where the majority of the spins 
are $+1$. These two phases are the \emph{metastable state}, respectively, the \emph{stable state} 
for the dynamics. In the limit as $n\to\infty$, the dynamics of the magnetization introduced in 
\eqref{eq:volmag} (which is Markov) converges to a Brownian motion on $[-1,+1]$ in the 
\emph{double-well potential} $a \mapsto R_{1,\beta,h}(a)$ (see Fig.~\ref{fig:CWfe}). 

%%%%%%%%%%%%%%%%%%%%%%%%%%%%%%%%%%%%%%%%%%%
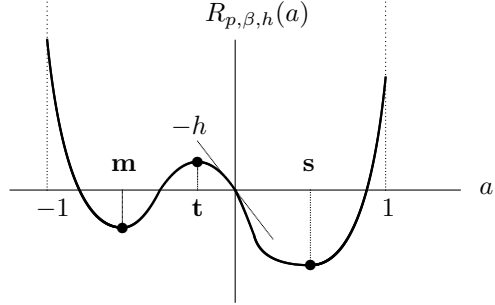
\begin{figure}[htbp]
\begin{center}
\setlength{\unitlength}{0.5cm}
\begin{picture}(10,6)(-4.5,-.4)
\put(-6,0){\line(12,0){12}}
\put(0,-3){\line(0,7){7}}
\qbezier[30](2,0)(2,-1)(2,-2)
\qbezier[30](-3,0)(-3,-.5)(-3,-1)
\qbezier[20](-1,0)(-1,.5)(-1,.75)
\qbezier[50](4,0)(4,2)(4,5)
\qbezier[50](-5,0)(-5,2)(-5,5)
\qbezier[80](1,-1.3)(0,0)(-1,1.3)
{\thicklines
\qbezier(.5,-1.2)(.7,-2)(2,-2)
\qbezier(.5,-1.2)(.2,-.4)(0,0)
\qbezier(2,-2)(3.5,-2)(4,3)
\qbezier(-3,-1)(-2.5,-1)(-2,0)
\qbezier(-3,-1)(-4.5,-1)(-5,4)
\qbezier(0,0)(-1,1.5)(-2,0)
}
\put(6.5,-.1){$a$}
\put(-.8,4.5){$R_{p,\beta,h}(a)$}
\put(-3.3,.5){$\mathbf{m}$}
\put(1.8,.5){$\mathbf{s}$}
\put(-1.15,-.7){$\mathbf{t}$}
\put(-1.7,1.5){$-h$}
\put(3.9,-.7){$1$}
\put(-5.3,-.7){$-1$}
\put(-1,.75){\circle*{.3}}
\put(2,-2){\circle*{.3}}
\put(-3,-1){\circle*{.3}}
\end{picture}
\end{center}
\vspace{1cm}
\caption{Plot of $R_{p,\beta,h}(a)$ as a function of the magnetization $a$. The metastable set 
$A_{\mathbf{M}}$ has magnetization $\mathbf{m}<0$, the stable set $A_{\mathbf{S}}$ 
has magnetization $\mathbf{s}>0$, the top set has magnetization $\mathbf{t}<0$. Note
that $R_{p,\beta,h}(-1) = -\tfrac12p+h$, $R_{p,\beta,h}(0) = - \beta^{-1}\log 2$, $R_{p,\beta,h}(+1) 
= -\tfrac12 p-h$ and $R_{p,\beta,h}'(-1) = -\infty$, $R_{p,\beta,h}'(0) = - h$, $R_{p,\beta, h}'(+1) = \infty$.}
\label{fig:CWfe}
\end{figure}
%%%%%%%%%%%%%%%%%%%%%%%%%%%%%%%%%%%%%%%%%%%

The following theorem can be found in Bovier and den Hollander~\cite[Chapter 13]{BdH15}.
For $p=1$, the metastable regime in \eqref{eq:metreg} becomes
\begin{equation}
\label{eq:metregCW}
\beta \in (1,\infty), \qquad h \in \big(0,\chi(\beta)\big).
\end{equation}

\begin{theorem}[{\bf Average crossover time  on $K_n$}]
\label{thm: class CW result}
Subject to \eqref{eq:metregCW}, as $n\to\infty$, uniformly in $\xi \in A_{\mathbf{M}_n}$, 
\begin{equation}
\mathbb{E}_{\xi}\left[\tau_{A_{\mathbf{S}_n}}\right]
= [1+o_n(1)]\,\frac{\pi}{1+\mathbf{t}} \sqrt{\frac{1-\mathbf{t}^{2}}{1-\mathbf{m}^{2}}}
\frac{1}{\beta\sqrt{R_{1,\beta,h}''(\mathbf{m})[-R_{1,\beta,h}''(\mathbf{t})]}}
\,e^{\beta n [R_{1,\beta,h}(\mathbf{t})-R_{1,\beta,h}(\mathbf{m})]}.
\end{equation}
\end{theorem}

\noindent
Fig.~\ref{fig:CWfe} illustrates the setting: the average crossover time from $A_{\mathbf{M}_n}$ to 
$A_{\mathbf{S}_n}$ depends on the energy barrier $R_{1,\beta,h}(\mathbf{t})-R_{1,\beta,h}(\mathbf{m})$ 
and on the curvature of $R_{1,\beta,h}$ at $\mathbf{m}$ and $\mathbf{t}$. Note that $\mathbf{m},\mathbf{s},
\mathbf{t}$ in Fig.~\ref{fig:CWfe} are the limits as $n\to\infty$ of $\mathbf{m}_n,\mathbf{s}_n,\mathbf{t}_n$
defined in \eqref{eq: mts} for $p=1$.

%%%

\subsection{Metastability on $\ER_n(p)$}
\label{ss:metatheorems}

Unlike for the spin-flip dynamics on $K_n$, the induced processes defined in \eqref{eq:volmag}
are \emph{not Markovian}. This is due to the random geometry of $\ER_n(p)$. However, we will 
see that they are \emph{almost Markovian}, a fact that we will exploit by comparing the dynamics 
on $\ER_n(p)$ with that on $K_n$, but with a ferromagnetic coupling strength $p/n$ rather than 
$1/n$ and with an external magnetic field that is a \emph{small perturbation} of $h$. 

As shown in Lemma~\ref{lem:doublewell} below, in the metastable regime the function $a \mapsto 
R_{p}(a)$ has a double-well structure just like in Fig.~\ref{fig:CWfe}, so that the metastable state 
$A_{\mathbf{M}}$ and the stable state $A_{\mathbf{S}}$ are separated by an \emph{energy barrier} 
represented by $A_{\mathbf{T}}$. 

We are finally in a position to state our main theorem.
\begin{theorem}[{\bf Average crossover time on $\ER_n(p)$}]
\label{thm:metER}
Subject to \eqref{eq:metregCW}, with $\mathbb{P}_{\ER_n(p)}$-probability tending to $1$ as $n\to\infty$,
and uniformly in $\xi \in A_{\mathbf{M}_n}$, 
\begin{equation}
\label{CWERasymp}
\mathbb{E}_{\xi}\left[\tau_{A_{\mathbf{S}_n}}\right]
= n^{E_n}\,e^{\beta n [R_{p,\beta,h}(\mathbf{t})-R_{p,\beta,h}(\mathbf{m})]}
\end{equation}
where the random exponent $E_n$ satisfies
\begin{equation}
\label{eq:expbd}
\lim_{n\to\infty} \mathbb{P}_{\ER_n(p)}\big(|E_n| \leq \beta (\mathbf{t}-\mathbf{m})\tfrac{11}{6}\big)=1.
\end{equation}
\end{theorem}

\noindent
Thus, apart from a polynomial error term, the average crossover time is the same as on the complete 
graph with ferromagnetic interaction strength $p/n$ instead of $1/n$.

%%%

\subsection{Discussion and outline}
\label{ss:disc}

We discuss the significance of our main theorem.

\medskip\noindent
{\bf 1.}
Theorem~\ref{thm:metER} provides an estimate on the average crossover time from
$A_{\mathbf{M}_n}$ to $A_{\mathbf{S}_n}$ on $\ER_n(p)$ (recall Fig.~\ref{fig:CWfe}). 
The estimate is \emph{uniform} in the starting configuration. The exponential term in 
the estimate is the \emph{same} as on $K_n$, but with a ferromagnetic interaction 
strength $p/n$ rather than $1/n$. The multiplicative error term is \emph{at most polynomial} 
in $n$. Such an error term is \emph{not} present on $K_{n}$, for which the prefactor is 
known to be a constant up to a multiplicative factor $1+o(1)$ (as shown in 
Theorem~\ref{thm: class CW result}). The randomness of $\ER_n(p)$ manifests 
itself through a more complicated prefactor, which we do not know how to identify. What 
is interesting is that, apparently, $\ER_n(p)$ is so homogeneous for large $n$ that the 
prefactor is at most polynomial. We expect the prefactor to be \emph{random} under
the law $\mathbb{P}_{\ER_n(p)}$.
 
\medskip\noindent
{\bf 2.} 
It is known that on $K_n$ the crossover time divided by it average has an exponential 
distribution in the limit as $n\to\infty$, as is typical for metastable behaviour. The same 
is true on $\ER_n(p)$. A proof of this fact can be obtained in a straightforward manner 
from the comparison properties underlying the proof of Theorem~\ref{thm:metER}. 
These comparison properties, which are based on \emph{coupling} of trajectories, 
also allow us to identify the \emph{typical set of trajectories} followed by the spin-flip 
dynamics prior to the crossover.  We will not spell out the details. 

\medskip\noindent
{\bf 3.}
The proof of Theorem~\ref{thm:metER} is based on estimates of transition probabilities 
and transition times between pairs of configurations with different volume, in combination 
with a \emph{coupling argument}. Thus we are following the \emph{path-wise} approach 
to metastabuilty (see \cite{BdH15} for background). Careful estimates are needed because 
on $\ER_n(p)$ the processes introduced in \eqref{eq:volmag} are \emph{not} Markovian, 
unlike on $K_n$. 

\medskip\noindent
{\bf 4.}
Bovier, Marello and Pulvirenti \cite{BMPpr} use capacity estimates and concentration of 
measure estimates to show that the prefactors form a \emph{tight} family of random variables
under the law $\mathbb{P}_{\ER_n(p)}$ as $n\to\infty$, which constitutes a considerable 
sharpening of \eqref{CWERasymp}. The result is valid for $\beta>\beta_c$ and $h$ small 
enough. The starting configuration is not arbitrary, but is drawn according to the 
\emph{last-exit-biased distribution} for the transition from $A_{\mathbf{M}_n}$ to 
$A_{\mathbf{S}_n}$, as is common in the \emph{potential-theoretic} approach to 
metastability. The exponential limit law is therefore not immediate.    

\medskip\noindent
{\bf 5.}
Another interesting model is where \emph{the randomness sits in the vertices rather
than in the edges}, namely, Glauber spin-flip dynamics with Hamiltonian
\begin{equation}
\label{eq:HamRF}
H_n(\sigma) = -\frac{1}{n} \sum_{1 \leq i,j \leq n} \sigma(i) \sigma(j)
- \sum_{1 \leq i \leq n} h_i\sigma(i),
\end{equation}
where $h_i$, $1 \leq i \leq n$, are i.i.d.\ random variables drawn from a common probability
distribution $\nu$ on $\R$. The metastable behaviour of this model was analysed in 
Bovier \emph{et al.}~\cite{BEGK01} (discrete $\nu$) and Bianchi \emph{et al.}~\cite{BBI09} 
(continuous $\nu$). In particular, the prefactor was computed up to a multiplicative factor 
$1+o(1)$, and turns out to be rather involved (see \cite[Chapters 14--15]{BdH15}). Our 
model is even harder because the interaction between the spins runs along the edges of 
$\ER_n(p)$, which has an \emph{intricate spatial structure}. Consequently, the so-called 
\emph{lumping technnique} (employed in \cite{BEGK01} and \cite{BBI09} to monitor the 
magnetization on the level sets of the magnetic field) can no longer be used. For the dynamics 
under \eqref{eq:HamRF} the exponential law was proved in Bianchi \emph{et al.}~\cite{BBI12}.

\medskip\noindent
{\bf Outline.} 
The remainder of the paper is organized as follows. In Section~\ref{s:preprep} we define 
the perturbed spin-flip dynamics on $K_n$ (Definition~\ref{def:PCW} below) and explain
why Definition~\ref{def:metreg} identifies the metastable regime (Lemma~\ref{lem:doublewell} 
below). In Section~\ref{s:prep} we collect a few basic facts about the geometry of $\ER_n(p)$ 
and the spin-flip dynamics on $\ER_n(p)$. In Section~\ref{s:capest} we derive rough capacity 
estimates for the spin-flip dynamics on $\ER_n(p)$. In Section~\ref{s:refcapest} we derive 
refined capacity estimates. In Section~\ref{sec:couplingscheme} we show that two copies 
of the spin-flip dynamics starting near the metastable state can be coupled in a short time.
In Section~\ref{s:proofmettheorem}, finally, we prove Theorem~\ref{thm:metER}.

%%%%%%%%% SECTION 2 %%%%%%%%%%%%%%%%%%%%%%%%%% 

\section{Preparations}
\label{s:preprep}

In Section~\ref{ss:perCW} we define the perturbed spin-flip dynamics on $K_n$ that will
be used as comparison object. In Section~\ref{ss:metaperturbedCW} we do a rough metastability 
analysis of the perturbed model. In Section~\ref{ss:doublewell} we show that $R_{p,\beta,h}$ has 
a double-well structure if and only if $(\beta,h)$ is in the metastable regime, in the sense of 
Definition~\ref{def:metreg} (Lemma~\ref{lem:doublewell} below). 

Define
\begin{equation}
J^*_n(a) = 2\beta\left(p\left(a+\tfrac{2}{n}\right)+h\right)
+\log\left(\frac{1-a}{1+a+\frac{2}{n}}\right), \quad a \in \Gamma_n.
\end{equation} 
We see from \eqref{eq:Stirling} that $J_n(a) = J^*_n(a) + O(n^{-2})$ when $\beta p = \frac{1}{1-a^2}$.
This will be useful for the analysis of the `free energy landscape'.   

%%%

\subsection{Perturbed Curie-Weiss}
\label{ss:perCW}

We will compare the dynamics on $\ER_n(p)$ with that on $K_n$, but with a ferromagnetic 
coupling strength $p/n$ rather than $1/n$, and with an external magnetic field field that is 
a \emph{small perturbation} of $h$. 

\begin{definition}[{\bf Perturbed Curie-Weiss}]
\label{def:PCW}
$\mbox{}$\\
{\rm (1) Let 
\begin{eqnarray}
\label{eq:Hu}
H_{n}^{u}\left(\sigma\right) &=& -\tfrac{p}{2n}\sum_{1\leq i,j\leq n}\sigma(i)\sigma(j)
-h_n^{u}\sum_{1\leq i\leq n}\sigma(i), \quad \sigma\in S_{n},\\
\label{eq:Hl}
H_{n}^{l}\left(\sigma\right) &=& -\tfrac{p}{2n}\sum_{1\leq i,j\leq n}\sigma(i)\sigma(j)
-h_n^{l}\sum_{1\leq i\leq n}\sigma(i), \quad\,\sigma\in S_{n},
\end{eqnarray}
be the Hamiltonians on $S_{n}$ corresponding to the Curie-Weiss model on $n$ vertices with 
ferromagnetic coupling strength $p/n$, and with external magnetic fields $h_n^{u}$ and $h_n^{l}$ 
given by 
\begin{equation}
h_n^{u} = h+\tfrac{\left(1+\epsilon\right)\log(n^{11/6})}{n},
\qquad  h_n^{l} = h-\tfrac{\left(1+\epsilon\right)\log(n^{11/6})}{n},
\label{eq: h perturbation}
\end{equation}
where $\epsilon>0$ is arbitrary. The indices $u$ and $l$ stand for upper and lower, and the 
choice of exponent $\tfrac{11}{6}$ will become clear in Section~\ref{s:capest}.\\
(2) The equilibrium measures on $S_{n}$ corresponding to (\ref{eq:Hu}) and (\ref{eq:Hl}) are 
denoted by $\mu_n^{u}$ and $\mu_n^{l}$, respectively (recall \eqref{eq:mu}).\\
(3) The Glauber dynamics based on  (\ref{eq:Hu}) and (\ref{eq:Hl}) are denoted by 
\begin{equation}
\label{eq:Kndyn}
\left\{\xi_{t}^{u}\right\}_{t\geq 0}, \qquad \left\{\xi_{t}^{l}\right\}_{t\geq 0}, 
\end{equation}
respectively.\\
(4) The analogues of \eqref{eq: mts} and \eqref{eq:MTSdef} are denoted by $\mathbf{m}^u_n,\mathbf{t}^u_n,
\mathbf{s}^u_n$, $\mathbf{M}^u_n,\mathbf{T}^u_n,\mathbf{S}^u_n$ and $\mathbf{m}^l_n,\mathbf{t}^l_n,\mathbf{s}^l_n$, 
$\mathbf{M}^l_n,\mathbf{T}^l_n,\mathbf{S}^l_n$, respectively.} 
\hfill$\blacksquare$
\end{definition}

In what follows we will \emph{suppress the $n$-dependence from most of the notation}. Almost all of the 
analysis in Sections~\ref{s:preprep}--\ref{s:proofmettheorem} pertains to the dynamics on 
$\ER_n(p)$.   

%%%

\subsection{Metastability for perturbed Curie-Weiss}
\label{ss:metaperturbedCW}

Recall that $\{\xi_{t}^{u}\}_{t\geq0}$ and $\{\xi_{t}^{l}\} _{t\geq0}$ denote the Glauber dynamics for the 
Curie-Weiss model driven by (\ref{eq:Hu}) and (\ref{eq:Hl}), respectively. An important feature is that their 
magnetization processes 
\begin{equation}
\begin{aligned}\
\left\{\theta_{t}^{u}\right\} _{t\geq0} 
& =\{m(\xi_{\tau_{s}^{l}}^{l})\} _{t\geq0}\\
\left\{\theta_{t}^{l}\right\} _{t\geq0} 
& =\{m(\xi_{\tau_{s}^{u}}^{u})\} _{t\geq0}
\end{aligned}
\label{eq: theta u l}
\end{equation}
are continuous-time Markov processes themselves (see e.g.\ Bovier and den 
Hollander~\cite[Chaper 13]{BdH15}). The state space of these two processes is 
$\Gamma_n=\left\{ -1,-1+\frac{2}{n},\ldots,1-\frac{2}{n},1\right\}$, and the transition rates are 
\begin{eqnarray}
\label{eq: qu}
&&q^{u}\left(a,a'\right)=
\begin{cases}
\frac{n}{2}\left(1-a\right)e^{-\beta [p(-2a-\frac{2}{n})-2h^{u}]_{+}}, 
& \mbox{if }a'=a+\frac{2}{n},\\
\frac{n}{2}\left(1+a\right)e^{-\beta [p(2a+\frac{2}{n})+2h^{u}]_{+}}, 
& \mbox{if }a'=a-\frac{2}{n},\\
0, & \mbox{otherwise},
\end{cases}
\\
\label{eq:ql}
&&q^{l}\left(a,a'\right)=
\begin{cases}
\frac{n}{2}\left(1-a\right) e^{-\beta [p(-2a-\frac{2}{n})-2h^{l}]_{+}}, 
& \mbox{if }a'=a+\frac{2}{n},\\
\frac{n}{2}\left(1+a\right) e^{-\beta [p(2a+\frac{2}{n})+2h^{l}]_{+}}, 
& \mbox{if }a'=a-\frac{2}{n},\\
0, & \mbox{otherwise},
\end{cases}
\end{eqnarray}
respectively. The processes in \eqref{eq: theta u l} are reversible with respect to the Gibbs measures
\begin{eqnarray}
\label{eq:nu-u}
\nu^{u}\left(a\right) &=& \frac{1}{z^{u}} e^{\beta n(\tfrac12 pa^{2}
+h^{u}a)}{n \choose \frac{1+a}{2}n},\quad a\in \Gamma_n,\\
\label{eq:nu-l}
\nu^{l}\left(a\right) &=& \frac{1}{z^{l}} e^{\beta n(\tfrac12 pa^{2}
+h^{l}a)}{n \choose \frac{1+a}{2}n},\quad a\in \Gamma_n,
\end{eqnarray}
respectively. 

Define 
\begin{eqnarray}
\label{eq:Psi u def}
\Psi^{u}\left(a\right) &=& -\tfrac12 pa^{2}-h^{u}a,\quad a\in \Gamma_n,\\
\label{eq:Psi l def}
\Psi^{l}(a) &=& -\tfrac12 pa^{2}-h^{l}a,\quad a\in \Gamma_n.
\end{eqnarray}
Note that (\ref{eq: qu}) and (\ref{eq:nu-u}) can be written as 
\begin{eqnarray}
q^{u}\left(a,a+\tfrac{2}{n}\right) & = & \tfrac{n}{2}\left(1-a\right)
e^{-\beta n [\Psi^{u}(a+\tfrac{2}{n})-\Psi^{u}(a)]_{+}},\nonumber\\
\nu^{u}\left(a\right) 
& = & \frac{1}{z^{u}} e^{-\beta n\Psi^{u}(a)}
{n \choose \frac{n}{2}\left(1+a\right)},
\label{eq:mu q Psi relation}
\end{eqnarray}
and similar formulas hold for (\ref{eq:ql}) and (\ref{eq:nu-l}). The properties of the function $\nu^{u}\colon 
\Gamma_n\to\left[0,1\right]$ can be analysed by looking at 
the ratio of adjacent values:
\begin{equation}
\frac{\nu^{u}\left(a+\tfrac{2}{n}\right)}{\nu^{u}\left(a\right)}
=\exp\left(2\beta\left(p\left(a+\tfrac{2}{n}\right)+h^{u}\right)
+\log\Big(\tfrac{1-a}{1+a+\tfrac{2}{n}}\Big)\right),
\label{eq:ratio mu}
\end{equation}
which suggests that `local free energy wells' in $\nu^{u}$ can be found by looking at where 
the sign of
\begin{equation}
\label{eq:sumtwo}
2\beta\left(p\left(a+\tfrac{2}{n}\right)+h^{u}\right)+\log\left(\tfrac{1-a}{1+a+\frac{2}{n}}\right)
\end{equation}
changes from negative to positive. To that end note that, in the limit $n\to\infty$, the second 
term is positive for $a<0$, tends to $\infty$ as $a \to -1$, is negative for $a\geq0$, tends to 
$-\infty$ as $a \to 1$, and tends to $0$ as $a\to0$. The first term is linear in $a$, and for 
appropriate choices of $p$, $\beta$ and $h^{u}$ (see Definition~\ref{def:metreg}) is negative 
near $a=-1$ and becomes positive at some value $a<0$. This implies that, for appropriate 
choices of $p$, $\beta$ and $h^{u}$, the sum of the two terms in \eqref{eq:sumtwo} can 
change sign $+\to-\to+$ on the interval $\left[-1,0\right]$, and can change sign $+\to-$ on 
$\left[0,1\right]$. Assuming that our choice of $p$, $\beta$ and $h^{u}$ corresponds to this 
change-of-signs sequence, we define $\mathbf{m}^{u}$, $\mathbf{t}^{u}$ and $\mathbf{s}^{u}$
as in \eqref{eq: mts} with $h$ replaced by $h^{u}$. This observation makes it clear that the 
sets in the right-hand side of \eqref{eq: mts} indeed are non-empty. 

The interval $\left[\mathbf{m}^{u},\mathbf{t}^{u}\right]$ poses a barrier for the process 
$\{\theta_{t}^{u}\}_{t \geq 0}$ due to a negative drift, which delays the initiation of the 
convergence to equilibrium while the process passes through the interval $\left[\mathbf{t}^{u},
\mathbf{s}^{u}\right]$. The same is true for the process $\{\xi_{t}^{u}\}_{t\geq 0}$. 
Similar observations hold for  $\{\theta_{t}^{l}\}_{t \geq 0}$ and $\{\xi_{t}^{l}\}_{t\geq 0}$. 
Recall Fig.~\ref{fig:CWfe}.

%%%

\subsection{Double-well structure}
\label{ss:doublewell} 

\begin{lemma}[{\bf Metastable regime}]
\label{lem:doublewell}
The potential $R_{p,\beta,h}$ defined in \eqref{eq:potential} has a double-well structure if and 
only if $\beta p > 1$ and $0 < h < p\chi(\beta p)$, with $\chi$ defined in \eqref{eq:chidef}.
\end{lemma}

\begin{proof}
In order for $R_{p,\beta,h}$ to have a double-well structure, the measure $\nu$ must have 
two distinct maxima on the interval $\left(-1,1\right)$. From \eqref{eq:InJnlim}, \eqref{eq:diffR=J} 
and \eqref{eq:ratio mu} it follows that  
\begin{equation}
J_{p,\beta,h}(a)=2\lambda\left(a+\tfrac{h}{p}\right)+\log\left(\tfrac{1-a}{1+a}\right),
\qquad \lambda=\beta p,
\end{equation}
must have one local minimum and two zeroes in $\left(-1,1\right)$. Since 
\begin{equation}
\label{eq:Jprime}
J_{p,\beta,h}'(a)=2\left(\lambda-\tfrac{1}{1-a^2}\right), \qquad a \in [-1,1],
\end{equation}
it must therefore be that $\lambda>1$. The local minimum is attained when 
\begin{equation}
\lambda=\tfrac{1}{1-a^{2}},
\end{equation}
i.e., when $a=a_\lambda=-\sqrt{1-\frac{1}{\lambda}}$ ($a_\lambda$ must be negative
because it lies in $(\mathbf{m},\mathbf{t})$; recall Fig.~\ref{fig:CWfe}). Since
\begin{equation}
0 > J_{p,\beta,h}(a_\lambda)=2\lambda\left(a_\lambda+\tfrac{h}{p}\right)
+\log\left(\tfrac{1-a_\lambda}{1+a_\lambda}\right),
\end{equation}
it must therefore be that 
\begin{equation}
\tfrac{h}{p} < \chi(\lambda)
\end{equation}
with $\chi(\lambda)$ given by \eqref{eq:chidef}.
\end{proof}

%%%%%%%%% SECTION 3 %%%%%%%%%%%%%%%%%%%%%%%%%%

\section{Basic facts}
\label{s:prep}

In this section we collect a few facts that will be needed in Section~\ref{s:capest}
to derive capacity estimates for the dynamics on $\ER_n(p)$. In Section~\ref{ss:concERnp} 
we derive a large deviation bound for the degree of typical vertices $\ER_n(p)$
(Lemma~\ref{lem: degree bound} below). In Section~\ref{ss:edgeERnp} we do the same 
for the edge-boundary of typical configurations (Lemma~\ref{lem: phi i k bounds} below). 
In Section~\ref{ss:jumpvol} we derive upper and lower bounds for the jump rates of the volume 
process (Lemmas~\ref{lem: Rate bounds}--\ref{lem:H near m} and Corollary~\ref{cor:pleqh}
below), and show that the return times to the metastable set \emph{conditional} on not hitting
the top set are small (Lemma~\ref{lem:returntail} below). In Section~\ref{ss:unifstart} we use the various 
bounds to show that the probability for the volume process to grow by $n^{1/3}$ is almost 
uniform in the starting configuration (Lemma~\ref{lem:ratio probab} and Corollary~\ref{cor:expratio} 
below).

\begin{definition}[{\bf Notation}]
\label{def:notsets}
{\rm For a vertex $v\in V$, we will write $v\in\sigma$ to mean $\sigma(v)=+1$ and $v\notin\sigma$ 
to mean $\sigma(v)=-1$. Similarly, we will denote by $\overline{\sigma}$ the configuration obtained 
from $\sigma$ by flipping the spin at every vertex, i.e., $\sigma(v)=+1$ if and only if $\overline{\sigma}
(v)=-1$. For two configurations $\sigma,\sigma'$ we will say that $\sigma\subseteq\sigma'$ if and 
only if $v\in\sigma\implies v\in\sigma'$. By $\sigma\cup\sigma'$ we denote the configuration satisfying 
$v\in\sigma\cup\sigma'$ if and only if $v\in\sigma$ or $v\in\sigma'$. A similar definition applies to 
$\sigma\cap\sigma'$. We will also write $\sigma\sim\sigma'$ when there is a $v\in V$ such that 
$\sigma=\sigma'\cup\left\{v\right\} $ or $\sigma'=\sigma\cup\left\{v\right\}$. We will say that $\sigma$ 
and $\sigma'$ are neighbours. We write $\deg(v)$ for the degree of $v \in V$.} \hfill$\blacksquare$
\end{definition}

%%%

\subsection{Concentration bounds for $\ER_n(p)$}
\label{ss:concERnp}

Recall that $\mathbb{P}_{\ER_n(p)}$  denotes the law $\ER_n(p)$.

\begin{lemma}[{\bf Concentration of degrees and energies}]
\label{lem: degree bound}
With $\mathbb{P}_{\ER_n(p)}$-probability tending to $1$ as $n\to\infty$ the following is true.
For any $\epsilon > 0$ and any $c>\sqrt{\frac{1}{8}\log2}$, 
\begin{eqnarray}
\label{eq:degree conc. ineq}
&pn-(1+\epsilon) \sqrt{n\log n}<\deg(v)
<pn+(1+\epsilon) \sqrt{n\log n} \qquad \forall\,v \in V,\\[0.2cm]
\label{eq:H conc ineq.}
&\frac{1}{n}\left(2p|\xi|(n-|\xi|)-cn^{3/2}\right)-2h|\xi|
\leq H_n(\xi)-H_n(\boxminus)\\ \nonumber
&\qquad\qquad\qquad \leq \frac{1}{n} \left(2p|\xi|(n-|\xi|)+ cn^{3/2}\right) - 2h|\xi| 
\qquad \forall\,\xi\in S_n.
\end{eqnarray}
\end{lemma}

\begin{proof}
These bounds are immediate from Hoeffding's inequality and a union bound.
\end{proof}

%%%

\subsection{Edge boundaries of $\ER_n(p)$}
\label{ss:edgeERnp}

We partition the configuration space as 
\begin{equation}
S_{n}=\bigcup_{k=0}^{n}A_{k},
\end{equation}
where $A_{k}$ is defined in \eqref{eq:def Ak}. For $0\leq k\leq n$ and $-pk\left(n-k\right)\leq 
i\leq\left(1-p\right)k\left(n-k\right)$, define 
\begin{equation}
\phi_{i}^{k}=\left|\left\{ \sigma\in A_{k}\colon\,\left|\partial_{E}\sigma\right|=pk\left(n-k\right)+i\right\}\right|,
\label{eq:phi}
\end{equation}
i.e., $\phi_{i}^{k}$ counts the configurations $\sigma$ with volume $k$ whose edge-boundary size 
$\left|\partial_{E}\sigma\right|$ deviates by $i$ from its mean, which is equal to $pk\left(n-k\right)$.
For $0\leq k\leq n$, let $\mathbb{P}_k$ denote the uniform distribution on $A_k$. 

\begin{lemma}[{\bf Upper bound on edge-boundary sizes}]
\label{lem: phi i k bounds}
With $\mathbb{P}_{\ER_n(p)}$-probability tending to $1$ as $n\to\infty$ the following are true.
For $-pk(n-k)\leq j\leq (1-p)k(n-k)$ and $\varrho\colon\,\mathbb{N}\to\mathbb{R}_{+}$,
\begin{equation}
\mathbb{P}_k\left[\phi_{j}^{k}\geq\varrho\left(n\right){n \choose k}p^{pk(n-k)+j}
(1-p)^{(1-p)k(n-k)-j}{k(n-k) \choose pk(n-k)+j}\right]\leq\frac{1}{\varrho(n)}
\label{eq:phi k i}
\end{equation}
and 
\begin{equation}
\begin{aligned}
\mathbb{P}_k\left[\sum_{j\geq i}\phi_{j}^{k}\geq\varrho\left(n\right){n \choose k}
e^{-\tfrac{2i^{2}}{k(n-k)}}\right] 
& \leq\frac{1}{\varrho\left(n\right)},\\
\mathbb{P}_k\left[\sum_{j\leq-i}\phi_{j}^{k}\geq\varrho\left(n\right){n \choose k}
e^{-\tfrac{2i^{2}}{k(n-k)}}\right] 
& \leq\frac{1}{\varrho\left(n\right)}.
\end{aligned}
\label{eq:phi k i 2}
\end{equation}
\end{lemma}

\begin{proof}
Write $\simeq$ to denote equality in distribution. Note that if $\sigma\simeq\mathbb{P}_k$, 
then $\left|\partial_{E}\sigma\right|\simeq\mathrm{Bin}\left(k\left(n-k\right),p\right)$, and hence 
\begin{equation}
\mathbb{P}_k\left[\left|\partial_{E}\sigma\right|=i\right]=p^{i}
\left(1-p\right)^{k\left(n-k\right)-i}{k\left(n-k\right) \choose i}.
\end{equation}
In particular, 
\begin{equation}
\mathbb{E}_k\left[\phi_{j}^{k}\right] 
= \mathbb{E}_k\left[\sum_{\sigma\in A_{k}}
\mathbbm{1}_{\left\{ \left|\partial_{E}\sigma\right|=pk(n-k)+j\right\} }\right]
= {n \choose k}p^{pk(n-k)+j}(1-p)^{(1-p)k(n-k)-j}{k\left(n-k\right) \choose pk(n-k)+j}.
\end{equation}
Hence, by Markov's inequality, the claim in (\ref{eq:phi k i}) follows. Moreover,
\begin{equation}
\mathbb{E}_k\Bigg[\sum_{j\geq i}\phi_{j}^{k}\Bigg] 
= \mathbb{E}_k\left[\sum_{\sigma\in A_{k}}
\mathbbm{1}_{\left\{ \left|\partial_{E}\sigma\right|\geq pk(n-k)+i\right\} }\right]\\
\leq {n \choose k} e^{-2\tfrac{i^{2}}{k(n-k)}},
\end{equation}
where we again use Hoeffding's inequality. Hence, by Markov's inequality, we get the first 
line in (\ref{eq:phi k i 2}). The proof of the second line is similar.  
\end{proof}

%%%

\subsection{Jump rates for the volume process}
\label{ss:jumpvol}

The following lemma establishes bounds on the rate at which configurations in $A_{k}$ jump forward 
to $A_{k+1}$ and backward to $A_{k-1}$. In Appendix \ref{app} we will sharpen the error in the 
prefactors in \eqref{eq:frate}--\eqref{eq: brate} from $2n^{2/3}$ to $O(1)$ and the error in the exponents
in \eqref{eq:frate}--\eqref{eq: brate} from $3n^{-1/3}$ to $O(n^{-1/2})$. The formulas in \eqref{eq:frate at end}
and \eqref{eq:frate at end alt} show that for small and large magnetization the rate forward, respectively,
backward are maximal.  

\begin{lemma}[{\bf Bounds on forward jump rates}]
\label{lem: Rate bounds}
With $\mathbb{P}_{\ER_n(p)}$-probability tending to $1$ as $n\to\infty$ the following are true.\\ 
(a) For $2n^{1/3} \leq k \leq n-2n^{1/3}$,
\begin{equation}
\begin{aligned}
&\big(n-k-2n^{2/3}\big) e^{-2\beta [\vartheta_{k}+3n^{-1/3}]_{+}}\\
&\qquad \leq\sum_{\xi\in A_{k+1}}r\left(\sigma,\xi\right)\leq\big(n-k-2n^{2/3}\big)
e^{-2\beta [\vartheta_{k}-3n^{-1/3}]_{+}}+2n^{2/3}, \qquad \sigma\in A_{k},
\end{aligned}
\label{eq:frate}
\end{equation}
and 
\begin{equation}
\begin{aligned}
&\big(k-2n^{2/3}\big) e^{-2\beta[-\vartheta_{k}+3n^{-1/3}]_{+}}\\
&\qquad \leq\sum_{\xi\in A_{k-1}}r\left(\sigma,\xi\right)\leq\big(k-2n^{2/3}\big)
e^{-2\beta[-\vartheta_{k}-3n^{-1/3}]_{+}}+2n^{2/3}, \qquad \sigma\in A_{k},
\end{aligned}
\label{eq: brate}
\end{equation}
where
\begin{equation}
\label{eq:thetakdef}
\vartheta_{k}=p\left(1-\tfrac{2k}{n}\right)-h.
\end{equation}
(b) For $n-\tfrac{n}{3}(p+h)\leq k< n$,
\begin{equation}
\sum_{\xi\in A_{k+1}}r\left(\sigma,\xi\right)=n-k, \qquad \sigma\in A_{k}.
\label{eq:frate at end}
\end{equation}
(c) For $0 < k \leq \tfrac{n}{3}(p-h)$,
\begin{equation}
\sum_{\xi\in A_{k-1}}r\left(\sigma,\xi\right)=k, \qquad \sigma\in A_k.
\label{eq:frate at end alt}
\end{equation}
\end{lemma}

\begin{proof}
The proof is via probabilistic counting. 

\medskip\noindent
(a) Write $\mathbb{P}$ for the law under which $\sigma \in S_n$ is a uniformly random configuration 
and $v \in \overline{\sigma}$ is a uniformly random vertex. By Hoeffding's inequality, the probability 
that $v$ has more than $p\left|\sigma\right|+n^{2/3}$ neighbours in $\sigma$ (i.e., $w\in V$ such 
that $\left(v,w\right)\in E$ and $\sigma\left(w\right)=+1$) is bounded by
\begin{equation}
\mathbb{P}\left[\left|E(v,\sigma)\right|
\geq p\left|\overline{\sigma}\right|+n^{2/3}\right] \leq e^{-2n^{1/3}}.
\end{equation}
where
\begin{equation}
\label{eq:Evdef}
E(v,\sigma) = \left\{w\in\sigma\colon\,\left(v,w\right)\in E\right\}.
\end{equation}
Define the event 
\begin{equation}
\label{R+cond}
R^{+}\left(\sigma\right)=\left\{ \exists\,\zeta\subseteq\overline{\sigma},\,
\zeta\in A_{2n^{2/3}}\colon\,\left|E(v,\sigma)\right|\geq p\left|\sigma\right|+n^{2/3}\,\,\forall\, v\in\zeta,\right\},
\end{equation}
i.e., the configuration $\overline{\sigma}$ has at least $2n^{2/3}$ vertices like $v$, each with at least 
$p\left|\sigma\right|+n^{2/3}$ neighbours in $\sigma$. Then, for $0 \leq k\leq n-2n^{2/3}$,
\begin{equation}
\mathbb{P}\left[R^{+}\left(\sigma\right)\right]\leq{\left|\overline{\sigma}\right| 
\choose 2n^{2/3}}\left(e^{-2n^{1/3}}\right)^{2n^{2/3}}\leq2^{n} e^{-4n}.
\label{eq:prob many pop neighbours}
\end{equation}
Hence the probability that some configuration $\sigma \in S_n$ satisfies condition $R^+(\sigma)$ is 
bounded by 
\begin{equation}
\mathbb{P}\left[\bigcup_{\sigma\in S_{n}}R^{+}\left(\sigma\right)\right]\leq4^{n}
e^{-4n} \leq e^{-2n}.
\label{eq:probl R union}
\end{equation}
Thus, with $\mathbb{P}_{\ER_n(p)}$-probability tending to $1$ as $n\to\infty$ there are no 
configurations $\sigma \in S_{n}$ satisfying condition $R^{+}(\sigma)$. The 
same holds for the event 
\begin{equation}
\label{eq:R- def}
R^{-}\left(\sigma\right)=\left\{ \exists\,\zeta\subseteq\overline{\sigma},\,\zeta\in A_{2n^{2/3}}
\colon\,\left|E(v,\sigma)\right|
\leq p\left|\sigma\right|-n^{2/3}\,\,\forall\, v\in\zeta\right\},
\end{equation}
for which
\begin{equation}
\mathbb{P}\left[\,\bigcup_{\sigma\in S_{n}}R^{-}\left(\sigma\right)\right]
\leq e^{-2n}.
\label{eq:R- union}
\end{equation}
Now let $\sigma\in A_{k}$, and observe that $\sigma$ has $n-k$ neighbours in $A_{k+1}$ and 
$k$ neighbours in $A_{k-1}$. But if $\xi=\sigma\cup\left\{ v\right\} \in A_{k+1}$, then by 
\eqref{eq:Hn as boundary size},
\begin{eqnarray}
H_{n}\left(\xi\right)-H_{n}\left(\sigma\right) 
& = & \tfrac{2}{n}\Big(\left|E(v,\overline{\sigma})\right|-\left|E(v,\sigma)\right|\Big)-2h\\ \nonumber
& = & \tfrac{2}{n}\big(\deg\left(v\right)-2\left|E(v,\sigma)\right|\big)-2h\\ \nonumber
& \leq & \tfrac{2}{n}\big(pn+n^{1/2}\log n-2\left|E(v,\sigma)\right|\big)-2h,
\label{eq:H change up}
\end{eqnarray}
where the last inequality uses (\ref{eq:degree conc. ineq}) with $\varrho(n) = \log n$. Similarly,
\begin{eqnarray}
H_{n}\left(\xi\right)-H_{n}\left(\sigma\right) 
& \geq & \tfrac{2}{n}\left(pn-n^{1/2}\log n -2\left|E(v,\sigma)\right|\right)-2h.
\label{eq:H change down}
\end{eqnarray}
The events $R^{+}(\sigma)$ in \eqref{R+cond} and $R^{-}(\sigma)$ in \eqref{eq:R- def} guarantee 
that for any configuration $\sigma$ at most $2n^{2/3}$ vertices in the configuration $\overline{\sigma}$ 
can have more than $n^{2/3}$ neighbours in $\sigma$. In other words, the configuration $\sigma$ has 
at most $2n^{2/3}$ neighbouring configurations in $A_{k+1}$ that differ in energy by more than $
6n^{-1/3} - 2h$. Since on the complement of $R^+(\sigma)$ with $\sigma \in A_k$ we have $|\{w\in
\sigma\colon\,(v,w)\in E\}| \leq 2pk+2n^{1/3}$ (because $n^{1/2}\log n \leq n^{2/3}$ for $n$ large 
enough), from (\ref{eq:probl R union}) and (\ref{eq:R- union}) we get that, with $\mathbb{P}_{\ER_n
(p)}$-probability at least $1-e^{-2n}$, 
\begin{equation}
\begin{aligned}
&\left|\left\{ \xi\in A_{k+1}\colon\,\xi\sim\sigma,\,
H_{n}\left(\xi\right)-H_{n}\left(\sigma\right)\geq\tfrac{2}{n}
\left(pn-2pk+3n^{2/3}\right)-2h\right\} \right| \leq 2n^{2/3},\\
&\left|\left\{ \xi\in A_{k+1}\colon\,\xi\sim\sigma,\, 
H_{n}\left(\xi\right)-H_{n}\left(\sigma\right)\leq\tfrac{2}{n}
\left(pn-2pk-3n^{2/3}\right)-2h\right\} \right|\leq2n^{2/3},
\end{aligned}
\end{equation}
and hence, by \eqref{eq:rate r}, the rate at which the Markov chain starting at $\sigma \in A_k$ jumps to 
$A_{k+1}$ satisfies
\begin{eqnarray}
\label{eq:bdnrateup1}
&&\sum_{\xi\in A_{k+1}}r\left(\sigma,\xi\right)\geq\big(n-k-2n^{2/3}\big)
e^{-2\beta[\theta_k+3n^{-1/3}]_{+}},\\
&&\sum_{\xi\in A_{k+1}}r\left(\sigma,\xi\right)\leq\big(n-k-2n^{2/3}\big)
e^{-2\beta[\theta_k-3n^{-1/3}]_{+}}+2n^{2/3}.
\label{eq:bndrateup2}
\end{eqnarray}
Here the term $n-k-2n^{2/3}$ comes from exclusion of the at most $2n^{2/3}$ neighbours in configurations
that differ from $\sigma$ in energy by more than $6n^{-1/3}-2h$. Similarly, with $\mathbb{P}_{\ER_n
(p)}$-probability at least $1-e^{-2n}$,
\begin{equation}
\begin{aligned}
\left|\left\{ \xi\in A_{k-1}\colon\,\xi\sim\sigma, \, H_{n}\left(\xi\right)-H_{n}
\left(\sigma\right)\geq\tfrac{2}{n}\left(-pn+2pk+3n^{2/3}\right)+2h\right\} \right| \leq 2n^{2/3},\\
\left|\left\{ \xi\in A_{k-1}\colon\,\xi\sim\sigma, \, H_{n}\left(\xi\right)-H_{n}\left(\sigma\right)
\leq\tfrac{2}{n}\left(-pn+2pk-3n^{2/3}\right)+2h\right\} \right| \leq 2n^{2/3},
\end{aligned}
\end{equation}
and hence, by \eqref{eq:rate r}, the rate at which the Markov chain starting at $\sigma\in A_k$ 
jumps to $A_{k-1}$ satisfies
\begin{eqnarray}
\label{eq:bndratedown1}
&&\sum_{\xi\in A_{k-1}}r\left(\sigma,\xi\right)\leq\big(k-2n^{2/3}\big)
e^{-2\beta[-\theta_k-3n^{-1/3}]_{+}}+2n^{2/3},\\
&&\sum_{\xi\in A_{k-1}}r\left(\sigma,\xi\right)\geq\big(k-2n^{2/3}\big)
e^{-2\beta[-\theta_k+3n^{-1/3}]_{+}}.
\label{eq:bndratedown2}
\end{eqnarray}
This proves (\ref{eq:frate}) and (\ref{eq: brate}). 

\medskip\noindent
(b) To get (\ref{eq:frate at end}), note that for $\xi = \sigma \cup \{v\}$ with $v \notin \sigma$,
\begin{eqnarray}
H_{n}\left(\xi\right)-H_{n}\left(\sigma\right) 
& = & \tfrac{2}{n}\Big(\left|E(v,\overline{\sigma})\right|-\left|E(v,\sigma)\right|\Big)-2h\\ \nonumber
& = & \tfrac{2}{n}\Big(2\left|E(v,\sigma)\right|-\deg\left(v\right)\Big)-2h\\ \nonumber
& \leq & 2\left(2(n-k)-p+n^{-1/2}\log n -h\right)
\end{eqnarray}
for $n$ large enough, which is $\leq 0$ when $n-k \leq \tfrac{n}{3}(p+h)$, so that $r\left(\sigma,\xi\right)=1$ 
by \eqref{eq:rate r}. 

\medskip\noindent
(c) To get (\ref{eq:frate at end alt}), note that for $\xi = \sigma \setminus \{v\}$ with $v \notin \sigma$,
\begin{eqnarray}
H_{n}\left(\xi\right)-H_{n}\left(\sigma\right) 
& = & \tfrac{2}{n}\Big(\left|E(v,\sigma)\right|-\left|E(v,\overline{\sigma})\right|\Big)+2h\\ \nonumber
& = & \tfrac{2}{n}\Big(2\left|E(v,\sigma)\right|-\deg\left(v\right)\Big)+2h\\ \nonumber
& \leq & 2\left(2k-p+n^{-1/2}\log n+h\right)
\end{eqnarray}
for $n$ large enough, which is $\leq 0$ when $k \leq \tfrac{n}{3}(p-h)$, so that $r\left(\sigma,\xi\right)=1$ 
by \eqref{eq:rate r}. 
\end{proof}

The following lemma is technical and merely serves to show that near $A_{\mathbf{M}}$ transitions 
involving a flip from $-1$ to $+1$ typically occur at rate 1.

\begin{lemma}[{\bf Attraction towards the metastable state}]
\label{lem:H near m}
$\mbox{}$\\
Suppose that $\left|\xi\right|=[1+o_{n}(1)]\,\mathbf{M}$. Then $r\left(\xi,\xi^{v}\right)=1$
for all but $O(n^{2/3})$ many $v\in\xi$.
\end{lemma}

\begin{proof}
We want to show that 
\begin{equation}
H_{n}\left(\xi^{v}\right)<H_{n}\left(\xi\right)\label{eq:H xiv less than xi}
\end{equation}
for all but $O(n^{2/3})$ many $v\in\xi$. Note that by \eqref{eq:R- def} and \eqref{eq:R- union} there 
are at most $2n^{2/3}$ many $v\in\xi$ such that $|E(v,\overline{\xi})|\leq p(n-|\xi|)-n^{2/3}$, and at 
most $2n^{2/3}$ many $v\in\xi$ such that $|E(v,\xi)| \geq p|\xi|+n^{2/3}$. Hence, by \eqref{eq:Hn as 
boundary size}, for all but at most $4n^{2/3}$ many $v\in\xi$ we have that
\begin{eqnarray}
H_{n}\left(\xi^{v}\right) & = & H_{n}\left(\xi\right)
+\tfrac{2}{n}\left(\left|E\left(v,\xi\right)\right|-\left|E\left(v,\overline{\xi}\,\right)\right|\right)+2h\\ \nonumber
 & = & H_{n}\left(\xi\right)+\tfrac{2p}{n}\left(2\left|\xi\right|-n\right)+2h+o_{n}(1)\\ \nonumber
 & = & H_{n}\left(\xi\right)+\tfrac{2p}{n}\left(2\mathbf{M}-n\right)+2h+o_{n}(1)\\ \nonumber
 & = & H_{n}\left(\xi\right)+2p\mathbf{m}+2h+o_{n}(1),
\end{eqnarray}
where we use \eqref{eq:MTSdef}. From the definition of $\mathbf{m}$ in \eqref{eq: mts} it follows that 
$2p\mathbf{m}+2h+o_{n}(1)<0$, where we recall from the discussion near the end of Section 
\ref{ss:metaperturbedCW} that $\mathbf{m}<0$ and hence $\log(\frac{1-\mathbf{m}}{1+\mathbf{m}}) > 0$. 
Hence \eqref{eq:H xiv less than xi} follows.
\end{proof}

We can now prove the claim made in Remark \ref{rem:hbounds}, namely, there is no metastable 
behaviour outside the regime in \eqref{eq:metreg}. Recall the definition of $\mathbf{S}_n$ in \eqref{eq:MTSdef}, 
which requires the function $J$ in \eqref{Jadef} to have two zeroes. If it has only one zero, then 
denote that zero by $a'$ and define $\mathbf{S}_n=\frac{n}{2}(a'+1)$. Let $A_{\mathbf{S}_n+O(n^{2/3})}$ 
be the union of all $A_{k}$ with $|k-\mathbf{S}_n|=O(n^{2/3})$.

\begin{corollary}[{\bf Non-metastable regime}]
\label{cor:pleqh} 
Suppose that $\beta \in (1/p,\infty)$ and $h \in (p,\infty)$. Then $\{\xi_{t}\}_{t\geq0}$ has a drift 
towards $A_{\mathbf{S}_n+O(n^{2/3})}$. Consequently, $\mathbb{E}_{\xi_{0}}[\tau_{\mathbf{s}}]
=  O(1)$ for any initial configuration $\xi_{0}\in S_{n}$. 
\end{corollary}

\begin{proof}
If $\beta \in (1/p,\infty)$ and $h \in (p,\infty)$, then the function $a \mapsto J_{p,\beta,h}(a)=2\beta(pa+h)
+\log(\frac{1-a}{1+a})$ has a unique root in the interval $(0,1)$. Indeed, $J_{p,\beta,h}(a)>0$ for $a \in 
[-1,0]$, $J_{p,\beta,h}'(0) = 2(\beta p-1)>0$, while $a \mapsto \log(\frac{1-a}{1+a})$ is concave and tends 
to $-\infty$ as $a \uparrow1$. We claim that the process $\{\xi_t\}_{t \geq 0}$ drifts towards that root, i.e., 
if we denote the root by $a'$, then the process drifts towards the set $A_{\frac{n}{2}(a'+1)}$, which by 
convention we identify with $A_{S_n}$. Note that if $h \in (p,\infty)$, then $\vartheta_{k}=p(1-\frac{2k}{n})
-h<0$ for all $0\leq k\leq n$ and so, by Lemma~\ref{lem: Rate bounds},
\begin{eqnarray}
\sum_{\xi\in A_{k+1}}r\left(\sigma,\xi\right) & \geq & n-k-2n^{2/3},\\ \nonumber
\sum_{\xi\in A_{k-1}}r\left(\sigma,\xi\right) & \leq & \big(k-2n^{2/3}\big)
e^{-2\beta[-\vartheta_k-3n^{-1/3}]}+2n^{2/3}.
\end{eqnarray}
Thus, for $k\leq\frac{n}{2}-4n^{2/3}$, $\sum_{\xi\in A_{k+1}}r(\sigma,\xi) > \sum_{\xi\in A_{k-1}}
r(\sigma,\xi)$. Similarly, for $k \geq \frac{n}{2} + 4n^{2/3}$, the opposite inequality holds. Therefore
there is a drift towards $A_{\mathbf{S}_n+O(n^{2/3})}$. 
\end{proof}

We close this section with a lemma stating that the average return time to $A_{\mathbf{M}_n}$ 
conditional on not hitting $A_{\mathbf{T}_n}$ is of order $1$ and has an exponential tail. This will
be needed to control the time between successive attempts to go from $A_{\mathbf{M}_n}$ to
$A_{\mathbf{T}_n}$, until the dynamics crosses $A_{\mathbf{T}_n}$ and moves to $A_{\mathbf{S}_n}$
(recall Fig.~\ref{fig:CWfe}).    

\begin{lemma}[{\bf Conditional return time to the metastable set}]

\label{lem:returntail}
There exist a $C>0$ such that, with $\mathbb{P}_{\ER_n(p)}$-probability tending to $1$ as $n\to\infty$,
uniformly in $\xi \in A_{\mathbf{M}_n}$,
\begin{equation}
\label{eq:exptailER}
\mathbb{P}_\xi\left[\tau_{A_{\mathbf{M}_n}} \geq k \mid \tau_{A_{\mathbf{M}_n}} < \tau_{A_{\mathbf{T}_n}} \right]
\leq e^{-Ck}, \qquad \forall\, k.
\end{equation}
\end{lemma}

\begin{proof}
The proof is given in Appendix \ref{app}.
\end{proof}

%%%

\subsection{Uniformity in the starting configuration}
\label{ss:unifstart} 

The following lemma shows that the probability of the event $\{ \tau_{A_{k+o(n^{1/3})}}<\tau_{A_{k}}\}$
is almost uniform as a function of the starting configuration in $A_{k}$.

\begin{lemma}[{\bf Uniformity of hitting probability of volume level sets}]
\label{lem:ratio probab}
$\mbox{}$\\
With $\mathbb{P}_{\ER_n(p)}$-probability tending to $1$ as $n\to\infty$, the following is true. For every 
$0\leq k < m\leq n$,
\begin{equation}
\frac{\max_{\sigma\in A_{k}}\mathbb{P}_{\sigma}\left[\tau_{A_{m}}<\tau_{A_{k}}\right]}
{\min_{\sigma\in A_{k}}\mathbb{P}_{\sigma}\left[\tau_{A_{m}}<\tau_{A_{k}}\right]}
\leq\left[1+o_{n}(1)\right] e^{K(m-k)n^{-1/3}}
\label{eq:hitting times ratio}
\end{equation}
with $K=K(\beta,h,p) \in (0,\infty)$.
\end{lemma}

\begin{proof}
The proof proceeds by estimating the probability of trajectories from $A_k$ to $A_{m}$. 
Observe that 
\begin{eqnarray}
\label{eq:tkest}
e^{-2\beta[\vartheta_{k}+3n^{-1/3}]_{+}} 
& \geq & e^{-2\beta[\vartheta_{k}]_{+}}\left(1-6\beta n^{-1/3}\right),
\quad \forall\,n,\\ \nonumber
e^{-2\beta[\vartheta_{k}-3n^{-1/3}]_{+}} 
& \leq & e^{-2\beta[\vartheta_{k}]_{+}}\left(1+7\beta n^{-1/3}\right),
\quad n \text{ large enough},
\end{eqnarray}
and that similar estimates hold for $e^{-2\beta[-\vartheta_{k}+3n^{-1/3}]_{+}}$ and 
$e^{-2\beta[-\vartheta_{k}-3n^{-1/3}]_{+}}$. We will bound the ratio in the left-hand side 
of (\ref{eq:hitting times ratio}) by looking at two random processes on $\left\{ 0,\ldots,n\right\} $, 
one of which bounds $\max_{\sigma\in A_{k}}\mathbb{P}_{\sigma}\left[\tau_{A_{m}}<\tau_{A_{k}}\right]$ 
from above and the other of which bounds $\min_{\sigma\in A_{k}}\mathbb{P}_{\sigma}\left[\tau_{A_{m}}
<\tau_{A_{k}}\right]$ from below.  The proof comes in 3 Steps.

\medskip\noindent
{\bf 1.}
We begin with the following observation. Suppose that $\{X_{t}^{+}\}_{t \geq 0}$ and $\{X_{t}^{-}\}_{t \geq 0}$ 
are two continuous-time Markov chains taking unit steps in $\left\{0,\ldots,n\right\}$ at rates $r^{-}(k,k \pm 1)$ 
and $r^{+}(k,k \pm 1)$, respectively. Furthermore, suppose that for every $0\leq k\leq n-1$, 
\begin{equation}
\label{eq: bound on rate up}
r^{-}\left(k,k+1\right)\leq\min_{\sigma\in A_{k}}\sum_{\xi\in A_{k+1}}r\left(\sigma,\xi\right)
\leq\max_{\sigma\in A_{k}}\sum_{\xi\in A_{k+1}}r\left(\sigma,\xi\right)\leq r^{+}\left(k,k+1\right),
\end{equation}
and for every $1\leq k \leq n$,
\begin{equation}
r^{-}\left(k,k-1\right)\geq\max_{\sigma\in A_{k}}\sum_{\xi\in A_{k-1}}r\left(\sigma,\xi\right)
\geq\min_{\sigma\in A_{k}}\sum_{\xi\in A_{k-1}}r\left(\sigma,\xi\right)\geq r^{+}\left(k,k-1\right).
\label{eq: bound on rate down}
\end{equation}
Then
\begin{equation}
\frac{\max_{\sigma\in A_{k}}\mathbb{P}_{\sigma}\left[\tau_{A_{m}}<\tau_{A_{k}}\right]}
{\min_{\sigma\in A_{k}}\mathbb{P}_{\sigma}\left[\tau_{A_{m}}<\tau_{A_{k}}\right]}
\leq\frac{\mathbb{P}_{k}^{X^{+}}\left[\tau_{m}<\tau_{k}\right]}{\mathbb{P}_{k}^{X^{-}}
\left[\tau_{m}<\tau_{k}\right]}.
\label{eq: compare jump process}
\end{equation}
Indeed, from (\ref{eq: bound on rate up}) and (\ref{eq: bound on rate down}) it follows that 
we can couple the three Markov chains $\{X_{t}^{+}\}_{t\geq0}$, $\{X_{t}^{-}\}_{t\geq0}$ and 
$\{\xi_{t}\}_{t\geq 0}$ in such a way that, for any $0\leq k\leq n$ and any $\sigma_{0}\in A_{k}$, 
if $X_{0}^{-}=X_{0}^{+}=|\sigma_{0}|=k$, then 
\begin{equation}
X_{t}^{-}\leq\left|\sigma_{t}\right|\leq X_{t}^{+}, \qquad t \geq 0.
\end{equation}
This immediately guarantees that, for any $0 \leq k \leq m \leq n$,
\begin{equation}
\mathbb{P}_{k}^{X^{-}}\left[\tau_{m}<\tau_{k}\right] 
\leq\min_{\sigma\in A_{k}}\mathbb{P}_{\sigma}\left[\tau_{A_{m}}<\tau_{A_{k}}\right]
\leq\max_{\sigma\in A_{k}}\mathbb{P}_{\sigma}\left[\tau_{A_{m}}<\tau_{A_{k}}\right]
\leq\mathbb{P}_{k}^{X^{+}}\left[\tau_{m}<\tau_{k}\right],
\end{equation}
which proves the claim in (\ref{eq: compare jump process}). To get \eqref{eq: bound on rate up} 
and \eqref{eq: bound on rate down}, we pick  $r^{-}\left(i,j\right)$ and $r^{+}\left(i,j\right)$ such that
\begin{equation}
r^{-}\left(i,j\right)
=\begin{cases}
\left(n-i-\left(2+6\beta\right)n^{2/3}\right) e^{-2\beta[\vartheta_{i}]_{+}}
+\left(2+6\beta\right)n^{2/3}\mathbbm{1}_{\left\{ i\geq n\left(1-\frac{1}{3}\left(p+h\right)\right)\right\}}, 
&j=i+1,\\
\min\left\{ i,\,\left(i+\left(-2+7\beta\right)n^{2/3}\right) e^{-2\beta[-\vartheta_{i}]_{+}}
+2n^{2/3}\right\},  
&j=i-1,\\
0, &\mbox{otherwise},
\end{cases}
\end{equation}
and 
\begin{equation}
r^{+}\left(i,j\right)
=\begin{cases}
\min\left\{ n-i,\,\left(n-i+\left(-2+7\beta\right)n^{2/3}\right) e^{-2\beta[\vartheta_{i}]_{+}}
+2n^{2/3}\right\},  
&j=i+1,\\
\left(i-\left(2+6\beta\right)n^{2/3}\right) e^{-2\beta[-\vartheta_{i}]_{+}}, 
&j=i-1,\\
0, &\mbox{otherwise},
\end{cases}
\end{equation}
and note that, by Lemma~\ref{lem: Rate bounds} and \eqref{eq:tkest}, 
\eqref{eq: bound on rate up}--\eqref{eq: bound on rate down} are indeed satisfied. 

\medskip\noindent
{\bf 2.}
We continue from \eqref{eq: compare jump process}. Our task is to estimate the right-hand side
of \eqref{eq: compare jump process}. Let $\mathcal{G}$ be the set of all unit-step paths from $k$ 
to $m$ that only hit $m$ after their final step:
\begin{equation}
\mathcal{G}=\bigcup_{M\in\mathbb{N}}\left\{ \left\{ \gamma_{i}\right\} _{i=0}^{M-1}\colon\,
\gamma_{0}=k,\,\gamma_{M}=m,\,\gamma_{i}\in\left\{ 0,\ldots,m-1\right\} \text{ and } 
\left|\gamma_{i+1}-\gamma_{i}\right|=1\mbox{ for } 0 \leq i<M\right\}.
\end{equation}
We will show that 
\begin{equation}
\frac{\mathbb{P}_{k}^{X^{+}}\left[X_{t}^{+}\mbox{ follows trajectory }\gamma\right]}
{\mathbb{P}_{k}^{X^{-}}\left[X_{t}^{-}\mbox{ follows trajectory }\gamma\right]}
\leq\exp\left(\left[24\beta+4e^{2\beta(p+h+1)}\right]
\left(m-k\right)n^{-1/3}\right)\quad \forall\,\,\gamma\in\mathcal{G},
\end{equation}
which will settle the claim. (Note that the paths realising $\{\tau_{m}<\tau_{k}\}$ form a 
subset of $\mathcal{G}$.) To that end, let $\gamma^{\star}\in\mathcal{G}$ be the path 
$\gamma^{\star}=\left\{k,k+1,\ldots,m\right\}$. We claim that
\begin{equation}
\sup_{\gamma\in\mathcal{G}}\frac{\mathbb{P}_{k}^{X^{+}}\left[X_{t}^{+}
\mbox{ follows trajectory }\gamma\right]}
{\mathbb{P}_{k}^{X^{-}}\left[X_{t}^{-}\mbox{ follows trajectory }\gamma\right]}
\leq\frac{\mathbb{P}_{k}^{X^{+}}\left[X_{t}^{+}\mbox{ follows trajectory }\gamma^{\star}\right]}
{\mathbb{P}_{k}^{X^{-}}\left[X_{t}^{-}\mbox{ follows trajectory }\gamma^{\star}\right]}.
\label{eq:directpath}
\end{equation}
Indeed, if $\gamma=\left(\gamma_{1},\ldots,\gamma_{M}\right)\in\mathcal{G}$, then by the Markov 
property we have that 
\begin{equation}
\mathbb{P}_{k}^{X^{+}}\left[X_{t}^{+}\mbox{ follows trajectory }\gamma\right] 
= \prod_{i=0}^{M-1}\mathbb{P}_{\gamma_{i}}^{X^{+}}\left[\tau_{\gamma_{i+1}}<\tau_{\gamma_{i}}\right],
\end{equation}
with a similar expression for $\mathbb{P}_{k}^{X^{-}}[X_{t}^{-}\mbox{ follows trajectory }\gamma]$.
Therefore, noting that $\gamma_{i}-1=2\gamma_{i}-\gamma_{i+1}$ when $\gamma_{i+1}=\gamma_{i}+1$ 
and $\gamma_{i}+1=2\gamma_{i}-\gamma_{i+1}$ when $\gamma_{i+1}=\gamma_{i}-1$, we have 
\begin{eqnarray}
&&\frac{\mathbb{P}_{k}^{X^{+}}\left[X_{t}^{+}\mbox{ follows trajectory }\gamma\right]}
{\mathbb{P}_{k}^{X^{-}}\left[X_{t}^{-}\mbox{ follows trajectory }\gamma\right]} 
= \prod_{i=0}^{M-1}\frac{\mathbb{P}_{\gamma_{i}}^{X^{+}}\left[\tau_{\gamma_{i+1}}<\tau_{\gamma_{i}}\right]}
{\mathbb{P}_{\gamma_{i}}^{X^{-}}\left[\tau_{\gamma_{i+1}}<\tau_{\gamma_{i}}\right]}\\ \nonumber
&& \qquad =  \prod_{i=1}^{M}\left(\frac{r^{+}\left(\gamma_{i},\gamma_{i+1}\right)}{r^{+}
\left(\gamma_{i},\gamma_{i+1}\right)+r^{+}\left(\gamma_{i},2\gamma_{i}-\gamma_{i+1}\right)}\right)
\left(\frac{r^{-}\left(\gamma_{i},\gamma_{i+1}\right)}{r^{-}
\left(\gamma_{i},\gamma_{i+1}\right)+r^{-}\left(\gamma_{i},2\gamma_{i}-\gamma_{i+1}\right)}\right)^{-1}.
\end{eqnarray}
Since, whenever $\gamma_{i+1}=\gamma_{i}-1$, 
\begin{eqnarray}
\frac{r^{-}\left(\gamma_{i},\gamma_{i+1}\right)}{r^{-}\left(\gamma_{i},\gamma_{i+1}\right)
+r^{-}\left(\gamma_{i},2\gamma_{i}-\gamma_{i+1}\right)} 
& = & \frac{r^{-}\left(\gamma_{i},\gamma_{i}-1\right)}{r^{-}\left(\gamma_{i},\gamma_{i}-1\right)
+r^{-}\left(\gamma_{i},\gamma_{i}+1\right)}\\ \nonumber
& \geq & \frac{r^{+}\left(\gamma_{i},\gamma_{i}-1\right)}{r^{+}\left(\gamma_{i},\gamma_{i}-1\right)
+r^{+}\left(\gamma_{i},\gamma_{i}+1\right)}\\ \nonumber 
& = & \frac{r^{+}\left(\gamma_{i},\gamma_{i+1}\right)}{r^{+}\left(\gamma_{i},\gamma_{i+1}\right)
+r^{+}\left(\gamma_{i},2\gamma_{i}-\gamma_{i+1}\right)},
\end{eqnarray}
we get
\begin{eqnarray}
&&\prod_{i=0}^{M-1}\left(\frac{r^{+}\left(\gamma_{i},\gamma_{i+1}\right)}
{r^{+}\left(\gamma_{i},\gamma_{i+1}\right)+r^{+}\left(\gamma_{i},2\gamma_{i}-\gamma_{i+1}\right)}\right)
\left(\frac{r^{-}\left(\gamma_{i},\gamma_{i+1}\right)}{r^{-}\left(\gamma_{i},\gamma_{i+1}\right)
+r^{-}\left(\gamma_{i},2\gamma_{i}-\gamma_{i+1}\right)}\right)^{-1}\\ \nonumber
&&\qquad \leq \prod_{i=k}^{m-1}\left(\frac{r^{+}\left(i,i+1\right)}{r^{+}\left(i,i+1\right)+r^{+}\left(i,i-1\right)}\right)
\left(\frac{r^{-}\left(i,i+1\right)}{r^{-}\left(i,i+1\right)+r^{-}\left(i,i-1\right)}\right)^{-1}\\ \nonumber
&&\qquad = \frac{\mathbb{P}_{k}^{X^{+}}\left[X_{t}^{+}\mbox{ follows trajectory }\gamma^{\star}\right]}
{\mathbb{P}_{k}^{X^{-}}\left[X_{t}^{-}\mbox{ follows trajectory }\gamma^{\star}\right]}.
\end{eqnarray}
This proves the claim in (\ref{eq:directpath}). 

\medskip\noindent
{\bf 3.}
Next, consider the ratio 
\begin{equation}
\label{eq:rate ratio1}
\frac{r^{-}\left(i,i+1\right)+r^{-}\left(i,i-1\right)}{r^{+}\left(i,i+1\right)
+r^{+}\left(i,i-1\right)} = \frac{A}{B}
\end{equation}
with
\begin{equation}
\begin{aligned}
A &=
\left(n-i-\left(2+6\beta\right)n^{2/3}\right) e^{-2\beta[\vartheta_{i}]_{+}}
+\left(2+6\beta\right)n^{2/3}\mathbbm{1}_{\left\{ i\geq n\left(1-\frac{1}{3}\left(p+h\right)\right)\right\}}\\
&\qquad +\left(i+\left(-2+7\beta\right)n^{2/3}\right) e^{-2\beta[-\vartheta_{i}]_{+}}
+2n^{2/3},\\
B &= 
\left(n-i+\left(-2+7\beta\right)n^{2/3}\right) e^{-2\beta[\vartheta_{i}]_{+}}
+2n^{2/3}\\
&\qquad +\left(i-\left(2+6\beta\right)n^{2/3}\right) e^{-2\beta[-\vartheta_{i}]_{+}},
\end{aligned}
\end{equation}
and the ratio 
\begin{equation}
\label{eq:rate ratio2}
\frac{r^{+}\left(i,i+1\right)}{r^{-}\left(i,i+1\right)} = \frac{C}{D}
\end{equation} 
with
\begin{equation} 
\begin{aligned}
C &= 
\left(n-i+\left(-2+7\beta\right)n^{2/3}\right) e^{-2\beta[\vartheta_{i}]_{+}}
+2n^{2/3},\\
D &= \left(n-i-\left(2+6\beta\right)n^{2/3}\right) e^{-2\beta[\vartheta_{i}]_{+}}
+\left(2+6\beta\right)n^{2/3}\mathbbm{1}_{\left\{ i\geq n\left(1-\frac{1}{3}\left(p+h\right)\right)\right\}}.
\end{aligned}
\end{equation}
Note from (\ref{eq:rate ratio1}) that for $\vartheta_{i}\geq0$ (i.e., $i\leq\frac{n}{2}(1-p^{-1}h)$, in 
which case also $i<n(1-\frac{1}{3}\left(p+h\right)))$,
\begin{equation}
\frac{r^{-}\left(i,i+1\right)+r^{-}\left(i,i-1\right)}{r^{+}\left(i,i+1\right)+r^{+}\left(i,i-1\right)}
\leq1+\frac{13\beta e^{2\beta(p-h)}}{n^{1/3}},
\label{eq: r ratios bound 1}
\end{equation}
and from (\ref{eq:rate ratio2}) it follows that in this case 
\begin{eqnarray}
\frac{r^{+}\left(i,i+1\right)}{r^{-}\left(i,i+1\right)} 
& \leq & 1+\frac{3\left(3+13\beta\right) e^{2\beta(p-h)}}{n^{1/3}\left(p+h\right)}.
\label{eq: r ratios bound 2}
\end{eqnarray}
Similarly, for $\vartheta_{i}<0$ we have that 
\begin{eqnarray}
\frac{r^{-}\left(i,i+1\right)+r^{-}\left(i,i-1\right)}{r^{+}\left(i,i+1\right)+r^{+}\left(i,i-1\right)} 
& \leq & 1+\frac{2 e^{2\beta(p+h)}}{n^{1/3}}
\label{eq: r ratios bound 3}
\end{eqnarray}
and
\begin{eqnarray}
\frac{r^{+}\left(i,i+1\right)}{r^{-}\left(i,i+1\right)} 
& \leq & 1+\frac{6\left(2+6\beta\right)}{n^{1/3}\left(p+h\right)}.
\label{eq: r ratios bound 4}
\end{eqnarray}
Combining \eqref{eq: r ratios bound 1}--\eqref{eq: r ratios bound 4}, we get that, for all $1\leq i\leq n-1$, 
\begin{equation}
\frac{r^{-}\left(i,i+1\right)+r^{-}\left(i,i-1\right)}{r^{+}\left(i,i+1\right)
+r^{+}\left(i,i-1\right)}\times\frac{r^{+}\left(i,i+1\right)}{r^{-}\left(i,i+1\right)}\leq1+Kn^{-1/3},
\end{equation}
where 
\begin{equation}
\label{eq:Ciden}
K=\max\left\{ e^{2\beta(p-h)}\left(\tfrac{9+39\beta}{p+h}+13\beta\right),\,
2 e^{2\beta(p+h)}+\tfrac{12+36\beta}{p+h}\right\}.
\end{equation}
Therefore 
\begin{equation}
\frac{\mathbb{P}_{k}^{X^{+}}\left[X_{t}^{+}\mbox{ follows trajectory }\gamma^{\star}\right]}
{\mathbb{P}_{k}^{X^{-}}\left[X_{t}^{-}\mbox{ follows trajectory }\gamma^{\star}\right]} 
\leq \prod_{i=k}^{m-1}\left(1+\tfrac{K}{n^{1/3}}\right)
\leq e^{Kn^{-1/3}(m-k)}.
\end{equation}
\end{proof}

An application of the path-comparison methods used in Step 2 of the proof of Lemma \ref{lem:ratio probab} yields the following.

\begin{corollary}
\label{cor:expratio} 
With $\mathbb{P}_{\ER_n(p)}$-probability tending to $1$ as $n\to\infty$ the following is true. For every 
$0\leq k < m\leq n$,
\begin{equation}
\frac{\max_{\sigma\in A_{k}}\mathbb{E}_{\sigma}\left[\tau_{A_{m}}<\tau_{A_{k}}\right]}
{\min_{\sigma\in A_{k}}\mathbb{E}_{\sigma}\left[\tau_{A_{m}}<\tau_{A_{k}}\right]}
\leq\left[1+o_{n}(1)\right] e^{K(m-k)n^{-1/3}}
\label{eq:hitting times ratio alt}
\end{equation}
with $K=K(\beta,h,p) \in (0,\infty)$.
\end{corollary}

%%%%%%%%%% SECTION 4 %%%%%%%%%%%%%%%%%%%%%%%%%%

\section{Capacity bounds}
\label{s:capest}

The goal of this section is to derive various capacity bounds that will be needed to prove 
Theorem~\ref{thm:metER} in Sections~\ref{sec:couplingscheme}--\ref{s:proofmettheorem}.
In Section~\ref{ss:capestKn} we derive capacity bounds for the processes $\{\xi_{t}^{l}\}_{t\geq0}$ 
and $\{\xi_{t}^{u}\}_{t\geq0}$ on $K_n$ introduced in (\ref{eq: theta u l}) (Lemma~\ref{lem:cap bounds u l} 
below). In Section~\ref{ss:capestERnp} we do the same for the process $\{\xi_{t}\} _{t\geq0}$ 
on $\ER_n(p)$ (Lemma~\ref{lem: cap G bound} below). In Section~\ref{ss:rankorder} we use the
bounds to rank-order the mean return times to $A_{\mathbf{M^{l}}}$, $A_{\mathbf{M}}$ and 
$A_{\mathbf{M^{u}}}$, respectively (Lemma \ref{lem: Hitting order} below). This ordering will 
be needed in the construction of a coupling in Section~\ref{sec:couplingscheme}.

Define the Dirichlet form for $\{\xi_t\}_{t \geq 0}$ by
\begin{equation}
\mathscr{E}\left(f,f\right)=\tfrac{1}{2}\sum_{\sigma,\sigma'\in S_{n}}\mu\left(\sigma\right)
r\left(\sigma,\sigma'\right)\left[f\left(\sigma\right)-f\left(\sigma'\right)\right]^{2}, 
\qquad f\colon\,S_{n}\to\left[0,1\right],
\label{eq: Dirichlet form}
\end{equation}
and for $\{\theta_{t}^{u}\}_{t\geq0}$ and $\{\theta_{t}^{l}\}_{t\geq0}$ by
\begin{eqnarray}
\mathscr{E}^{u}\left(f,f\right) 
& = & \tfrac{1}{2}\sum_{a,a'\in \Gamma_n}\nu^{u}\left(a\right)q^{u}
\left(a,a'\right)\left[f\left(a\right)-f\left(a'\right)\right]^{2},\\ \nonumber
\mathscr{E}^{l}\left(f,f\right) & = & \tfrac{1}{2}
\sum_{a,a'\in \Gamma_n}\nu^{l}\left(a\right)q^{l}\left(a,a'\right)
\left[f\left(a\right)-f\left(a'\right)\right]^{2},
\qquad f\colon\, \Gamma_n\to\left[0,1\right]. 
\end{eqnarray}
For $A,B\subseteq S_{n}$, define the capacity between $A$ and $B$ for $\{\xi_t\}_{t \geq 0}$ by 
\begin{equation}
\mathrm{cap}\left(A,B\right)=\min_{f\in Q\left(A,B\right)}\mathscr{E}\left(f,f\right),
\label{eq:capAB}
\end{equation}
where
\begin{equation}
Q\left(A,B\right)=\left\{f\colon\, S_{n}\to\left[0,1\right],\, f_{|A}\equiv1,\, f_{|B}\equiv0\right\},
\label{eq:QAB}
\end{equation}
and similarly for $\mathrm{cap}^{u}\left(A,B\right)$ and $\mathrm{cap}^{l}\left(A,B\right)$.

%%%

\subsection{Capacity bounds on $K_n$}
\label{ss:capestKn}

First we derive capacity bounds for $\{\xi^l_{t}\} _{t\geq0}$ and $\{\xi^u_{t}\} _{t\geq0}$ on $K_n$. 
A useful reformulation of (\ref{eq:capAB}) is given by 
\begin{equation}
\mathrm{cap}\left(A,B\right)=\sum_{\sigma\in A}\sum_{\sigma'\in S_{n}}
\mu\left(\sigma\right)r\left(\sigma,\sigma'\right)\mathbb{P}_{\sigma}\left(\tau_{B}<\tau_{A}\right).
\label{eq:cap identity}
\end{equation}

\begin{lemma}[{\bf Capacity bounds for $\{\xi^u_{t}\} _{t\geq0}$ and $\{\xi^l_{t}\} _{t\geq0}$}]
\label{lem:cap bounds u l} 
For $a,b \in \left[\mathbf{m}^{u},\mathbf{s}^{u}\right]$ with $a < b $, 
\begin{equation}
\frac{\left(1-b + \frac{2}{n} \right)}{2n\left(b-a\right)^{2}}\leq\frac{\mathrm{cap}^{u}\left(a,b\right)}
{C^{\star}\left(b\right)}\leq\frac{n\left(1-b\right)}{2}
\label{eq: cap theta lemma}
\end{equation}
with 
\begin{equation}
C^{\star}\left(b\right)=\frac{1}{z^{u}} e^{-\beta n\Psi^{u}(b)}
{n \choose \frac{n}{2}\left(1+\min \left( b, \mathbf{t}^{u}\right)\right)}.
\end{equation}
For $a,b \in\left[\mathbf{m}^{l},\mathbf{s}^{l}\right]$ with $a<b$, analogous bounds hold for 
$\mathrm{cap}^{l}\left(a,b\right)$.
\end{lemma}

\begin{proof}
We will prove the upper and lower bounds only for $\mathrm{cap}^{u}\left(a,b\right)$, the proof 
for $\mathrm{cap}^{l}\left(a,b\right)$ being identical. Note from the definition in (\ref{eq:capAB}) that 
\begin{eqnarray}
&&\mathrm{cap}^{u}\left(a,b\right)\\ \nonumber 
&& = \min_{f\in Q(a,b)}\sum_{i=0}^{\frac{n}{2}\left(b-a\right)-1}\nu^{u}
\left(a+\tfrac{2i}{n}\right)q^{u}\left(a+\tfrac{2i}{n},a+\tfrac{2\left(i+1\right)}{n}\right)
\left[f\left(a+\tfrac{2i}{n}\right)-f\left(a+\tfrac{2\left(i+1\right)}{n}\right)\right]^{2},
\end{eqnarray}
where it is easy to see that the set $Q\left(a,b\right)$ in (\ref{eq:QAB}) may be reduced to 
\begin{equation}
\label{eq:capab}
Q\left(a,b\right)=\Big\{ f\colon \Gamma_n\to\left[0,1\right],\, f\left(x\right)=1
\mbox{ for } x\leq a,\, f\left(x\right)=0\mbox{ for }x\geq b\Big\}.
\end{equation}
Note that for every $f\in Q(a,b)$ there is some $0\leq i\leq\frac{n}{2}\left(b-a\right)-1$ such that
\begin{equation}
\left|f\left(a+\tfrac{2i}{n}\right)-f\left(a+\tfrac{2\left(i+1\right)}{n}\right)\right|
\geq\left(\tfrac{n}{2}\left(b-a\right)\right)^{-1}.
\label{eq:jump1}
\end{equation}
Also note that, by (\ref{eq:mu q Psi relation}), 
\begin{eqnarray}
&&\nu^{u}\left(a+\tfrac{2i}{n}\right)q^{u}\left(a+\tfrac{2i}{n},a+\tfrac{2\left(i+1\right)}{n}\right)\\ \nonumber
&& = \frac{1}{z^{u}}\frac{n}{2}\left(1-a-\tfrac{2i}{n}\right)
e^{-\beta n\max\big\{ \Psi^{u}\big(a+\tfrac{2i}{n}\big),
\Psi^{u}\big(a+\tfrac{2(i+1)}{n}\big)\big\}}
{n \choose \tfrac{n}{2}\left(1+a\right)+i},
\end{eqnarray}
and that, for any $\delta\in\mathbb{R}$, 
\begin{equation}
\Psi^{u}\left(a+\delta\right)-\Psi^{u}\left(a\right)=-\delta\left(pa+h^{u}+\tfrac{p}{2}\delta\right),
\label{eq:Psi difference}
\end{equation}
so that 
\begin{equation}
\label{eq:capab*}
\max\left\{ \Psi^{u}\left(a+\tfrac{2i}{n}\right),
\Psi^{u}\left(a+\tfrac{2\left(i+1\right)}{n}\right)\right\} \leq \tfrac{2}{n}\left(pa+h+\tfrac{p}{n}\right) 
+ \Psi^{u}\left(a+\tfrac{2i}{n}\right).
\end{equation}
Combining, \eqref{eq:capab}--\eqref{eq:capab*} with $\delta=\frac{2}{n}$, we get 
\begin{eqnarray}
&&\mathrm{cap}^{u}\left(a,b\right)\\ \nonumber
&&\qquad \geq \min_{0\leq i\leq\frac{n}{2}\left(b-a\right)-1}\frac{2\left(1-b + \tfrac{2}{n}\right)
e^{-2\beta(p+h^{u})}}{nz^{u}\left(b-a\right)^{2}}
e^{-\beta n\Psi^{u}\big(a+\tfrac{2i}{n}\big)}{n \choose \frac{n}{2}\left(1+a\right)+i}\\ \nonumber
&&\qquad = \frac{2\left(1-b + \frac{2}{n} \right) e^{-2\beta(p+h^{u} 
+ \frac{p}{n})}}{nz^{u}\left(b-a\right)^{2}}
e^{-\beta n\Psi^{u}\left(\min(b,\mathbf{t}^{u})\right)}{n \choose \frac{n}{2}
\left(1+\min(b,\mathbf{t}^{u})\right)},
\end{eqnarray}
where we use (\ref{eq:Psi difference}) that, by the definition of $\mathbf{m}^{u}$, $\mathbf{t}^{u}$, 
$\mathbf{s}^{u}$, for $a,b\in\left[\mathbf{m}^{u},\mathbf{s}^{u}\right]$ with $a<b$, the function $i \mapsto 
e^{-\beta n\Psi^{u}\left(a+\tfrac{2i}{n}\right)}{n \choose \frac{n}{2}(1+a)+i}$ is decreasing on $[\mathbf{m}^{u},
\mathbf{t}^{u}]$ and increasing on $[\mathbf{t}^{u},\mathbf{s}^{u}]$. This settles the lower bound 
in \eqref{eq: cap theta lemma}. 

Arguments similar to the ones above give 
\begin{eqnarray}
\nonumber
\mathrm{cap}^{u}\left(a,b\right) 
& \leq & \nu^{u}\left(\min(b,\mathbf{t}^u)-\tfrac{2}{n}\right)q^{u}
\left(\min(b,\mathbf{t}^u)-\tfrac{2}{n},\min(b,\mathbf{t}^u)\right)\\
& \leq & \frac{n\left(1-b\right) e^{2\beta\left(p+h^{u}\right)}}{2z^{u}}
e^{-\beta n\Psi^{u}\left(\min(b,\mathbf{t}^u)\right)}
{n \choose \frac{n}{2}\left(1+\min(b,\mathbf{t}^u)\right)},
\end{eqnarray}
where for the first equality we use the test function $f\equiv1$ on $\left[-1,\min(b,\mathbf{t}^u)-\frac{2}{n}\right]$
and $f\equiv0$ on $\left[\min(b,\mathbf{t}^u),1\right]$ in \eqref{eq:capab}. 
\end{proof}

%%%

\subsection{Capacity bounds on $\ER_n(p)$}
\label{ss:capestERnp} 

Next we derive capacity bounds for $\{\xi_{t}\} _{t\geq0}$ on $\ER_n(p)$. The proof is analogous to 
what was done in Lemma~\ref{lem:cap bounds u l} for $\{\theta_{t}^{u}\} _{t\geq0}$ and 
$\{\theta_{t}^{l}\}_{t\geq0}$ on $K_n$. 

Define the set of direct paths between $A \subseteq S_{n}$ and $ B\subseteq S_{n}$ by 
\begin{equation}
\mathcal{L}_{A,B}=\left\{ \gamma=(\gamma_0,\ldots,\gamma_{|\gamma|}\colon A\to B:\,\left|\gamma_{i+1}\right|
=\left|\gamma_{i}\right|+1\mbox{ for all }\gamma_{i}\in\gamma\right\},
\label{eq:direct paths}
\end{equation}
which may be empty. Abbreviate $\theta_k = p(1-\frac{k}{n})-h$. 

\begin{lemma}[{\bf Capacity bounds for $\{\xi_{t}\} _{t\geq0}$}]
\label{lem: cap G bound}
With $\mathbb{P}_{\ER_n(p)}$-probability tending to $1$ as $n\to\infty$ the following is true. 
For every $0\leq k<k'\leq n$ and every $\varrho\colon\mathbb{N}\to\mathbb{R}_{+}$ satisfying 
$\lim_{n\to\infty} \varrho\left(n\right)=\infty$,
\begin{eqnarray}
\mathrm{cap}\left(A_{k},A_{k'}\right)
& \leq & \frac{1}{Z}\,e^{-\beta H_{n}\left(\boxminus\right)}\,O\big(\varrho(n)n^{11/6}\big)
{n \choose k_{m}}\, e^{-\beta\, 2k_{m}\theta_{k_{m}}},
\label{eq:cap G ubound}\\ \nonumber
\mathrm{cap}\left(A_{k},A_{k'}\right) 
& \geq & \frac{1}{Z}\,e^{-\beta H_{n}\left(\boxminus\right)}\,\Omega\left(n^{-1} 
e^{-\left(\beta+\tfrac{1}{\sqrt{3}}\right)
\sqrt{\log n}}\right)
\label{eq:cap G lowbound}
{n \choose k_{m}}\, e^{-\beta\,2k_{m}\theta_{k_{m}}},
\end{eqnarray}
where
\begin{eqnarray}
k_{m} = \mathrm{argmin}_{\substack{k\leq j \leq k'}} {n \choose j}e^{-\beta\,2j\theta_{j}}.
\end{eqnarray}
\end{lemma}

\begin{proof}
Recall from \eqref{eq: Dirichlet form} and \eqref{eq:capAB} that 
\begin{equation}
\mathrm{cap}\left(A_{k},A_{k'}\right)  = \min_{f\in Q}\mathscr{E}\left(f,f\right)
= \min_{f\in Q}\,\,\tfrac{1}{2} \sum_{\sigma,\xi\in S_{n}}\mu\left(\sigma\right)r\left(\sigma,\xi\right)
\left[f\left(\sigma\right)-f\left(\xi\right)\right]^{2},
\end{equation}
where 
\begin{equation}
Q(A_k,A_{k'})=\left\{ f\colon\, S_{n}\to\left[0,1\right]\colon\, f|_{A_{k}}\equiv1,\, f|_{A_{k'}}\equiv0\right\}.
\end{equation}
The proof comes in 3 Steps.

\medskip\noindent
{\bf 1.}
We first prove the upper bound in \eqref{eq:cap G ubound}. Let $B=\bigcup_{j=k}^{k_{m}-1}A_{j}$, 
and note that, by \eqref{eq:mu}, 
\begin{eqnarray}
\mathrm{cap}\left(A_{k},A_{k'}\right) 
& \leq & \tfrac{1}{2}\sum_{\sigma,\xi\in S_{n}}\mu\left(\sigma\right)r\left(\sigma,\xi\right)
\left[\mathbbm{1}_{B}\left(\sigma\right)-\mathbbm{1}_{B}\left(\xi\right)\right]^{2}
= \sum_{\sigma\in A_{k_{m}-1}}\sum_{\xi\in A_{k_{m}}}\mu\left(\sigma\right)r\left(\sigma,\xi\right)\nonumber \\
& = & \frac{1}{Z}\sum_{\sigma\in A_{k_{m}-1}}\sum_{\xi\in A_{k_{m}},\xi\sim\sigma}
e^{-\beta\max\left\{ H_{n}\left(\xi\right),H_{n}\left(\sigma\right)\right\}}\nonumber \\
& = & \frac{1}{Z}\left(\sum_{\sigma\in A_{k_{m}-1}}\sum_{\substack{\xi\in A_{k_{m}},\xi\sim\sigma\\
H_{n}\left(\sigma\right)\geq H_{n}\left(\xi\right)}}
e^{-\beta H_{n}\left(\sigma\right)}+\sum_{\sigma\in A_{k_{m}-1}}
\sum_{\substack{\xi\in A_{k_{m}},\xi\sim\sigma\\
H_{n}\left(\sigma\right)<H_{n}\left(\xi\right)}}
e^{-\beta H_{n}\left(\xi\right)}\right)\nonumber \\
& \leq & \frac{1}{Z}\,\max\left\{ k_{m},n-k_{m}\right\}\,\left(\sum_{\sigma\in A_{k_{m}-1}}
e^{-\beta H_{n}\left(\sigma\right)}+\sum_{\xi\in A_{k_{m}}}
e^{-\beta H_{n}\left(\xi\right)}\right).
 \label{eq: cap ak ak' last line}
\end{eqnarray}
Recall from (\ref{eq:phi}) that $\phi_{i}^{k}$ denotes the cardinality of the set of all $\sigma\in A_{k}$ 
with $\left|\partial_{E}\sigma\right|=pk\left(n-k\right)+i$. Note from (\ref{eq:Hn as boundary size}) 
that for any $\xi\in A_{k_{m}}$ such that $\left|\partial_{E}\xi\right|=pk_{m}\left(n-k_{m}\right)+i$,
\begin{equation}
e^{-\beta H_{n}\left(\xi\right)} = e^{-\beta H_{n}\left(\boxminus\right)} e^{-\beta\,(2k_{m}\theta_{k_{m}}+\tfrac{2i}{n})}.
\end{equation}
There are ${n \choose k_{m}}$ terms in the sum, and therefore we get 
\begin{eqnarray}
\sum_{\xi\in A_{k_{m}}} e^{-\beta H_{n}\left(\xi\right)}
& = & e^{-\beta H_{n}\left(\boxminus\right)} 
\sum_{i=-pk_{m}\left(n-k_{m}\right)}^{\left(1-p\right)k_{m}\left(n-k_{m}\right)}\phi_{i}^{k_{m}}
e^{-\beta\left(2k_{m}\theta_{k_{m}}+\tfrac{2i}{n}\right)}
\nonumber \\
& = & e^{-\beta H_{n}\left(\boxminus\right)}\left(\sum_{i<-Y}\phi_{i}^{k_{m}}
e^{-\beta\left(2k_{m}\theta_{k_{m}}+\tfrac{2i}{n}\right)}
+ \sum_{i\geq-Y}e^{-\beta\left(2k_{m}\theta_{k_{m}}+\tfrac{2i}{n}\right)}\right)\\ \nonumber
& \leq & e^{-\beta H_{n}\left(\boxminus\right)}\left({n \choose k_{m}} 
e^{-\beta(2k_{m}\theta_{k_{m}}-\tfrac{2Y}{n})}+\sum_{i<-Y}\phi_{i}^{k_{m}}
e^{-\beta\left(2k_{m}\theta_{k_{m}}+\tfrac{2i}{n}\right)}\right) 
\label{eq:exp bounds 1}
\end{eqnarray}
with $Y=\sqrt{\log(\varrho(n)^{2}n^{5/6})k_{m}(n-k_{m})}$. The choice of $Y$ will become clear shortly. 
The summand in the right-hand side can be bounded as follows. By (\ref{eq:H conc ineq.}) in 
Lemma~\ref{lem: degree bound}, the sum over $i<-Y$ can be restricted to $-cn^{3/2}\leq i<-Y$,
since with high probability no configuration has a boundary size that deviates by more than 
$cn^{3/2}$ from the mean. But, using Lemma~\ref{lem: phi i k bounds}, we can also bound from 
above the number of configurations that deviate by at most $Y$ from the mean, i.e., we can bound 
$\phi_{i}^{k_{m}}$  for $-cn^{3/2}\leq i<-Y$. Taking a union bound over $0\leq k\leq n$ and 
$-cn^{3/2}\leq i<-Y$, we get
\begin{equation}
\mathbb{P}\left[\bigcup_{k=0}^{n}\,\,\bigcup_{i=-cn^{3/2}}^{-Y}\left\{ \phi_{i}^{k_{m}}
\geq\varrho\left(n\right)n^{5/2}{n \choose k_{m}}
e^{-\frac{2i^{2}}{k_{m}\left(n-k_{m}\right)}}\right\} \right]\leq\frac{1}{\varrho\left(n\right)}.
\label{eq:phi i k' bound}
\end{equation}
Thus, with $\mathbb{P}_{\ER_n(p)}$-probability at least $\geq1-\frac{1}{\varrho\left(n\right)}$,
\begin{eqnarray}
\sum_{i<-Y}\phi_{i}^{k_{m}} e^{-\beta\left(2k_{m}\,\theta_{k_{m}}+\tfrac{2i}{n}\right)}
&&\leq \sum_{i>Y}\varrho\left(n\right)n^{5/2}{n \choose k_{m}}
e^{-2i\left(\tfrac{i}{k_{m}\left(n-k_{m}\right)}-\tfrac{\beta}{n}\right)}
e^{-\beta\,2k_{m}\,\theta_{k_{m}}}\\ \nonumber
&&\leq \varrho\left(n\right)n^{5/2}{n \choose k_{m}}
e^{-\beta\,2k_{m}\,\theta_{k_{m}}}
e^{-2\log(\varrho(n)n^{5/6})}\\ \nonumber
&&\leq\frac{n^{5/6}}{\varrho\left(n\right)}{n \choose k_{m}}
e^{-\beta\,2k_{m}\,\theta_{k_{m}}},
\end{eqnarray}
where we use that, for $i>Y$ and $n$ sufficiently large, 
\begin{equation}
\tfrac{i}{k_{m}\left(n-k_{m}\right)}-\tfrac{\beta}{n}\geq
\sqrt{\tfrac{\log(\varrho\left(n\right)^{2}n^{5/6})}{k_{m}\left(n-k_{m}\right)}}-\tfrac{\beta}{n}
\geq\sqrt{\tfrac{\log(\varrho\left(n\right)n^{5/6})}{k_{m}\left(n-k_{m}\right)}}.
\end{equation}
The above inequality also clarifies our choice of $Y$. Substituting this into (\ref{eq:exp bounds 1}), 
we see that
\begin{eqnarray}
\sum_{\xi\in A_{k_{m}}} e^{-\beta H_{n}\left(\xi\right)}
&\leq& [1+o_{n}\left(1\right)]\, e^{-\beta H_{n}\left(\boxminus\right)} e^{\tfrac{2\beta Y}{n}}
e^{-\beta\,2k_{m}\,\theta_{k_{m}}}{n \choose k_{m}}\\ \nonumber
&=& O\big(\varrho(n)n^{5/6}\big)\,
e^{-\beta H_{n}\left(\boxminus\right)} e^{-\beta\,2k_{m}\theta_{k_{m}}}{n \choose k_{m}}.
\label{eq: exp sum bnd1}
\end{eqnarray}
A similar bound holds for $\sum_{\xi\in A_{k_{m}-1}} e^{-\beta H_{n}\left(\xi\right)}$. A union bound 
over $1\leq k_{m} \leq n$ increases the exponent $\frac{5}{6}$ to $\frac{11}{6}$. Together with 
(\ref{eq: cap ak ak' last line}), this proves the upper bound in \eqref{eq:cap G ubound}. 

\medskip\noindent
{\bf 2.}
We next derive a combinatorial bound that will be used later for the proof of the lower bound 
in \eqref{eq:cap G ubound}. Note that if $f\in Q(A_k,A_{k'})$ and $\gamma\in\mathcal{L}_{A_{k},A_{k'}}$
(recall (\ref{eq:direct paths})), then there must be some $1\leq i\leq k'-k$ such that 
\begin{equation}
\left|f\left(\gamma_{i}\right)-f\left(\gamma_{i+1}\right)\right|\geq\left(k'-k\right)^{-1}.
\label{eq:must jump}
\end{equation}
A simple counting argument shows that 
\begin{equation}
\left|\mathcal{L}_{A_{k},A_{k'}}\right|={n \choose k}\,\frac{\left(n-k\right)!}{\left(n-k'\right)!},
\end{equation}
since for each $\sigma\in A_{k}$ there are $\left(n-k\right)\times\left(n-k-1\right)\times\cdots
\times\left(n-k'+1\right)$ paths in $\mathcal{L}_{A_{k},A_{k'}}$ from $\sigma$ to $A_{k'}$. Let
\begin{equation}
b_{i}=\left|\left\{ \left(\sigma,\xi\right)\in A_{k+i-1}\times A_{k+i}\colon\,
\left|f\left(\sigma\right)-f\left(\xi\right)\right|
\geq\left(k'-k\right)^{-1},\,\sigma\sim\xi\right\} \right|, \quad 1\leq i\leq k'-k.
\end{equation}
We claim that 
\begin{equation}
\label{eq:lozenge}
\exists\,1\leq i_{\blacklozenge}\leq k'-k\colon \qquad  
b_{i_{\blacklozenge}}\geq\frac{k}{k'-k}{n \choose k+i_{\blacklozenge}}. 
\end{equation}
Indeed, the number of paths in $\mathcal{L}_{A_{k},A_{k'}}$ that pass through $\sigma\in A_{k+i_{\blacklozenge}-1}$ 
followed by a move to $\xi\in A_{k+i_{\blacklozenge}}$ equals
\begin{equation}
z_{i_{\blacklozenge}}=\frac{\left(k+i_{\blacklozenge}-1\right)!}{k!}\times
\frac{\left(n-k-i_{\blacklozenge}\right)!}{\left(n-k'\right)!},
\label{eq:pathcount}
\end{equation}
where the first term in the product counts the number of paths from $\sigma \in A_{k+i_{\blacklozenge}+1}$ 
to $A_{k}$, while the second term counts the number of paths from $\xi \in A_{k+i_{\blacklozenge}}$ to $A_{k'}$. 
Thus, if \eqref{eq:lozenge} fails, then 
\begin{equation}
\sum_{i=1}^{k'-k}b_{i}z_{i} <  \frac{k}{k'-k}\sum_{i=1}^{k-k'}\frac{1}{k+i}\frac{n!}{k!\left(n-k'\right)!}
\leq \frac{n!}{k!\left(n-k'\right)!}=\left|\mathcal{L}_{A_{k},A_{k'}}\right|,
\end{equation}
which in turn implies that (\ref{eq:must jump}) does not hold for some $\gamma\in\mathcal{L}_{A_{k},A_{k'}}$
(use that $b_{i}z_{i}$ counts the paths that satisfy condition (\ref{eq:must jump})), which is 
a contradiction. Hence the claim in \eqref{eq:lozenge} holds.

\medskip\noindent
{\bf 3.}
In this part we prove the lower bound in \eqref{eq:cap G ubound}. By Lemma~\ref{lem: phi i k bounds} we 
have that, with $\mathbb{P}_{\ER_n(p)}$-probability at least $1-\frac{1}{\varrho(n)}$, for any $Y\geq0$, 
\begin{equation}
\sum_{j\geq\sqrt{Y\left(k+i_{\blacklozenge}\right)\left(n-k-i_{\blacklozenge}\right)}}
\phi_{j}^{k+i_{\blacklozenge}}\leq\varrho\left(n\right){n \choose k+i_{\blacklozenge}}
e^{-2Y}.
\end{equation}
Picking $Y=\log(\varrho\left(n\right)k^{-1}2n^{3/2})$, we get that 
\begin{equation}
\sum_{j\geq\sqrt{Y\left(k+i_{\blacklozenge}\right)\left(n-k-i_{\blacklozenge}\right)}}
\phi_{j}^{k+i_{\blacklozenge}} \leq\frac{1}{4}\frac{k^{2}\varrho(n)}{\varrho(n)^{2}n^{3}}
{n \choose k+i_{\blacklozenge}} 
\leq\frac{1}{2}\frac{k}{n\left(k'-k\right)}{n \choose k+i_{\blacklozenge}},
\end{equation}
and so at least half of the configurations contributing to $b_{i_{\blacklozenge}}$ have an 
edge-boundary of size at most
\begin{equation}
p\left(k+i_{\blacklozenge}\right)\left(n-k-i_{\blacklozenge}\right)
+\sqrt{Y\left(k+i_{\blacklozenge}\right)\left(n-k-i_{\blacklozenge}\right)}.
\end{equation} 
If $\xi\in A_{k+i_{\blacklozenge}}$ is such a configuration, then by Lemma~\ref{lem: degree bound} 
the same is true for any $\sigma\sim\xi$ (i.e., configurations differing at only one vertex), since 
\begin{equation}
\left|\partial_{E}\sigma\right|\leq\left|\partial_{E}\xi\right|+ \max_{v\in\sigma\triangle\xi} \deg\left(v\right)
\leq\left|\partial_{E}\xi\right|+pn+o\left(\rho(n)\sqrt{n\log n}\,\right).
\end{equation}
This implies 
\begin{eqnarray}
\label{eq:Dirichlet ineq 1}
\mathscr{E}\left(f,f\right) 
& = & \tfrac{1}{2}\sum_{\sigma,\xi\in S_{n}}\mu\left(\sigma\right)r\left(\sigma,\xi\right)
\left[f\left(\sigma\right)-f\left(\xi\right)\right]^{2}\nonumber \\
&&\geq \frac{1}{Zn^{2/3}}\sum_{\xi\in A_{k+i_{\blacklozenge}}}\sum_{\sigma\in A_{k+i_{\blacklozenge}-1}}
e^{-\beta\max\left\{ H_{n}\left(\sigma\right),H_{n}\left(\xi\right)\right\}}\nonumber \\
&& \geq e^{-\beta H_{n}\left(\boxminus\right)} \frac{k}{2Zn^{2}} {n \choose k+i_{\blacklozenge}}
\exp\left(-\beta\left(2(k+i_{\blacklozenge})\theta_{k+i_{\blacklozenge}}
+2\tfrac{\sqrt{Y\left(k+i_{\blacklozenge}\right)\left(n-k-i_{\blacklozenge}\right)}}{n}\right)\right).
\end{eqnarray}
Therefore
\begin{eqnarray}
\label{eq:Dirichlet ineq 2}
\mathscr{E}\left(f,f\right) 
& \geq & e^{-\beta H_{n}\left(\boxminus\right)} e^{-\beta\sqrt{Y}}
\min_{1\leq i\leq k'-k}\frac{k}{2Zn^{2}} {n \choose k+i}
e^{-\beta\left(2\left(k+i\right)\theta_{k+i}\right)}\\ \nonumber
& = & e^{-\beta H_{n}\left(\boxminus\right)} e^{-\beta\sqrt{Y}}
\frac{k}{2Zn^{2}} {n \choose k_{m}}
e^{-\beta\left(2\left(k+i\right)\theta_{k-{m}}\right)}.
\end{eqnarray}
Since (\ref{eq:Dirichlet ineq 2}) is true for any $f\in Q(A_k,A_{k'})$, the lower bound in 
\eqref{eq:cap G ubound} follows, with $k_m$ defined in (\ref{eq:cap G lowbound}).
\end{proof}

%%%
\subsection{Hitting probabilities on $\ER_n(p)$}
\label{ss:rankorder}

Let $\mu_{A_{\mathbf{M}}}$ be the equilibrium distribution $\mu$ conditioned to the set $A_{\mathbf{M}}$.
Write $\mathbb{P}^{l}$ and $\mathbb{P}^{u}$ denote the laws of the processes $\{\xi_{t}^{l}\} _{t\geq0}$ 
and $\{ \xi_{t}^{u}\} _{t\geq0}$, respectively. The following lemma is the crucial sandwich
for comparing the crossover times of the dynamics on $\ER_n(p)$ and the perturbed dynamics on $K_n$.   

\begin{lemma}[{\bf Rank ordering of hitting probabilities}]
\label{lem: Hitting order} 
With $\mathbb{P}_{\ER_n(p)}$-probability tending to $1$ as $n\to\infty$, 
\begin{equation}
\begin{aligned}\max_{\xi\in A_{\mathbf{M}^l}}
\mathbb{P}_{\xi}^{l}\left[\tau_{\mathbf{S}^l}<\tau_{\mathbf{M}^l}\right] 
\leq\mathbb{P}_{\mu_{A_{\mathbf{M}}}}\left[\tau_{\mathbf{S}}<\tau_{\mathbf{M}}\right] 
\leq \min_{\sigma\in A_{\mathbf{M}^u}}\mathbb{P}^{u}_{\sigma}\left[\tau_{\mathbf{S}^u}
<\tau_{\mathbf{M}^u}\right]. 
\end{aligned}
\label{eq:lemma ordered times}
\end{equation}
\end{lemma}

\begin{proof}

The proof comes in 3 Steps.

\medskip\noindent
{\bf 1.}
Recall from (\ref{eq:magnetization}) that the magnetization of $\sigma\in A_{k}$ is 
$m\left(\sigma\right)=2\frac{k}{n}-1$. We first observe that the maximum and the 
minimum in (\ref{eq:lemma ordered times}) are redundant, because by symmetry 
\begin{equation}
\begin{aligned}
\max_{\xi\in A_{k}}\mathbb{P}_{\xi}^{l}\left[\tau_{A_{k'}}<\tau_{A_{k}}\right]
&= \min_{\xi\in A_{k}}\mathbb{P}_{\xi}^{l}\left[\tau_{A_{k'}}<\tau_{A_{k}}\right],\\
\min_{\xi\in A_{k}}\mathbb{P}_{\xi}^{u}\left[\tau_{A_{k'}}<\tau_{A_{k}}\right]
&= \max_{\xi\in A_{k}}\mathbb{P}_{\xi}^{u}\left[\tau_{A_{k'}}<\tau_{A_{k}}\right].
\end{aligned}
\end{equation}
Recall that $\{\xi_{t}^{l}\}_{t\geq0}$ is the Markov process on $S_{n}$ governed by the
Hamiltonian $H_{n}^{l}$ in (\ref{eq:Hl}), and that the associated magnetization process 
$\{\theta_{t}^{l}\} _{t\geq0}=\{m(\xi_{t}^{l})\} _{t\geq0}$ is a Markov process on the set 
$\Gamma_n$ in \eqref{eq:InJndef} with transition rates given by $q^{l}$ in (\ref{eq:ql}). Denoting 
by $\hat{\mathbb{P}}^{l}$ the law of $\{\theta_{t}^{l}\} _{t\geq0}$, we get from (\ref{eq:cap identity}) 
that for any $0\leq k \leq k' < n$, and with $a=\frac{2k}{n}-1$ and $b=\frac{2k'}{n}-1$,
\begin{eqnarray}
cap^{l}\left(a,b\right) &=& \sum_{u \in \Gamma_n} \nu ^{l}\left(a\right) q^{l}\left(a,u\right)
\hat{\mathbb{P}}_{a}^{l}\left[\tau_{b}<\tau_{a}\right]\\ \nonumber
&=&\nu ^{l}\left(a\right) \left[ q^{l}\left(a,a+\tfrac{2}{n}\right) 
+ q^{l}\left(a,a-\tfrac{2}{n}\right) \right] 
\hat{\mathbb{P}}_{a}^{l}\left[\tau_{b}<\tau_{a}\right], 
\end{eqnarray}
and therefore
\begin{equation}
\max_{\xi\in A_{k}}\mathbb{P}_{\xi}^{l}\left[\tau_{A_{k'}}<\tau_{A_{k}}\right] 
= \hat{\mathbb{P}}_{a}^{l}\left[\tau_{b}<\tau_{a}\right]
= \left[\nu^{l}\left(a\right)\left(q^{l}\left(a,a+\tfrac{2}{n}\right)
+q^{l}\left(a,a-\tfrac{2}{n}\right)\right)\right]^{-1}\mathrm{cap}^{l}\left(a,b\right).
\label{eq: Pl1}
\end{equation}
By (\ref{eq:mu q Psi relation}), using the abbreviations
\begin{equation}
\Psi_{1}=\max\left\{ \Psi^{l}\left(a\right),\Psi^{l}\left(a+\tfrac{2}{n}\right)\right\}, \qquad 
\Psi_{2}=\max\left\{ \Psi^{l}\left(a\right),\Psi^{l}\left(a-\tfrac{2}{n}\right)\right\},
\end{equation}
we have, with the help of (\ref{eq:Psi difference}),
\begin{eqnarray}
\nu^{l}\left(a\right)\left(q^{l}\left(a,a+\tfrac{2}{n}\right)+q^{l}\left(a,a-\tfrac{2}{n}\right)\right)
&=& \frac{1}{z^{l}}\frac{n}{2}{n \choose \frac{n}{2}\left(1+a\right)}
\Big(\left(1-a\right) e^{-\beta n\Psi_{1}}+\left(1+a\right)e^{-\beta n\Psi_{2}}\Big)
\label{eq: Pl2}\\
&\geq& \frac{1}{z^{l}}n e^{-2\beta\left(p\left|a\right|+h^{l}+\tfrac{p}{n}\right)}
e^{-\beta n\Psi^{l}\left(a\right)} {n \choose \frac{n}{2}\left(1+a\right)}. \nonumber
\end{eqnarray}
From Lemma~\ref{lem:cap bounds u l} we have that 
\begin{equation}
\mathrm{cap}^{l}\left(a,b\right)\leq\tfrac{n\left(1-a\right)}{2z^{l}}
e^{-\beta n\Psi^{l}\left(b\right)}{n \choose \frac{n}{2}\left(1+b\right)}.
\label{eq:Pl3}
\end{equation}
Putting (\ref{eq: Pl1}), (\ref{eq: Pl2}) and (\ref{eq:Pl3}) together, we get 
\begin{eqnarray}
\label{eq:AkP}
\max_{\xi\in A_{k}}\mathbb{P}_{\xi}^{l}\left[\tau_{A_{k'}}<\tau_{A_{k}}\right] 
\leq \tfrac{\left(1-a\right)}{2} e^{2\beta\left(p\left|a\right|+h^{l}
+\tfrac{p}{n}\right)}
e^{-\beta n\left[\Psi^{l}\left(b\right)-\Psi^{l}\left(a\right)\right]}
{n \choose \frac{n}{2}\left(1+b\right)}{n \choose \frac{n}{2}\left(1+a\right)}^{-1}.
\end{eqnarray}
Similarly, denoting by $\hat{\mathbb{P}}^{u}$ the law of $\{\theta_{t}^{u}\} _{t\geq0}$, we have
\begin{eqnarray}
\label{eq:ust1}
\min_{\xi\in A_{k}}\mathbb{P}_{\xi}^{u}\left[\tau_{A_{k'}}<\tau_{A_{k}}\right] 
& = & \hat{\mathbb{P}}_{a}^{u}\left[\tau_{b}<\tau_{a}\right]\\ \nonumber
 & = & \left[\nu^{u}\left(a\right)\left(q^{u}\left(a,a+\tfrac{2}{n}\right)
 +q^{u}\left(a,a-\tfrac{2}{n}\right)\right)\right]^{-1}\mathrm{cap}^{u}\left(a,b\right),
\end{eqnarray}
where
\begin{equation}
\begin{aligned}
\left[\nu^{u}\left(a\right)\left(q^{u}\left(a,a+\tfrac{2}{n}\right)
+q^{u}\left(a,a-\tfrac{2}{n}\right)\right)\right]^{-1}
\geq\left[\frac{n}{z^{u}} e^{2\beta\left(p\left|a\right|+h^{l}+\tfrac{p}{n}\right)}
e^{-\beta n\Psi^{u}\left(a\right)}{n \choose \frac{n}{2}\left(1+a\right)}\right]^{-1},
\end{aligned}
\end{equation}
and, Lemma \ref{lem:cap bounds u l},
\begin{equation}
\label{eq:ust2}
\mathrm{cap}^{u}\left(a,b\right) \geq  \frac{1}{2nz^{u}}
e^{-\beta n\Psi^{u}\left(b\right)}{n \choose \frac{n}{2}\left(1+b\right)}.
\end{equation}
Putting \eqref{eq:ust1}--\eqref{eq:ust2} together, we get
\begin{equation}
\min_{\xi\in A_{k}}\mathbb{P}_{\xi}^{u}\left[\tau_{A_{k'}}<\tau_{A_{k}}\right] 
\geq \frac{1}{n} e^{-\beta n\left[\Psi^{l}\left(b\right)-\Psi^{l}
\left(a\right)\right]}{n \choose \frac{n}{2}\left(1+b\right)}
{n \choose \frac{n}{2}\left(1+a\right)}^{-1}.
\label{eq: partial bound 1}
\end{equation}

\medskip\noindent
{\bf 2.}
Recall from (\ref{eq:cap identity}) that 
\begin{equation}
\label{eq:capdef1}
\mathrm{cap}\left(A_{k},A_{k'}\right)
=\sum_{\sigma\in A_{k}}\sum_{\xi \in S_{n}}\mu(\sigma)r(\sigma,\xi)
\mathbb{P}_{\sigma}\left[\tau_{A_{k'}}<\tau_{A_{k}}\right].
\end{equation}
Split
\begin{eqnarray}
&&\sum_{\sigma \in A_{k}}\sum_{\xi \in S_{n}} \mu(\sigma)r(\sigma,\xi) 
=  \sum_{\sigma \in A_{k}} \sum_{\xi \in A_{k+1}} \mu(\sigma)r(\sigma,\xi)
+ \sum_{\sigma \in A_{k}} \sum_{\xi \in A_{k-1}} \mu(\sigma)r(\sigma,\xi)\\ \nonumber
&& = \frac{1}{Z} \sum_{\sigma \in A_{k}} \sum_{\xi \in A_{k+1}}
e^{-\beta\max\left\{ H_{n}\left(\sigma\right),H_{n}\left(\xi\right)\right\}}
+\frac{1}{Z}\sum_{\xi\in A_{k}}\sum_{\xi'A_{k-1}}
e^{-\beta\max\left\{ H_{n}\left(\sigma\right),H_{n}\left(\xi\right)\right\}}.
\end{eqnarray}
By Lemma~\ref{lem: phi i k bounds} and a reasoning similar to that leading to 
(\ref{eq:Dirichlet ineq 1}), 
\begin{eqnarray}
\sum_{\xi\in A_{k}} e^{-\beta H_{n}\left(\xi\right)} 
&&=  e^{-\beta H_{n}\left(\boxminus\right)} 
\sum_{i=-pk\left(n-k\right)}^{\left(1-p\right)k\left(n-k\right)}\phi_{i}^{k}
e^{-\beta\left(2k\theta_k+2\tfrac{i}{n}\right)}
\label{eq:exp bounds 2}\\
&&\geq \tfrac{1}{2}{n \choose k} e^{-\beta H_{n}\left(\boxminus\right)} 
e^{-\beta\left(2k\theta_k+2\tfrac{\sqrt{Yk(n-k)}}{n}\right)}\nonumber \\
&&\geq \tfrac{1}{2}{n \choose k} e^{-\beta H_{n}(\boxminus)} 
e^{-\beta\sqrt{\log(\sqrt{2\varrho(n)})}}
e^{-\beta\,2k\theta_k}\nonumber 
\end{eqnarray}
with $Y=\log(\sqrt{2\varrho(n)})$. Indeed, by (\ref{eq:phi k i 2}) fewer than $\frac{1}{2}{n \choose k}$ 
configurations in $A_{k}$ have an edge-boundary of size $\geq pk\left(n-k\right)+\sqrt{k\left(n-k\right)Y}$.
Moreover, if $\xi\sim\xi'$, then, by Lemma \ref{lem: degree bound},
\begin{equation}
e^{\beta[ H_{n}(\xi')-H_{n}(\xi)]}
\leq [1+o(1)]\,e^{\beta\left(p+h\right)},
\end{equation}
and since we may absorb this constant inside the error term $\varrho\left(n\right)$, we get that 
\begin{eqnarray}
\label{eq: Gnp denom lbound}
&&\sum_{\sigma \in A_{k}} \sum_{\xi \in A_{k+1}}
e^{-\beta\max\left\{H_{n}(\sigma),H_{n}(\xi)\right\}}
\geq e^{-\beta H_{n}(\boxminus)}\tfrac{1}{2}\left(n-k\right){n \choose k}
e^{-\tfrac{\beta}{2}\sqrt{\log\sqrt{2\varrho(n)}}}
e^{-\beta\,2k\,\theta_k},\\
\label{eq:Gnp denom lbound2}
&& \sum_{\sigma \in A_{k}} \sum_{\xi \in A_{k-1}}
e^{-\beta\max\left\{H_{n}(\sigma),H_{n}(\xi)\right\}}
\geq e^{-\beta H_{n}(\boxminus)}\tfrac{1}{2} k {n \choose k}
e^{-\tfrac{\beta}{2}\sqrt{\log\sqrt{2\varrho(n)}}}
e^{-\beta\,2k\,\theta_k}, 
\end{eqnarray}
and hence 
\begin{equation}
\begin{aligned}
\sum_{\sigma \in A_{k}} \sum_{\xi \in S_{n}}\mu(\sigma)r(\sigma,\xi)
\geq e^{-\beta H_{n}(\boxminus)}\frac{1}{2Z}{n \choose k}
e^{-\tfrac{\beta}{2}\sqrt{\log\sqrt{2\varrho(n)}}}
e^{-\beta\,2k\,\theta_k}.
\label{eq: Gnp denom lbound3}
\end{aligned}
\end{equation}

\medskip\noindent
{\bf 3.}
Similar bounds can be derived for $\mathbb{P}_{\mu_{A_{k}}}\left[\tau_{A_{k'}} < \tau_{A_{k}}\right]$. 
Indeed, note that, by Lemma \eqref{lem:H near m}, $r\left(\sigma,\xi \right) = 1$ for all $\sigma \in A_{k} $ 
and all but $O(n^{2/3})$ many configurations $\xi \in S_{n}$. Therefore
\begin{eqnarray}
\mathrm{cap}(A_{k},A_{k'}) &=& n\,[1+o(1)] \sum_{\sigma \in A_{k}} \mu(\sigma) 
\mathbb{P}_{\sigma}\left[\tau_{A_{k'}} < \tau_{A_{k}}\right]\\ \nonumber
&=& n\,[1+o(1)]\,\mu(A_{k})\,\mathbb{P}_{\mu_{A_{k}}}\left[\tau_{A_{k'}} < \tau_{A_{k}}\right]
\end{eqnarray}
and hence
\begin{equation}
\mathbb{P}_{\mu_{A_{k}}}\left[\tau_{A_{k'}} < \tau_{A_{k}}\right] = [1+o(1)]\, 
\frac{\mathrm{cap}(A_{k}, A_{k'})}{n\mu(A_{k})}.
\end{equation}
Note that
\begin{equation}
\label{muAk}
\mu(A_{k}) = \frac{1}{Z}\sum_{\sigma \in A_{k}}\,e^{-\beta H_{n}(\sigma)},
\end{equation}
and we have already produced bounds for a sum like \eqref{muAk} in Lemma \ref{lem: cap G bound}. 
Referring to \eqref{eq: exp sum bnd1}, we see that  
\begin{eqnarray}
\mathbb{P}_{\mu_{A_{k}}}\left[\tau_{A_{k'}} < \tau_{A_{k}}\right] 
&\leq& [1+o(1)]\, \frac{\mathrm{cap}( A_{k}, A_{k'})}{\frac{1}{Z}\,e^{-(\beta + \frac{1}{\sqrt3})\sqrt{\log n}}
{n \choose k} e^{-\beta2k\theta_{k}}},\\
\mathbb{P}_{\mu_{A_{k}}}\left[\tau_{A_{k'}} < \tau_{A_{k}}\right] 
&\geq& [1+o(1)]\, \frac{\mathrm{cap}(A_{k}, A_{k'}}{\frac{1}{Z}\,n^{17/6}
e^{-(\beta + \frac{1}{\sqrt3})\sqrt{\log n}}{n \choose k}e^{-\beta 2k\theta_{k}}}.
\end{eqnarray}
Finally, we note that if we let $\Delta_{h}=h-h^{u}$, then 
\begin{equation}
\frac{e^{-\beta n[\Psi^{u}(\mathbf{s}^{u})-\Psi^{u}(\mathbf{m}^{u})]}}
{e^{-\beta n[\Psi(\mathbf{s})-\Psi(\mathbf{m})]}}
= e^{\beta nC_{\beta,h,p}\Delta_{h}},
\end{equation}
where $C_{\beta,h,p}$ is a constant that depends on the parameters $\beta$, $p$ and $h$. A similar 
expression follows for the ratio
\begin{equation}
{n \choose \frac{n}{2}(1+\mathbf{s}^u)}{n \choose \frac{n}{2}(1+\mathbf{m}^u)}^{-1}
\left[{n \choose \frac{n}{2}(1+\mathbf{s})}{n \choose \frac{n}{2}(1+\mathbf{m}}^{-1}\right]^{-1}.
\end{equation}
From this the statement of the lemma follows.
\end{proof}

%%%%%%% SECTION 5 %%%%%%%%%%%%%%%%%%%%%%%%%%

\section{Invariance under initial states and refined capacity estimates}
\label{s:refcapest}

In this section we use Lemma~\ref{lem: Hitting order} to control the time it takes $\{m(\xi_{t})\} _{t\geq0}$ 
to cross the interval $[\mathbf{t}^{u},\mathbf{s}^{u}]\cap[\mathbf{t}^{l},\mathbf{s}^{l}]$, which will be a 
good indicator of the time it takes $\{\xi_{t}\} _{t\geq0}$ to reach the basin of the stable state $\mathbf{s}$. 
In particular, our aim is to control this time by comparing it with the time it takes $\{\theta_{t}^{u}\} _{t\geq0}$ 
and $\{\theta_{t}^{l}\} _{t\geq0}$ defined in \eqref{eq: theta u l} to do the same for $\mathbf{s}^{u}$ and
$\mathbf{s}^{l}$. In Section~\ref{ss:rare} we derive bounds on the probability of certain rare events for \
the dynamics on $\ER_n(p)$ (Lemmas~\ref{lem:stray}--\ref{lem:ret} below). In Section~\ref{ss:proofuht} 
we use these bounds to prove that hitting times are close to being uniform in the starting configuration.  

%%%

\subsection{Estimates for rare events}
\label{ss:rare}

In this section we prove four lemmas that serve as a preparation for the coupling in 
Section~\ref{sec:couplingscheme}. Lemma~\ref{lem:stray} shows that the dynamics starting anywhere 
in $A_{\mathbf{M}}$ is unlikely to stray away from $A_{\mathbf{M}}$ by much during a time 
window that is comparatively small. Lemma~\ref{lem:tv bound exponentials} bounds the total 
variation distance between two exponential random variable whose means are close. 
Lemma~\ref{lem:update jumps} bounds the tail of the distribution of the first time when all the vertices
have been updated. Lemma~\ref{lem:ret} bounds the number of returns to $A_{\mathbf{M}}$ before
$A_{\mathbf{S}}$ is hit.

We begin by deriving upper and lower bounds on the number of jumps $N_{\xi}\left(t\right)$ taken by the 
process $\{\xi_{t}\} _{t\geq0}$ up to time $t$. By Lemma~\ref{lem: Rate bounds}, the jump rate from any 
$\sigma \in S_{n}$ is bounded by 
\begin{equation}
n\,e^{-2\beta\left(p+h\right)}\leq\sum_{\sigma'\in S_{n}}r\left(\sigma,\sigma'\right)\leq n.
\end{equation}
Hence $N_{\xi}\left(t\right)$ can be stochastically bounded from above by a Poisson random variable 
with parameter $tn$, and from below by a Poisson random variable with parameter $tne^{(-2\beta(p+h)}$. 
It therefore follows that, for any $M \geq 0$, 
\begin{equation}
\begin{aligned}
\mathbb{P}\left[N_{\xi}\left(t\right) \geq M\right] 
&\leq \chi_M(nt),\\
\mathbb{P}\left[N_{\xi}\left(t\right) < M\right] 
&\leq 1- \chi_M\big(nt\,e^{-2\beta(p+h)}\big),
\end{aligned}
\label{eq: jump rate bounds}
\end{equation}
where we abbreviate $\chi_M(u) = e^{-u} \sum_{k \geq M} u^k/k!$, $u \in \R$, $M \in \N$.  

\subsubsection{Localisation}
The purpose of the next lemma is to show that the probability of $\{\xi_{t}\}_{t \geq 0}$ straying 
too far from $A_{\mathbf{M}}$ within its first $n^{2}\log n$ jumps is very small. The seemingly 
arbitrary choice of $n^{2}\log n$ is in fact related to the Coupon Collector's problem.  

\begin{lemma}[{\bf Localisation}]
\label{lem:stray}
Let $\xi_{0}\in A_{\mathbf{M}}$,  $T=\inf\left\{ t\geq0\colon N_{\xi}\left(t\right)\geq n^{2}\log n\right\}$, 
and let $C_{1}\in\mathbb{R}$ be a sufficiently large constant, possibly dependent on $p$ and $h$ 
(but not on $n$). Then 
\begin{equation}
\label{eq:claimloc}
\mathbb{P}_{\xi_{0}}\left[\xi_{t}\in A_{\mathbf{M}+C_{1}n^{5/6}}
\mbox{ for some } 0\leq t\leq T\right] \leq e^{-n^{2/3}}.
\end{equation}
\end{lemma}

\begin{proof}
The idea of the proof is to show that $\{\xi_{t}\}_{t \geq 0}$ returns many times to $A_{\mathbf{M}}$ 
before reaching $A_{\mathbf{M}+C_{1}n^{5/6}}$. The proof comes in 3 Steps.

\medskip\noindent
{\bf 1.}
We begin by showing that $T\leq n^{2}\log n$ with probability $\geq 1- e^{-n^{3}}$ (in other words, it takes less than $n^{2}\log (n)$ time to make $n^{2}\log (n)$ steps). Indeed, by the 
second line of (\ref{eq: jump rate bounds}), 
\begin{eqnarray} \label{eq: N jump ubound}
&& \mathbb{P}\left[T>n^{2}\log n\right]\\[0.2cm] \nonumber
&&\qquad = \mathbb{P}\left[N_{\xi}\left(n^{2}\log n\right)<n^{2}\log n\right] \\
&& \qquad \leq 1-\chi_{n^{2} \log n}\Big((n^{3} \log n)\,e^{-2\beta\left(p+h\right)}\Big) \nonumber \\
&& \qquad \leq \sum_{k=0}^{n^{2}\log n}\exp\left(-(n^{3}\log n)\,
e^{-2\beta\left(p+h\right)}+k\log\left(\tfrac{e\,n^{3}\log n}{k}\right)\right)\nonumber \\
&& \qquad \leq (n^{2}\log n)\exp\left(-(n^{3}\log n)\,
e^{-2\beta\left(p+h\right)}+n^{5/2}\right)\nonumber \nonumber \\
&& \qquad\leq e^{-n^{3}},\nonumber
\end{eqnarray}
where for the second inequality we use that $k! \geq (\frac{k}{e})^{k}$, $k \in \N$, and for the third 
inequality that, for $n$ sufficiently large,
\begin{equation}
k\log\left(\tfrac{en^{3}\log n}{k}\right) \leq (n^{2}\log n) \log\left(e\,n^{3}\log n\right)\leq n^{5/2}.
\end{equation}
Next, observe that 
\begin{eqnarray}
&&\mathbb{P}_{\xi_{0}}\left[\xi_{t}\in A_{\mathbf{M}+C_{1}n^{5/6}}
\mbox{ for some }0\leq t\leq T\right]\\ \nonumber
&& = \mathbb{P}_{\xi_{0}}\left[\xi_{t}\in A_{\mathbf{M}+C_{1}n^{5/6}}
\mbox{ for some }0\leq t\leq T,\, T\leq n^{2}\log n \right]\\ \nonumber
&& \quad +\,  \mathbb{P}_{\xi_{0}}\left[\xi_{t}\in A_{\mathbf{M}+C_{1}n^{5/6}}
\mbox{ for some }0\leq t\leq T,\, T>n^{2}\log n \right]\\ \nonumber
&& \leq (n^{2}\log n) \max_{\sigma\in A_{\mathbf{M}}}
\mathbb{P}_{\sigma}\left[\tau_{A_{\mathbf{M}+C_{1}n^{5/6}}}<\tau_{A_\mathbf{M}}\right]
+ e^{-n^{3}}.\nonumber
\end{eqnarray}
Here, the inequality follows from \eqref{eq: N jump ubound} and the observation that the event 
$\xi_{t}\in A_{\mathbf{M}+C_{1}n^{5/6}} \mbox{ for some } 0\leq t\leq T$ with $T \leq n^{2}\log n$
is contained in the event that $A_{\mathbf{M}+C_{1}n^{5/6}}$ is visited before the $(n^{2}\log n)$-th 
return to $A_{\mathbf{M}}$. From Lemma~\ref{lem: Hitting order} and \eqref{eq:AkP} it follows that
\begin{equation}
\label{eq:Pest1}
\begin{aligned}
\max_{\sigma\in A_{\mathbf{M}}}\mathbb{P}_{\sigma}
\left[\tau_{A_{\mathbf{m}+C_{1}n^{5/6}}}<\tau_{A_\mathbf{M}}\right]\leq\tfrac{\left(1-a\right)}{2}
e^{2\beta\left(p\left|a\right|+h^{l}+\tfrac{p}{n}\right)}
e^{-\beta n[\Psi\left(b\right)-\Psi\left(a\right)]}
{n \choose \frac{n}{2}\left(1+b\right)}{n \choose \frac{n}{2}\left(1+a\right)}^{-1}
\end{aligned}
\end{equation}
with $a=\mathbf{m}/n$ and $b=\left(\mathbf{m}+C_{1}n^{5/6}\right)/n$. 

\medskip\noindent
{\bf 2.}
Our assumption on the parameters $\beta$, $p$ and $h$ is that $2\beta(p(a+\frac{2}{n})+h)
+\log(\frac{1-a}{1+a+\frac{2}{n}})$ is negative in two disjoint regions. Recall that the first region 
lies between $a_{1}=\frac{2\mathbf{M}}{n}-1$ and $a_{2}=\frac{2\mathbf{T}}{n}-1$. This, in particular, implies that the derivative of $2\beta(p(a+\frac{2}{n})+h)
+\log(\frac{1-a}{1+a+\frac{2}{n}})$ at $a=a_{1}$ is
\begin{equation}
2\beta p-\tfrac{1}{1-a_{1}}-\tfrac{1}{1+a_{1}}=-\delta_{1}<0
\label{eq:delta1}
\end{equation}
for some $\delta_{1}>0$. Recall that $\Psi(a)=-\frac{p}{2}a^{2}-ha$, so that $\Psi(b)-\Psi(a)
= (a-b)(\tfrac{p}{2}(a+b)+h)$, which gives
\begin{eqnarray}
\label{diff1}
&&e^{-\beta n\left[\Psi\left(b\right)-\Psi\left(a\right)\right]}
{n \choose \frac{n}{2}\left(1+b\right)}{n \choose \frac{n}{2}\left(1+a\right)}^{-1}\\ \nonumber
&&=\exp\left(\beta n\left(b-a\right)\left(pa+h\right)+\beta n\left(b-a\right)^{2}\tfrac{p}{2}
+\tfrac{n}{2}\log\left(\tfrac{\left(1+a\right)^{\left(1+a\right)}\left(1-a\right)^{\left(1-a\right)}}
{\left(1+b\right)^{\left(1+b\right)}\left(1-b\right)^{\left(1-b\right)}}\right) + O(\log n)\right),
\end{eqnarray}
where we use Stirling's approximation in the last line. Since $b=a+C_{1}n^{-1/6}$, we have
\begin{eqnarray}
\label{diff2}
\text{r.h.s. }\eqref{diff1} = \exp\left(\beta C_{1}n^{5/6}\left(pa+h\right)+\tfrac{p}{2}\beta C_{1}^{2}n^{2/3}
+\tfrac{n}{2}\log F\right)
\end{eqnarray}
with 
\begin{equation}
F=\left(1-U_n(a)\right)^{1+a} \left(1+V_n(a)\right)^{1-a}\left(W_n(a)\right)^{C_{1}n^{-1/6}},
\end{equation}
where
\begin{equation}
U_n(a) = \tfrac{C_{1}n^{-1/6}}{1+a+C_{1}n^{-1/6}}, \qquad
V_n(a) = \tfrac{C_{1}n^{-1/6}}{1-a-C_{1}n^{-1/6}}.
\end{equation}
From the Taylor series expansion of $\log\left(1+x\right)$ for $0\leq \left|x\right|<1$, we 
obtain
\begin{equation}
\begin{aligned}
&\tfrac{n}{2}\left(1+a\right)\log\left(1-U_n(a)\right)
\leq \tfrac{n}{2}\left(1+a\right)\left(-U_n(a)-\tfrac{1}{2}\left(U_n(a)\right)^{2}\right),\\
&\tfrac{n}{2}\left(1-a\right)\log\left(1+V_n(a)\right)
\leq \tfrac{n}{2}\left(1-a\right)\left(V_n(a)
-\tfrac{1}{2}\left(V_n(a)\right)^{2}+O(n^{-1/2})\right),
\end{aligned}
\end{equation}
and 
\begin{eqnarray}
\tfrac12 C_{1}n^{5/6}\log\left(\frac{U_n(a)}{V_n(a)}\right)
&=& \tfrac12 C_{1}n^{5/6}\log\left(\tfrac{1-a}{1+a}
\tfrac{1-a-C_{1}n^{-1/6}}{1-a}\tfrac{1+a}{1+a+C_{1}n^{-1/6}}\right)\\ \nonumber
&\leq& \tfrac12 C_{1}n^{5/6}\left(\log\left(\tfrac{1-a}{1+a}\right)
-\tfrac{C_{1}n^{-1/6}}{1-a}-U_n(a)-O(n^{-2/3})\right).
\end{eqnarray}
By the definition of $\mathbf{m}$, we have
\begin{equation}
C_{1}n^{5/6}\left(\beta\left(pa+h\right)+\log\left(\tfrac{1-a}{1+a}\right)\right)\leq0.
\end{equation}
Hence we get
\begin{equation}
\label{diff2*}
\begin{aligned}
&\beta C_{1}n^{5/6}\left(pa+h\right)+\tfrac{p}{2}\beta n^{2/3}C_{1}^{2}+\tfrac{n}{2}\log F\\ 
&\leq \tfrac{p}{2}\beta n^{2/3}C_{1}^{2}-\tfrac{\tfrac12 C_{1}\left(1+a\right)n^{5/6}}
{1+a+C_{1}n^{-1/6}}+\tfrac{\tfrac12 C_{1}\left(1-a\right)n^{5/6}}{1-a-C_{1}n^{-1/6}}
-\tfrac12 C_{1}^{2}n^{2/3}\,G
\end{aligned}
\end{equation}
with 
\begin{equation}
G=\tfrac{1}{1-a}+\tfrac{1}{1+a+C_{1}n^{-1/6}}+\left(\tfrac{1-a}{2}\right)
\left(\tfrac{1}{1-a-C_{1}n^{-1/6}}\right)^{2}+\left(\tfrac{1+a}{2}\right)\left(\tfrac{1}{1+a+C_{1}n^{-1/6}}\right)^{2}.
\end{equation}
Hence
\begin{eqnarray}
\text{r.h.s. }\eqref{diff2*}
&\leq& \tfrac{p}{2}\beta n^{2/3}C_{1}^{2}+\tfrac12 C_{1}^{2} n^{2/3}\left(\tfrac{1}{1-a-cn^{-1/6}}
+\tfrac{1}{1+a+C_{1}n^{-1/6}}\right)-\tfrac12 C_{1}^{2}n^{1/6}\,G\\ \nonumber
&\leq& n^{2/3} \tfrac12 C_{1}^{2}\left(p\beta-\tfrac{1}{2}\tfrac{1}{1-a-C_{1}n^{-1/6}}
-\tfrac{1}{2}\tfrac{1}{1+a+C_{1}n^{-1/6}}+O(n^{-1/6})\right)\\ \nonumber
&=& n^{2/3} \tfrac12  C_{1}^{2}\left(p\beta-\tfrac{1}{2}\tfrac{1}{1-a}-\tfrac{1}{2}\tfrac{1}{1+a}
+O(n^{-1/3})\right) \leq -\tfrac14 C_{1}^{2}\delta_{1}n^{2/3}.
\end{eqnarray}

\medskip\noindent
{\bf 3.}
Combine \eqref{eq:Pest1}, \eqref{diff1} and \eqref{diff2*}, and pick $C_{1}$ large enough, to get the claim in
\eqref{eq:claimloc}.
\end{proof}

%%%

\subsubsection{Update times}

The following two lemmas give useful bounds for the coupling scheme. The symbol $\simeq$ stands
for equality in distribution.

\begin{lemma}[{\bf Total variation between exponential distributions}]
\label{lem:tv bound exponentials}
$\mbox{}$\\
Let $X\simeq\mathrm{Exp}\left(\lambda\right)$ and $Y\simeq\mathrm{Exp}\left(\lambda+\delta\right)$.
Then the total variation distance between the distributions of $X$ and $Y$ is bounded by 
\begin{equation}
d_{TV}\left(X,Y\right)\leq\frac{2\delta}{\lambda+\delta}.
\end{equation}
\end{lemma}

\begin{proof}
Elementary.
\end{proof}

\begin{lemma}[{\bf Update times}] 
\label{lem:update jumps}
Let $T_{update}^{\xi}$ be the first time $\{\xi_{t}\} _{t\geq0}$ has experienced an update 
at every site:
\begin{equation}
T_{update}^{\xi}=\inf\left\{ t\geq0\colon\,\forall\,v\in V\,\exists\,0 \leq s\leq t\colon\,
\xi_{s}\left(v\right)=-\xi_{0}\left(v\right)\right\} .
\end{equation}
Then, for any $y>0$, 
\begin{equation}
\mathbb{P}\left[T_{update}^{\xi}\geq y\right]
\leq\frac{\exp\left(-\lambda y + \log n\right)}{1-\exp\left(-\lambda y\right)},
\qquad \lambda=e^{-\beta\left(2p+h\right)}.
\end{equation}
\end{lemma}

\begin{proof}
Recall that for $\sigma\in S_{n}$ and $v\in V$, $\sigma^{v}$ denotes the configuration satisfying 
$\sigma^{v}\left(w\right)=\sigma\left(w\right)$ for $w\neq v$, and $\sigma^{v}\left(v\right)
=-\sigma\left(v\right)$. From \eqref{eq:Hn as boundary size} and \eqref{eq:rate r} it
follows that 
\begin{equation}
r\left(\sigma,\sigma^{v}\right)\geq\lambda,
\end{equation}
and so $T_{update}^{\xi}$ is dominated by the maximum of $n$ i.i.d.\ $\mathrm{Exp}\left(\lambda\right)$ 
random variables. Therefore
\begin{eqnarray}
\mathbb{P}\left[T_{update}^{\xi}\leq y\right] 
& \geq & \left(1-e^{-\lambda y}\right)^{n}
= \exp\big(n\log (1-e^{-\lambda y})\big)\\ \nonumber
& \geq & \exp\left(-\tfrac{ne^{-\lambda y}}{1-e^{-\lambda y}}\right)
\geq 1-\tfrac{ne^{-\lambda y}}{1-e^{-\lambda y}},
\end{eqnarray}
which proves the claim.
\end{proof}

%%%

\subsubsection{Returns}

The next lemma establishes a lower bound on the number of returns to $A_{\mathbf{M}}$ before 
reaching $A_{\mathbf{S}}$. Let $g_{\xi_{0}}\left(A_{\mathbf{M}},A_{\mathbf{S}}\right)$ denote the
number of jumps that $\left\{ \xi_{t}\right\} _{t\geq0}$ makes into the set $A_{\mathbf{M}}$ before 
reaching $A_{\mathbf{S}}$. More precisely, let $\left\{ s_{i}\right\} _{i\in\mathbb{N}_{0}}$ denote 
the jump times of the process $\left\{ \xi_{t}\right\} _{t\geq0}$, i.e.,
$s_0=0$ and 
\begin{equation}
\label{eq:si}
s_{i}=\inf\left\{ s>s_{i-1}\colon\,\xi_{s}\neq\xi_{s_{i-1}}\right\},
\end{equation}
and define for the process $\xi_{t}$ commencing at $\xi_{0}$
\begin{equation}
g_{\xi_{0}}\left(A_{\mathbf{M}},A_{\mathbf{S}}\right)
=\left|\left\{ i\in\mathbb{N}_0\colon\,\xi_{s_{i}}\in A_{\mathbf{M}},
\xi_{s} \notin A_{\mathbf{S}}\,\,\forall\, s \leq s_{i}\right\} \right|.
\end{equation}

\begin{lemma}[{\bf Bound on number of returns}]
\label{lem:ret}
For any $\xi_{0}\in A_{\mathbf{M}}$ and any $\delta > 0$, 
\begin{equation}
\mathbb{P}_{\xi_{0}}\big[g_{\xi_{0}}\left(A_{\mathbf{M}},A_{\mathbf{S}}\right) 
< e^{([R_{p}(t) - R_{p}(m)] - \delta)n}\big] 
\leq e^{-\delta n + Cn^{2/3}}
\label{eq:numvisits}
\end{equation}
for some constant $C$ that does not depend on $n$.
\end{lemma}

\begin{proof}
Let $Y$ be a geometric random variable with probability of success given by $e^{-(R_{p}(t) - R_{p}(m))n
+Cn^{2/3}}$. Then, by Lemma~\ref{lem: Hitting order}, every time the process $\{\xi_{t}\}_{t \geq 0}$ 
starts over from $A_{M}$, it has a probability less than $\mathbb{P}^{u}_{\xi}[\tau_{S^{u}} < \tau_{M^{u}}]$ 
of making it to $A_{S}$. Using the bounds from that lemma, it follows that $Y$ is stochastically dominated 
by $g_{\xi_{0}}\left(A_{\mathbf{M}},A_{\mathbf{T}}\right)$. Hence
\begin{equation}
\mathbb{P}\left[Y\leq e^{([R_{p}(t) - R_{p}(m)]-\delta)n}\right] 
\leq e^{([R_{p}(t) - R_{p}(m)]-\delta)n}\,e^{-[R_{p}(t) - R_{p}(m)]n+Cn^{2/3}}\\
\leq e^{-\delta n+Cn^{2/3}}.
\end{equation}
\end{proof}

%%%

\subsection{Uniform hitting time}
\label{ss:proofuht}

In this section we show that if Theorem~\ref{thm:metER} holds for \emph{some} initial configuration in 
$A_{\mathbf{M}}$, then it holds for \emph{all} initial configurations in $A_{\mathbf{M}}$. The proof of 
this claim, which will be needed in Section~\ref{s:proofmettheorem}, relies on a \emph{coupling construction} 
in which the two processes starting anywhere in $A_{\mathbf{M}}$ meet with a sufficiently high probability 
long before either one reaches $A_{\mathbf{S}}$. Details of the coupling construction are given in 
Section~\ref{sec:couplingscheme}.

The idea of the proof is that for $\{\xi_{t}\}_{t \geq 0}$ starting in $A_{\mathbf{M}}$ 
the starting configuration is irrelevant for the metastable crossover time because the latter is very large. 
We will verify this by showing that ``local mixing'' takes place long before the crossover to $A_{\mathbf{S}}$ 
occurs. More precisely, we will show that if $\xi_{0},\tilde{\xi}_{0}$ are any two initial configurations in 
$A_{\mathbf{M}}$, then there is a coupling such that the trajectory $t \mapsto \xi_{t}$ intersects the 
trajectory $t \mapsto \tilde{\xi}_{t}$ well before either strays too far from $A_{\mathbf{M}}$. The coupling 
is such that there is a small but sufficiently large probability that $\xi_{t}$ and $\tilde{\xi}_{t}$ are identical 
once every spin at every vertex has had a chance to update, which occurs after a time $t$ that is not too 
large. It follows that after a large number of trials with high probability the two trajectories intersect.

\begin{proof}
Consider two copies of the process, $\{\xi_{t}\} _{t\geq0}$ and $\{\tilde{\xi}_{t}\} _{t\geq0}$. Let $\delta > 0$
and $T_{0}=e^{([R_{p,\beta,h}(\mathbf{t})-R_{p,\beta,h}(\mathbf{m})]-\delta)n}$. In order to simplify the 
notation and differentiate between the two processes, we abbreviate the crossover time $\tau_{A_{\mathbf{S}}}$ 
by 
\begin{equation}
\tau^{\xi}=\inf\left\{ t\geq s\colon\xi_{t}\in A_{\mathbf{S}}\right\},
\end{equation}
with a similar definition for $\tau^{\tilde{\xi}}$. We will show that $\mathbb{E}_{\xi_{0}}[\tau^{\xi}] 
\leq [1+o_{n}(1)]\mathbb{E}_{\tilde{\xi}_{0}} [\tau^{\tilde{\xi}} ]$, with the proof for the inequality in the other 
direction being identical. 

\medskip\noindent
{\bf 1.}
We start with the following observation. From Corollary~\ref{cor:expratio}, we immediately get that 
\begin{equation}
\mathbb{E}_{\xi_{0}}\big[\tau^{\xi}\big] / \mathbb{E}_{\xi_{0}}\big[\tau^{\tilde{\xi}}\big]
= e^{O(n^{2/3})}.
\label{eq:1expvaleq}
\end{equation}
Furthermore, the relation in \eqref{eq:1expvaleq} together with the initial steps in the proof of 
Theorem~\ref{thm:metER} implies that, for any initial configuration $\xi_{0}$,
\begin{equation}
\mathbb{E}_{\xi_{0}}\left[\tau^{\xi}\right]=e^{n[R_{p,\beta,h}(\mathbf{t})-R_{p,\beta,h}(\mathbf{m})] 
+ O(n^{2/3})}.
\label{eq:2expvaleq}
\end{equation}
\emph{Note:} Step 2 in Section ~\ref{s:proofmettheorem} shows that if $\xi_{0}$ is distributed according to 
the law $\mu_{A_{\mathbf{M}}}$, then 
\begin{equation}
\mathbb{E}_{\xi_{0}}[\tau^{\xi}] = e^{n[R_{p,\beta,h}(\mathbf{t})
-R_{p,\beta,h}(\mathbf{m})] + O(\log n)}.
\end{equation} 
(Recall from Section \ref{ss:rankorder} that $\mu_{A_{\mathbf{M}}}$ is the equilibrium distribution $\mu$ 
conditioned on the set $A_{\mathbf{M}}$.) Let $\{\xi_{t},\tilde{\xi}_{t}\} _{t\geq0}$ be the coupling of the 
two processes described in Section~\ref{sec:couplingscheme}, and note that 
\begin{equation}
\mathbb{E}_{\xi_{0}}\left[\tau_{A_{\mathbf{S}}}\right] 
= \hat{\mathbb{E}}_{\left(\xi_{0},\tilde{\xi}_{0}\right)}[\tau^{\xi}]
= \hat{\mathbb{E}}_{\left(\xi_{0},\tilde{\xi}_{0}\right)}\left[\tau^{\xi}
\mathbbm{1}_{\left\{ \xi_{T_{0}}=\tilde{\xi}_{T_{0}}\right\} }\right]
+\hat{\mathbb{E}}_{\left(\xi_{0},\tilde{\xi}_{0}\right)}\left[\tau^{\xi}
\mathbbm{1}_{\left\{ \xi_{T_{0}}\neq\tilde{\xi}_{T_{0}}\right\} }\right],
\label{eq:exp-couplke-time decomp}
\end{equation}
where $\hat{\mathbb{E}}$ denotes expectation with respect to the law of the joint process. The 
above inequality splits the expectation based on whether the coupling has succeeded (in merging 
the two processes) by time $T_{0}$ or not. Note that 
\begin{equation}
\tau^{\xi}\mathbbm{1}_{\left\{ \xi_{T_{0}}=\tilde{\xi}_{T_{0}}\right\} }
\leq\tau^{\tilde{\xi}}\mathbbm{1}_{\left\{ \xi_{T_{0}}=\tilde{\xi}_{T_{0}},\,
\tau^{\tilde{\xi}} \geq T_{0}\right\} }+\tau^{\xi}
\mathbbm{1}_{\left\{ \xi_{T_{0}}=\tilde{\xi}_{T_{0}},\,\tau^{\tilde{\xi}}<T_{0}\right\} },
\end{equation}
and
\begin{equation}
\begin{aligned}
&\tau^{\xi}\mathbbm{1}_{\left\{ \xi_{T_{0}}=\tilde{\xi}_{T_{0}},\,
\tau^{\tilde{\xi}}<T_{0}\right\} }\\ 
&\qquad = \tau^{\xi}\mathbbm{1}_{\left\{ \xi_{T_{0}}=\tilde{\xi}_{T_{0}},\,
\tau^{\tilde{\xi}}<T_{0},\,\left|\tilde{\xi}_{T_{0}}\right|<\mathbf{S}\right\} }
+\tau^{\xi}\mathbbm{1}_{\left\{ \xi_{T_{0}}=\tilde{\xi}_{T_{0}},\,
\tau^{\tilde{\xi}}<T_{0},\,\left|\tilde{\xi}_{T_{0}}\right|\geq\mathbf{S}\right\} }\\
&\qquad  \leq \tau^{\xi}\mathbbm{1}_{\left\{ \xi_{T_{0}}=\tilde{\xi}_{T_{0}},\,
\tau^{\tilde{\xi}}<T_{0},\,\left|\tilde{\xi}_{T_{0}}\right|<\mathbf{S}\right\} }+T_{0}.
\end{aligned}
\end{equation}
Also note from the definition of the coupling that, for any $\sigma \in S_{n}$ and any $A\subseteq S_{n}$, 
$\hat{\mathbb{E}}_{(\sigma,\sigma)}[\tau_{A}^{\xi}]=\mathbb{E}_{\sigma}[\tau_{A}]$ 
because the two trajectories merge when they start from the same site. Hence  
\begin{eqnarray}
&&\hat{\mathbb{E}}_{\left(\xi_{0},\tilde{\xi}_{0}\right)}\left[\tau^{\xi}
\mathbbm{1}_{\left\{ \xi_{T_{0}}=\tilde{\xi}_{T_{0}},\,\tau^{\tilde{\xi}}<T_{0},\,
\left|\tilde{\xi}_{T_{0}}\right|<\mathbf{S}\right\} }\right]\\ \nonumber 
&& \qquad = \sum_{\sigma\in\bigcup_{i<\chi_{_{1}}}A_{i}}
\hat{\mathbb{E}}_{\left(\xi_{0},\tilde{\xi}_{0}\right)}\left[\tau^{\xi}
\mathbbm{1}_{\left\{ \xi_{T_{0}}=\tilde{\xi}_{T_{0}}
= \sigma,\,\tau^{\tilde{\xi}}<T_{0}\right\} }\right]\\ \nonumber 
&& \qquad \leq \sum_{\sigma\in\bigcup_{i<\chi_{_{1}}}A_{i}}
\left(\hat{\mathbb{E}}_{\left(\sigma,\sigma\right)}\big[\tau^{\xi}\big]
+T_{0}\right)\hat{\mathbb{P}}_{\left(\xi_{0},\tilde{\xi}_{0}\right)}
\left[\xi_{T_{0}}=\tilde{\xi}_{T_{0}}=\sigma,\tau^{\tilde{\xi}}<T_{0}\right]\\ \nonumber 
&& \qquad \leq \Big(T_{0}+\max_{\sigma\in\bigcup_{i<\chi_{_{1}}}A_{i}}
\mathbb{E}_{\sigma}\left[\tau^{\sigma}\right]\Big)\,
\mathbb{P}_{\tilde{\xi}_{0}}\left[\tau<T_{0}\right],
\end{eqnarray}
where we use the Markov property. Similarly, observe that 
\begin{eqnarray}
\tau^{\xi}\mathbbm{1}_{\left\{ \xi_{T_{0}}\neq\tilde{\xi}_{T_{0}}\right\} } 
& = & \tau^{\xi}\mathbbm{1}_{\left\{ \xi_{T_{0}}\neq\tilde{\xi}_{T_{0}},\,
\tau^{\xi}\leq T_{0}\right\} }+\tau^{\xi}
\mathbbm{1}_{\left\{ \xi_{T_{0}}\neq\tilde{\xi}_{T_{0}},\,\tau^{\xi}>T_{0}\right\} }\\ \nonumber
& \leq & T_{0}+\tau^{\xi_{T_{0}}}\mathbbm{1}_{\left\{ \xi_{T_{0}}\neq\tilde{\xi}_{T_{0}}\,
\tau^{\xi}>T_{0}\right\}},
\end{eqnarray}
and
\begin{eqnarray}
&&\hat{\mathbb{E}}_{\left(\xi_{0},\tilde{\xi}_{0}\right)}\left[\tau^{\xi}
\mathbbm{1}_{\left\{ \xi_{T_{0}}\neq\tilde{\xi}_{T_{0}},\,\tau^{\xi}
>T_{0}\right\} }\right]\\ \nonumber  
&& \qquad = \sum_{\sigma\colon\left|\sigma\right|<\mathbf{S}}\sum_{\sigma'\neq\sigma}
\hat{\mathbb{E}}_{\left(\xi_{0},\tilde{\xi}_{0}\right)}\left[\tau^{\xi}
\mathbbm{1}_{\left\{ \xi_{T_{0}}=\sigma,\,\tilde{\xi}_{T_{0}}
=\sigma',\,\tau^{\xi}>T_{0}\right\} }\right]\\ \nonumber
&& \qquad = \sum_{\sigma\colon\left|\sigma\right|<\mathbf{S}}\sum_{\sigma'\neq\sigma}
\hat{\mathbb{E}}_{\left(\sigma,\sigma'\right)}\big[T_{0}+\tau^{\xi}\big]\,
\hat{\mathbb{E}}_{\left(\xi_{0},\tilde{\xi}_{0}\right)}\left[\mathbbm{1}_{\left\{ \xi_{T_{0}}
=\sigma,\,\tilde{\xi}_{T_{0}}=\sigma',\,\tau^{\xi}>T_{0}\right\} }\right]\\ \nonumber
&& \qquad \leq \max_{\sigma\in\bigcup_{i<\mathbf{S}}A_{i}}\left(T_{0}+\mathbb{E}_{\sigma}
\big[\tau^{\xi}\big]\right)\,\hat{\mathbb{P}}_{\left(\xi_{0},\tilde{\xi}_{0}\right)}
\big[\xi_{T_{0}}\neq\tilde{\xi}_{T_{0}}\big].
\end{eqnarray}
Thus, (\ref{eq:exp-couplke-time decomp}) becomes
\begin{equation}
\mathbb{E}_{\xi_{0}}\left[\tau^{\xi}\right] 
\leq 2T_{0}+\mathbb{E}_{\tilde{\xi}_{0}}\left[\tau_{A_{\mathbf{S}}}\right]
+\left(\mathbb{P}_{\tilde{\xi}_{0}}\left[\tau_{A_{\mathbf{S}}}<T_{0}\right]
+\hat{\mathbb{P}}_{\left(\xi_{0},\tilde{\xi}_{0}\right)}\big[\xi_{T_{0}}
\neq \tilde{\xi}_{T_{0}}\big]\right) \Big(T_{0}
+\max_{\sigma\in\bigcup_{i<\mathbf{S}}A_{i}}
\mathbb{E}_{\sigma}\left[\tau_{A_{\mathbf{S}}}\right]\Big).
\end{equation}
We will show that the leading term in the right-hand side is $\mathbb{E}_{\tilde{\xi}_{0}}
[\tau_{A_{\mathbf{S}}}]$, and all other terms are of smaller order. From \eqref{eq:2expvaleq} 
we know that $T_{0}$ is of smaller order, and that 
\begin{equation}
\max_{\sigma\in\bigcup_{i<\mathbf{S}}A_{i}}
\mathbb{E}_{\sigma}\left[\tau_{A_{\mathbf{S}}}\right] 
= e^{O(n^{2/3})}\mathbb{E}_{\tilde{\xi}_{0}}\left[\tau_{A_{\mathbf{S}}}\right].
\end{equation}
Hence it suffices to show that the sum $\mathbb{P}_{\tilde{\xi}_{0}}[\tau_{A_{\mathbf{S}}}<T_{0}]
+\hat{\mathbb{P}}_{(\xi_{0},\tilde{\xi}_{0})}[\xi_{T_{0}} \neq \tilde{\xi}_{T_{0}}]$ is exponentially 
small. We will show that it is bounded from above by $e^{-\delta n}$.

\medskip\noindent
{\bf 2.}
By Corollary~\ref{cor:longcoupling}, the probability $\hat{\mathbb{P}}_{(\xi_{0},\tilde{\xi}_{0})}
[\xi_{T_{0}}\neq \tilde{\xi}_{T_{0}}]$ is bounded from above by $e^{-\delta n + O(n^{2/3})}$. To 
bound $\mathbb{P}_{\tilde{\xi}_{0}}[\tau_{A_{\mathbf{S}}}<T_{0}]$, we first need to limit the number 
of steps that $\{\tilde{\xi_t}\}_{t\geq0}$ can take until time $T_{0}$. From (\ref{eq: jump rate bounds}) 
and Stirling's approximation we have that 
\begin{eqnarray}
\mathbb{P}\left[N_{\tilde{\xi}}\left(T_{0}\right)\ge 3nT_{0}\right] 
& \leq & \sum_{k=0}^{\infty}
e^{nT_{0}+k}\left(\frac{nT_{0}}{3nT_{0}+k}\right)^{3nT_{0}+k}\\ \nonumber
& \leq & e^{nT_{0}}\left(\tfrac{1}{3}\right)^{3nT_{0}}
\sum_{k=0}^{\infty} e^k \left(\tfrac{1}{3}\right)^{k}
\leq 11\left(0.91\right)^{3n e^{\tfrac12 [R_{p,\beta,h}(\mathbf{t})-R_{p,\beta,h}(\mathbf{m})]}}.
\end{eqnarray}
It therefore follows that with high probability $\{\tilde{\xi}_t\}_{t \geq 0}$ does not make more than 
$3nT_{0}$ steps until time $T_{0}$. Hence 
\begin{equation}
\mathbb{P}_{\tilde{\xi}_{0}}\big[\tau_{A_{\mathbf{S}}}<T_{0}\big] 
\leq \mathbb{P}_{\tilde{\xi}_{0}}\left[\tau_{A_{\mathbf{S}}}<T_{0},
N_{\tilde{\xi}}\left(T_{0}\right)<3nT_{0}\right]
+11\left(0.91\right)^{3n e^{\tfrac12\ [R_{p,\beta,h}(\mathbf{t})-R_{p,\beta,h}(\mathbf{m})]}}.
\end{equation}
Finally, note that the event $\{\tau_{A_{\mathbf{S}}}^{\tilde{\xi}}<T_{0},N_{\tilde{\xi}}(T_{0})<3nT_{0}\}$
implies that $\{\tilde{\xi}_{t}\}_{t\geq 0}$ makes fewer than $3nT_{0}$ returns to the set $A_{\mathbf{M}}$ 
before reaching $A_{\mathbf{S}}$. By Lemma ~\ref{lem:ret}, the probability of this event is bounded from 
above by $3nT_{0} e^{([R_{p,\beta,h}(\mathbf{t})-R_{p,\beta,h}(\mathbf{m})]-\delta)n+Cn^{2/3}}
= 3ne^{-\tfrac12 \delta n+Cn^{2/3}}$, and hence 
\begin{equation}
\mathbb{P}_{\tilde{\xi}_{0}}\big[\tau_{A_{\mathbf{S}}}^{\tilde{\xi}}<T_{0}\big]
\leq 4n\, e^{-\tfrac12\delta n+Cn^{2/3}}.
\label{eq: quick escape bound}
\end{equation}
Finally, from (\ref{eq: quick escape bound}) we obtain 
\begin{equation}
\mathbb{E}_{\xi_{0}}\big[\tau_{A_{\mathbf{S}}}^{\xi}\big]
=\mathbb{E}\big[\tau_{A_{\mathbf{S}}}^{\tilde{\xi}}\big]\left[1+o\left(1\right)\right],
\end{equation}
which settles the claim. 
\end{proof}

%%%%%%%%%%%% SECTION 6 %%%%%%%%%%%%%%%%%%%%%%%%

%%%%%%%%%%%%% SECTION 7 %%%%%%%%%%%%%%%%%%%%%

\section{A coupling scheme}
\label{sec:couplingscheme}

In this section we define a coupling of $(\xi_{t})_{t \geq 0}$ and $\{\tilde{\xi}_{t}\}_{t \geq 0}$ with arbitrary 
starting configurations in $A_{\mathbf{M}}$. The coupling is divided into a short-term scheme, defined in 
Section~\ref{ss:short} and analysed in Lemma \ref{lem:couplingscheme} below, followed by a long-term 
scheme, defined in Section~\ref{ss:long} and analysed Corollary \ref{cor:longcoupling} below. The goal
of the coupling is to keep the process $\{m(\xi_{t})\}_{t\geq0}$ bounded by $\{\theta_{t}^{u}\}_{t\geq0}$ 
from above and bounded by $\{\theta_{t}^{l}\}_{t\geq0}$ from below (the precise meaning will become 
clear in the sequel). 

%%%

\subsection{Short-term scheme}
\label{ss:short}

\begin{lemma}[{\bf Short-term coupling}]
\label{lem:couplingscheme}
With $\mathbb{P}_{\ER_n(p)}$-probability tending to $1$ as $n\to\infty$, there is a coupling 
$\{\xi_{t},\tilde{\xi}_{t}\}_{t \geq 0}$ of $\{\xi_{t}\}_{t \geq 0}$ and $\{\tilde{\xi}_{t}\}_{t \geq 0}$ 
such that 
\begin{equation}
\mathbb{P}\big[\xi_{2n}\neq\tilde{\xi}_{2n}] \leq e^{n^{-2/3}}
\end{equation}
for any initial states $\xi_{0} \in A_{\mathbf{M}}$ and $\tilde{\xi}_{0} \in A_{\mathbf{M}}$.
\end{lemma}

\begin{proof}
The main idea behind the proof is as follows. Define
\begin{equation}
\label{eq:W1}
W_{1}^{t} =\lbrace{ v \in V\colon\, \xi_{t}(v)=-\tilde{\xi}_{t}(v)\rbrace}
= \xi_{t} \Delta \tilde{\xi}_{t},
\end{equation} 
i.e., the symmetric difference between the two configurations $\xi_{t}$ and $\tilde{\xi}_{t}$, and
\begin{equation}
\label{eq:W2}
W_{2}^{t} = \lbrace{ v \in V\colon\, \xi_{t}(v)=\tilde{\xi}_{t}(v)\rbrace} = V \setminus W_{1}^{t}.
\end{equation}
The coupling we are about to define will result in the set $W_{1}^{t}$ shrinking at a higher rate than the set 
$W_{2}^{t}$, which will imply that $W_{1}^{t}$ contracts to the empty set. The proof comes in 8 Steps.

\medskip\noindent
{\bf 1.}
We begin with bounds on the relevant transition rates that will be required in the proof. Recall from 
Lemma \ref{lem: Rate bounds} (in particular, \eqref{eq:probl R union} and \eqref{eq:R- union}) that with 
$\mathbb{P}_{\ER_{n}(p)}$-probability at least $1-e^{-2n}$ there are at most $2n^{2/3}$ vertices 
$v\in\overline{\xi_{t}}$ (i.e., $\xi_{t}\left(v\right)=-1$) such that $|E(v,\xi_t)|=\left|\left\{ w\in\xi_{t} \colon\,
\left(v,w\right)\in E\right\} \right|\geq p\left|\xi_{t}\right|+n^{2/3}$, and similarly at most $2n^{2/3}$ 
vertices $v\in\overline{\xi_{t}}$ such that $\left|E(v,\xi_t)\right|\leq p\left|\xi_{t}\right|-n^{2/3}$. Analogous 
bounds are true for $\tilde{\xi}_{t}$, $t\geq0$. Denote the set of \emph{bad} vertices for $\xi_{t}$ by 
\begin{equation}
B_{t}=\big\{ v\in \overline{\xi_{t}}\colon\,\big| |E(v,\xi_t)|-p|\xi_{t}|\big|\geq n^{2/3}\},
\end{equation}
and the set of bad vertices for $\tilde{\xi}_{t}$ by $\tilde{B}_{t}$. Let $\hat{B}_{t}=B_{t}\cup\tilde{B}_{t}$. 
Recall that $\xi_{t}^{v}$ denotes the configuration obtained from $\xi_{t}$ by flipping the sign at vertex 
$v\in V$. If $v\notin\hat{B}_{t}$, then from \eqref{eq:Hn as boundary size} and Lemma \ref{lem: degree bound}
it follows that, for $v\notin\xi_{t}$, 
\begin{eqnarray}
H_{n}\left(\xi_{t}^{v}\right)-H_{n}\left(\xi_{t}\right) 
& = & \tfrac{2}{n}\left(\left|\partial_{E}\xi_{t}^{v}\right|-\left|\partial_{E}\xi_{t}\right|\right)-2h\\ \nonumber
& = & \tfrac{2}{n}\left(\mathrm{deg}\left(v\right)-2\left|E(v,\xi_t)\right|\right)-2h\\ \nonumber
& \leq & \tfrac{2}{n}\left(pn+n^{1/2}\log n-2p\left|\xi_{t}\right|+2n^{2/3}\right)-2h,
\end{eqnarray}
and similarly, for $v\in\xi_{t}$, 
\begin{equation}
H_{n}\left(\xi_{t}^{v}\right)-H_{n}\left(\xi_{t}\right)\leq\tfrac{2}{n}\left(pn+n^{1/2}\log n
-2p\left(n-\left|\xi_{t}\right|\right)+2n^{2/3}\right)+2h.
\end{equation}
Again, by \eqref{eq:Hn as boundary size} and Lemma \ref{lem: degree bound},
we have similar lower bounds, namely, if $v\notin \hat{B}_{t}$, then, for $v\notin \xi_{t}$, 
\begin{equation}
H_{n}(\xi_{t}^{v})-H_{n}(\xi_{t})
\geq\tfrac{2}{n}\left(pn-n^{1/2}\log n-2p\left|\xi_{t}\right|-2n^{2/3}\right)-2h,
\end{equation}
and, for $v\in\xi_{t}$,
\begin{equation}
H_{n}(\xi_{t}^{v})-H_{n}(\xi_{t})
\geq\tfrac{2}{n}\left(pn-n^{1/2}\log n-2p\left(n-\left|\xi_{t}\right|\right)
-2n^{2/3}\right)+2h.
\end{equation}
Identical bounds hold for $H_{n}(\tilde{\xi}_{t}^{v})-H_{n}(\tilde{\xi}_{t})$. Therefore, if 
$v\notin\hat{B}_{t}$, and if either $v\in\xi_{t}\cap\tilde{\xi}_{t}$
or $v\notin\xi_{t}\cup\tilde{\xi}_{t}$, then
\begin{equation}
\label{absratediff}
\begin{aligned}
\left|r(\xi_{t},\xi_{t}^{v})-r(\tilde{\xi}_{t},\tilde{\xi}_{t}^{v})\right| 
&=\left| e^{-\beta[H_{n}(\xi_{t}^{v})-H_{n}(\xi_{t})]_{+}}
-e^{-\beta[H_{n}(\tilde{\xi}_{t}^{v})-H_{n}(\tilde{\xi}_{t})]_{+}}\right|\\
&=e^{-\beta[H_{n}(\xi_{t}^{v})-H_{n}(\xi_{t})]_{+}}
\left|1-e^{\beta\big([H_{n}(\xi_{t}^{v})
-H_{n}(\xi_{t})]_{+}-[H_{n}(\tilde{\xi}_{t}^{v})
-H_{n}(\tilde{\xi}_{t})]_{+}\big)}\right|\\
&\leq\left[1+o_{n}(1)\right] e^{-\beta[H_{n}(\xi_{t}^{v})
-H_{n}(\xi_{t})]_{+}}
\left(e^{8\beta n^{-1/3}
+ \tfrac{4p}{n}\big(|\xi_{t}|-|\tilde{\xi}_{t}| \big)}-1\right)\\
& \leq\left[1+o_{n}(1)\right]\left(e^{8\beta n^{-1/3}
+ \tfrac{4p}{n}\big(|\xi_{t}|-|\tilde{\xi}_{t}|\big)}-1\right)\\
& \leq\left[1+o_{n}(1)\right]\left(8\beta n^{-1/3}
+ \tfrac{4p}{n}\big(|\xi_{t}|-|\tilde{\xi}_{t}|\big)\right).
\end{aligned}
\end{equation}

\medskip\noindent
{\bf 2.}
Having established the above bounds on the transition rates, we give an explicit construction of the coupling 
$\{\xi_{t},\tilde{\xi}_{t}\} _{t\geq0}$. 

\begin{definition}
{\rm
$\mbox{}$
\begin{itemize} 
\item[(I)] 
We first define the coupling for time $t=0$. For $t>0$ this coupling will be \emph{renewed} after each 
renewal of $\{\xi_{t},\tilde{\xi}_{t}\}_{t \geq 0}$, i.e., whenever either of the two processes jumps to a 
new state. To that end, for every $v\in W_{2}^{0}$ (i.e., $\xi_{0}(v)=\tilde{\xi}_{0}(v)$), couple the exponential 
random variables $e_{0}^{v}\simeq \mathrm{Exp}(r(\xi_{0},\xi_{0}^{v}))$ and $\tilde{e}_{0}^{v} \simeq \mathrm{Exp}
(r(\tilde{\xi}_{0},\tilde{\xi}_{0}^{v}))$ associated with the transitions $\xi_{0}\to\xi_{0}^{v}$ and 
$\tilde{\xi}_{0} \to\tilde{\xi}_{0}^{v}$ according to the following scheme: 
\begin{enumerate}
\item 
Choose a point 
\[
\left(x,y\right)\in\{\left(x',y'\right)\colon\,0\leq x'<\infty,\,0\leq y'\leq r\left(\xi_{0},
\xi_{0}^{v}\right) e^{-r\left(\xi_{0},\xi_{0}^{v}\right)x'}\}
\] 
uniformly and set $e_{0}^{v}=x$. Note that, indeed, this gives $e_{0}^{v} \simeq \mathrm{Exp}(r(\xi_{0},\xi_{0}^{v}))$.
\item 
If the value $y$ from step 1 satisfies $y\leq r(\tilde{\xi}_{0},\tilde{\xi}_{0}^{v})
\exp(-r(\tilde{\xi}_{0},\tilde{\xi}_{0}^{v})x)$, then set $\tilde{e}_{0}^{v}
=e_{0}^{v}=x$. Else, choose
\[
\left(x*,y*\right)\in\left\{ \left(x',y'\right)\colon0\leq x'<\infty,\, 
r\left(\xi_{0},\xi_{0}^{v}\right) e^{-r\left(\xi_{0},\xi_{0}^{v}\right)x'}
<y'\leq r(\tilde{\xi}_{0},\tilde{\xi}_{0}^{v})
e^{-r(\tilde{\xi}_{0},\tilde{\xi}_{0}^{v})x'}\right\}
\]
uniformly and independently from the sampling in step 1, and set $\tilde{e}_{0}^{v}=x*$. Note that this 
too gives $e_{0}^{v}\simeq \mathrm{Exp}(r(\tilde{\xi}_{0},\tilde{\xi}_{0}^{v}))$.
\end{enumerate}
\item[(II)] 
For every $v\in W_{1}^{0}$, sample the random variables $e_{0}^{v}\simeq \mathrm{Exp}(r(\xi_{0},\xi_{0}^{v}))$ 
and $\tilde{e}_{0}^{v} \simeq \mathrm{Exp}(r(\tilde{\xi}_{0},\tilde{\xi}_{0}^{v}))$ associated with the transitions 
$\xi_{0}\to\xi_{0}^{v}$ and $\tilde{\xi}_{0} \to\tilde{\xi}_{0}^{v}$ independently. At time $t=0$, we use the above 
rules to define the \emph{jump} times associated with any vertex $v\in V$. Recall that $W_{2}^{0}$ is the set 
of vertices where the two configurations agree in sign. The aim of the coupling defined above is to preserve 
that agreement. Following every renewal, we re-sample all transition times anew (i.e., we choose new copies 
of the exponential variables as was done above). We proceed in this way until the first of the following two 
events happens: either $\xi_{t} = \tilde{\xi}_{t}$, or $n\log n$ transitions have been made by either one of the 
two processes. 
\end{itemize}
}
\end{definition}

\medskip\noindent
{\bf 3.}
Note that the purpose of limiting the number of jumps to $n\log n$ is to permit us to employ Lemma \ref{lem:stray}, 
which in turn we use to maintain control on the two processes being similar in volume. Further down we will also 
show that, with high probability, in time $2n$ no more than $n\log n$ transitions occur. By \eqref{absratediff} and 
Lemma \ref{lem:tv bound exponentials}, if $v\notin\hat{B}_{t}$, then
\begin{equation}
\mathbb{P}\left[e_{t}^{v}\neq e_{t}^{v}\right]
\leq\frac{2(8\beta n^{-1/3}+ \frac{4p}{n}
(|\xi_{t}|-|\tilde{\xi}_{t}|))}{e^{-2\beta(p+h)}}.
\label{eq:tv1}
\end{equation}
On the other hand, if $v\in\hat{B}_{t}$ and we let $z=\frac{2\beta\left(p+h\right)}{1-e^{-2\beta (p+h)}}$, then
\begin{eqnarray}
\mathbb{P}\left[e_{t}^{v}\neq e_{t}^{v}\right] 
&=& d_{TV}\left(e_{t}^{v},e_{t}^{v}\right)
\leq e^{-2\beta\left(p+h\right)}\int_{0}^{z} dx\,
\exp\left(-xe^{-2\beta\left(p+h\right)}\right)\nonumber \\
&=& 1-\exp\left(-\frac{2\beta\left(p+h\right)e^{-2\beta\left(p+h\right)}}
{1-e^{-2\beta\left(p+h\right)}}\right).
\label{eq:tv2}
\end{eqnarray}
Observe that, for $v\in W_{1}^{t}$, with  $\mathbb{P}_{\ER_n(p)}$-high probability
\begin{equation}
\label{eq:sumrates}
\sum_{v\in W_{1}^{t}} \left[r(\xi_{t},\xi_{t}^{v}) + r(\tilde{\xi}_{t},\tilde{\xi}_{t}^{v})\right] 
\geq \left[1+o_{n}(1)\right]\left|W_{1}^{t}\right|.
\end{equation}
Indeed, by the concentration inequalities of Lemma \ref{lem: degree bound} and the bound in 
Lemma \ref{lem:stray}, it follows that $|\xi_{t}|$ and $|\tilde{\xi}_{t}|$ are of similar magnitude:
\begin{equation}
\label{eq:maxprob}
\mathbb{P}\big[||\xi_{t}| - |\tilde{\xi}_{t}|| \geq n^{5/6}\big]
\leq e^{-n^{2/3}}.
\end{equation} 
Therefore, with $\mathbb{P}_{\ER_{n}(p)}$-high probability, for all but $O(n^{2/3})$ such $v$,
\begin{equation}
\label{eq:oppositeH}
H(\xi_{t}) - H(\xi_{t}^{v}) = \left[1+o_{n}(1)\right] 
\big[H\big(\tilde{\xi}_{t}^{v}\big)-H\big(\tilde{\xi}_{t})\big],
\end{equation}
from which \eqref{eq:sumrates} follows. The rate at which the set $W_{2}^{t}$ shrinks is equal to the rate 
at which it loses $v\in W_{2}^{t}$ such that also $v \notin \hat{B}_{t}$, plus the rate at which it loses 
$v\in W_{2}^{t}$ such that $v \in \hat{B}_{t}$. From \eqref{absratediff} it follows that the former is 
bounded from above by $|W_{2}^{t}|(8\beta n^{-1/3} + \frac{4p}{n}(|\xi_{t}|-|\tilde{\xi}_{t}|))$, while
 by \eqref{eq:probl R union} the latter is bounded by $4n^{2/3}$. Therefore, defining the stopping time
\begin{equation}
\upsilon_{i} = \inf \left\{t\colon\, |W_1^{t}|=i\right\},
\end{equation}
we have that
\begin{equation}
\label{eq:prob Dt contracts}
\mathbb{P}_{(\xi_{t},\tilde{\xi}_{t})} \big[ \upsilon_{|W_1^{t}|-1} < \upsilon_{|W_1^{t}|+1} \big] 
\geq \frac{|W_{1}^{t}|}{|W_{2}^{t}| [8\beta n^{-1/3} 
+ \frac{4p}{n}(|\xi_{t}| -|\tilde{\xi}_{t}|)]
+ 4n^{2/3}+|W_{1}^{t}|}.
\end{equation}
From Lemma \ref{lem:stray} we know that (with probability $\geq 1-e^{-n^{2/3}}$) neither $|\xi_{t}|$ 
nor $|\tilde{\xi}_{t}|$ will stray beyond $\mathbf{M} + Cn^{5/6}$ and $\mathbf{M}- Cn^{5/6}$ within 
$n^{2}\log n$ steps. Thus,
\begin{equation}
\left||\xi_{t}|-|\tilde{\xi}_{t}|\right| \leq Cn^{5/6}.
\end{equation}
Hence, for $\left|W_{1}^{t}\right| \geq n^{6/7}$ we have that \eqref{eq:prob Dt contracts} is equal to 
$1-o_{n}(1)$. 

\medskip\noindent
{\bf 4.}
Next suppose that  $\left|W_{1}^{t}\right| < n^{6/7}$. To bound the rate at which the set $W_{2}^{t}$ 
shrinks, we argue as follows. The rate at which a matching vertex $v$ becomes non-matching equals
\begin{equation}
|r( \xi_{t},\xi_{t}^v) - r(\tilde{\xi}_{t},\tilde{\xi}_{t}^{v})|.
\end{equation}
Let
\begin{equation}
\begin{aligned}
B_{1} &= -2h + \tfrac{2}{n}\left( \mathrm{deg}\left(v\right) - 2|E(v,\xi_{t})|\right),\\
B_{2} &= -2h + \tfrac{2}{n}\left( \mathrm{deg}\left(v\right) - 2|E(v,\tilde{\xi}_{t})|\right),\\
B_{3} &= -h + \tfrac{1}{n}\left( \mathrm{deg}\left(v\right) - 2|E(v,\xi_{t}\cap \tilde{\xi}_{t})|\right).
\end{aligned}
\end{equation}
For $v\notin \xi_{t} \cup \tilde{\xi}_{t}$, we can estimate
\begin{equation}
\label{eq:ratdiffbnd}
\begin{aligned}
|r( \xi_{t},\xi_{t}^v) - r(\tilde{\xi}_{t},\tilde{\xi}_{t}^{v})| 
&= \left| e^{-\beta\left[B_{1}\right]_{+}} - e^{-\beta\left[B_{2}\right]_{+}}\right|\\
&\leq e^{-2\beta [B_{3}]_+} \left| e^{-\frac{4\beta}{n} 
|E(v,\xi_{t} \backslash \tilde{\xi}_{t})|} 
-e^{-\frac{4\beta}{n} |E(v,\tilde{\xi}_{t} \backslash \xi_{t})|} \right| \\
&\leq e^{-2\beta [B_{3}]_+} \tfrac{4\beta}{n}\left| |E(v,\xi_{t} \backslash \tilde{\xi}_{t})| 
-|E(v,\tilde{\xi}_{t} \backslash \xi_{t})| \right| \\
& \leq \left[1+o_{n}(1)\right] e^{-2\beta[-p\mathbf{m}-h]_+}
\tfrac{4\beta}{n} \left|E\left(v, W_{1}^{t}\right)\right|,
\end{aligned}
\end{equation}
where we note that $W_1^t = \xi_{t} \backslash \tilde{\xi}_{t} \cup \tilde{\xi}_{t} \backslash \xi_{t}$ and 
use that, by Lemma \ref{lem: degree bound} and the bound $|\xi_{t} \backslash \tilde{\xi}_{t}| \leq 
|W_{1}^{t}| \leq n^{6/7}$,
\begin{equation}
\begin{aligned}
&\frac{1}{n}\left(\mathrm{deg}(v) - 2|E(v,\xi_{t}\cap \tilde{\xi}_{t})|\right) 
= [1+o_{n}(1)]\,p\left(1 - \tfrac{2|\xi_{t}\cap \tilde{\xi}_{t}|}{n}\right)\\
&= [1+o_{n}(1)]\,p\left(1 - \tfrac{2|\xi_{t}|}{n}\right)
= [1+o_{n}(1)]\,p\left(1 - \tfrac{2\mathbf{M}}{n}\right)
= -[1+o_{n}(1)]\,p\mathbf{m}.
\end{aligned}
\end{equation}
Note that since $\xi_{t}$ and $\tilde{\xi}_{t}$ disagree at most at $n^{6/7}$ vertices, and since $|\xi_{t}|
=[1+o_{n}(1)]\mathbf{M} = [1+o_{n}(1)]\frac{n}{2}(1+\mathbf{m})$, we have that $| v \in \xi _{t} \cup 
\tilde{\xi}_{t} | = [1+o_{n}(1)]\frac{n}{2}(1-\mathbf{m})$. Furthermore, since $|E(v, W_{1}^{t})| \leq 
[1+o_n(1)] p|W_{1}^{t}|$ and $|V|=n$, we have that
\begin{equation}
\label{eq:sum abs rate difference}
\sum_{v\notin \xi_{t}\cup\tilde{\xi}_{t}} 
| r( \xi_{t},\xi_{t}^v) - r(\tilde{\xi}_{t},\tilde{\xi}_{t}^{v})| 
\leq \left[1+o_{n}(1)\right] e^{-2\beta[-p\mathbf{m}-h]_+}\,2\beta p (1-\mathbf{m})\,|W_{1}^{t}|.
\end{equation}
For $v\in\xi_{t}\cap\tilde{\xi}_{t}$, on the other hand, Lemma \ref{lem:H near m} gives that 
\begin{equation}
r(\xi_{t},\xi_{t}^{v})=r(\tilde{\xi}_{t},\tilde{\xi}_{t}^{v})=1
\end{equation}
for all but $O(n^{2/3})$ many such $v$. If $v$ is such that $r(\xi_{t},\xi_{t}^{v}) \neq 
r(\tilde{\xi}_{t},\tilde{\xi}_{t}^{v})$, then a computation identical to the one leading to 
\eqref{eq:sum abs rate difference} gives that 
\begin{equation}
\sum_{v\in\xi_{t}\cap\tilde{\xi}_{t}} |r\left(\xi_{t},\xi_{t}^{v}\right)
-r(\tilde{\xi}_{t},\tilde{\xi}_{t}^{v})| = O(n^{-1/6}) \big|W_{1}^{t}\big|.
\label{eq:small unmatching rate}
\end{equation}
Combining \eqref{eq:sum abs rate difference} and \eqref{eq:small unmatching rate}, we obtain
\begin{equation}
\label{eq:shrink1}
\sum_{v \in W_{t}^2} |r\left(\xi_{t},\xi_{t}^{v}\right) -r(\tilde{\xi}_{t},\tilde{\xi}_{t}^{v})| 
\leq \left[1+o_{n}(1)\right] e^{-2\beta[-p\mathbf{m}-h]_+}\,2\beta p (1-\mathbf{m})\,\big|W_{1}^{t}\big|,
\end{equation}
which bounds the rate at which $W_2^{t}$ shrinks.

\medskip\noindent
{\bf 5.}
To bound the rate at which the set $W_{1}^{t}$ shrinks, we argue as follows. The rate at which 
a non-matching vertex $v$ becomes a matching equals
\begin{equation}
r(\xi_{t},\xi_{t}^{v})+r(\tilde{\xi}_{t},\tilde{\xi}_{t}^{v}).
\end{equation}
Note that, for every $v\in W_{1}^{t}$, 
\begin{equation}
\label{eq:jumpgreaterthan1}
H\left(\xi_{t}^{v}\right)-H\left(\xi_{t}\right)
=-\left[1+o_{n}(1)\right]\big[H(\tilde{\xi}_{t}^{v})-H(\tilde{\xi}_{t})\big],
\end{equation}
since, up to an arithmetic correction of magnitude $|W_{1}^{t}|=O(n^{6/7})$, $v$ has the same number 
of neighbours in $\xi_{t}$ as in $\tilde{\xi}_{t}$. Hence it follows that
\begin{equation}
\sum_{v \in W_{1}^{t}} \left[r(\xi_{t},\xi_{t}^{v})+r(\tilde{\xi}_{t},\tilde{\xi}_{t}^{v})\right]
=\left[1+o_{n}(1)\right]\big(e^{-2\beta [-p\mathbf{m}-h]_+}+e^{-2\beta [p\mathbf{m}+h]_+}\big)|W_{1}^{t}|,
\label{eq:matching rate}
\end{equation}
which bounds the rate at which $W_{1}^{t}$ shrinks. 

\medskip\noindent
{\bf 6.}
Combining \eqref{eq:shrink1} and \eqref{eq:matching rate},
and noting that $p\mathbf{m}+h<0$, we see that $\left|W_{1}^{t}\right|$ is contracting when
\begin{equation}
[1+o_n(1)] \left(e^{2\beta(p\mathbf{m}+h)} + 1\right)\,\left|W_{1}^{t}\right|
>\left[1+o_{n}(1)\right]\left(e^{2\beta(p\mathbf{m}+h)}\,4\beta p\right)\,\left|W_{1}^{t}\right|.
\end{equation}
For this in turn it suffices that 
\begin{equation}
\label{eq:mineq}
e^{2\beta(p\mathbf{m}+h)} + 1 > e^{2\beta(p\mathbf{m}+h)} 2\beta p(1-\mathbf{m}).
\end{equation}

\medskip\noindent
{\bf 7.}
Note from the definition of $\mathbf{m}$ in \eqref{eq: mts} that, up to a correction factor of $1+o_{n}(1)$,
$\mathbf{m}$ solves the equation $J(\mathbf{m})=0$ with
\begin{equation}
J_{p,\beta,h}(a) = 2\lambda\left(a+\tfrac{h}{p}\right)+\log\left(\tfrac{1-a}{1+a}\right), \qquad \lambda = \beta p,
\end{equation}
i.e., 
\begin{equation}
\tfrac{1+\mathbf{m}}{1-\mathbf{m}} = e^{2\lambda\left(\mathbf{m}+\frac{h}{p}\right)}.
\end{equation}
Hence from \eqref{eq:mineq} it follows that $|W_{1}^{t}|$ is contracting whenever we are in the metastable 
regime and the inequality 
\begin{equation}
\lambda < \frac{1}{1-\mathbf{m}^2}
\label{eq:1lambda}
\end{equation}
is satisfied. From \eqref{eq:Jprime} it follows that the equality
\begin{equation}
\lambda = \frac{1}{1-a^2}
\end{equation}
holds for $a=a_{\lambda}=-\sqrt{1-1/\lambda}$, which in turn is bounded between the values $\mathbf{m} 
< a_{\lambda} < \mathbf{t}< 0$, and therefore 
\begin{equation}
 \frac{1}{1-\mathbf{m}^2} > \frac{1}{1-a_{\lambda}^2} = \lambda.
\end{equation}
This shows that $|W_{1}^{t}|$ is contracting whenever we are in the metastable regime.

\medskip\noindent
{\bf 8.}
To conclude, we summarise the implication of the contraction of the process $|W_{1}^{t}|$. The probability 
in \eqref{eq:prob Dt contracts} is equal to $1-O_{n}(n^{\frac{5}{6}-\frac{6}{7}})$ for $|W_{1}^{t}| > n^{6/7}$, 
and is strictly larger than $\frac{1}{2}$ for $|W_{1}^{t}| \leq n^{6/7}$. Furthermore, from \eqref{eq:sumrates}
we know that the rate at which $W_1^{t}$ shrinks is $\geq 1$. This allows us to ensure that sufficiently 
many steps are made by time $2n$ to allow $W_{1}^{t}$ to contract to the empty set. In particular, the 
number steps taken by $W_{1}^{t}$ up to time $2n$ is bounded from below by a Poisson point process 
$N(t)$ with unit rate, for which we have 
\begin{eqnarray}
\mathbb{P}\left[N\left(2n\right) \leq \tfrac{3n}{2}\right] 
\leq 2n\frac{\left(2n\right)^{\left(\tfrac{3n}{2}\right)}e^{-2n}}{\left(\tfrac{3n}{2}\right)!}
\leq 2n\left(\tfrac{4n}{3}\right)^{\left(\tfrac{3n}{2}\right)}e^{-\tfrac{n}{2}}
\leq 1.07^{\left(-\tfrac{n}{2}\right).}
\end{eqnarray}
In other words, with probability exponentially close to $1$, we have that at least $3n/2$ jumps are made in 
time $2n$. To bound the probability that $W_{1}^{t}$ has not converged to the empty set, note that this probability
decreases in the number of transitions made by $W_{1}^{t}$. Therefore, without loss of generality, we may 
assume that $\frac{3n}{2}$ transitions were made, and that we start with $|W_{1}^{0}|=n$. We claim that, with 
high probability, in time $2n$, $W_{1}^{t}$ takes at most $\frac{100n}{\log n}$ increasing steps (i.e., $i \to i+1$) 
in the interval $[n^{5/6},n]$. Indeed, note that the probability of the latter occurring is less than 
\begin{equation}
2^{M}O\big(n^{-1/42}\big)^{\frac{100n}{\log n}} = O\big(e^{-n}\big).
\end{equation}
It follows that at least $\frac{n}{2}[1+o_{n}(1)]$ steps are taken in the interval $[0,n^{5/6}]$. But then, using 
\eqref{eq:prob Dt contracts}, we have that the probability of an increasing step is at most $\frac{1}{2}-\epsilon$
for some $\epsilon > 0$, and therefore the probability of that event is at most
\begin{equation}
2^{\frac{n}{2}[1+o_{n}(1)]}\left(\tfrac{1}{2}+\epsilon\right)^{\frac{n}{4}[1+o_{n}(1)]}
\left(\tfrac{1}{2}-\epsilon\right)^{\frac{n}{4}[1+o_{n}(1)]} 
= 4^{\frac{n}{4}[1+o_{n}(1)]}\left(\tfrac{1}{4}-\epsilon^{2}\right)^{\frac{n}{4}[1+o_{n}(1)]}
= \left(1-4\epsilon^{2}\right)^{\frac{n}{4}[1+o_{n}(1)]}.
\end{equation}
Finally, observing that in the entire proof so far, the largest probability for any of the bounds not to hold is 
$O(e^{n^{-2/3}})$ (see \eqref{eq:maxprob} and the paragraph following \eqref{eq:prob Dt contracts}), we
get
\begin{equation}
\mathbb{P}\left[|W_{1}^{2n}|>0\right] \leq O\big(e^{-n^{2/3}}\big)
\end{equation}
and so the claim of the lemma follows.
\end{proof}

%%%

\subsection{Long-term scheme}
\label{ss:long}

\begin{corollary}[{\bf Long-term coupling}]
\label{cor:longcoupling}
Let $\delta >0$. With $\mathbb{P}_{\ER_n(p)}$-probability tending to $1$ as $n\to\infty$, there is a coupling 
of $\{\xi_{t}\}_{t \geq 0}$ and $\{\tilde{\xi}_{\tilde{t}}\}_{t \geq 0}$, and there are times $t$ and $\tilde{t}$ with 
$\max(t,\tilde{t})<e^{n\Gamma^{\star}-\delta n}$, such that
\begin{equation}
\mathbb{P}\big[\xi_{t} \neq \tilde{\xi}_{\tilde{t}}\big] \leq e^{-n\delta+ O(n^{2/3})}.
\end{equation}  
\end{corollary}

\begin{proof}
Let $\mathbf{s}_{i}$ be the $i^{th}$ return-time of $\{\xi_t\}_{t \geq 0}$ to $A_{\mathbf{M}}$. Define 
$\{\mathbf{\tilde{s}}_{i}\}_{i\geq0}$ in an analogous manner for $\{\tilde{\xi}_t\}_{t \geq 0}$. Then we 
can define a coupling of $\{\xi_{t}\}_{t \geq 0}$ and $\{\tilde{\xi}_{t}\}_{t \geq 0}$ as follows. For $i\geq0$ 
and $0\leq s\leq2n$, couple $\xi_{\mathbf{s_{i}}+s}$ and $\tilde{\xi}_{\mathbf{\tilde{s}_{i}}+s}$ as 
described in Lemma \ref{lem:couplingscheme}. For times $t \in (\mathbf{s_{i}}+2n, \mathbf{s_{i+1}})$ 
and $\tilde{t} \in (\mathbf{\tilde{s}_{i}}+2n,\mathbf{\tilde{s}_{i+1}})$, let $\{\xi_{t}\}_{t \geq 0}$ and 
$\{\tilde{\xi}_{\tilde{t}}\}_{t \geq 0}$ run independently of each other. Terminate this coupling at the 
first time $t$ such that $t=\mathbf{s_{i}}+s$ for some $s\leq 2n$ and $\xi_{t}=\tilde{\xi}_{\tilde{t}}$ 
with $\tilde{t} =\mathbf{\tilde{s}_{i}}+s$, from which point onward we simply let $\xi_{t}=\tilde{\xi}_{\tilde{t}}$. 
It is easy to see that the coupling above is an attempt at repeating the coupling scheme of Lemma 
\ref{lem:couplingscheme} until the paths of the two processes have crossed. To avoid having 
to wait until both processes are in $A_{\mathbf{M}}$ at the same time, the coupling defines a joint 
distribution of $\xi_{t}$ and $\tilde{\xi}_{\tilde{t}}$.

Note that, by Lemma \ref{lem:ret}, with probability of at least $1-e^{-\delta n + O(n^{2/3})}$, $\{\xi_{t}\}_{t \geq 0}$  
will visit $A_{\mathbf{M}}$ at least $e^{\Gamma^{\star}n - \delta n}$ times before reaching $A_{\mathbf{S}}$, for 
any $\delta>0$. The same statement is true for $\{\tilde{\xi}_{\tilde{t}}\}_{\tilde{t} \geq 0}$. Assuming that the 
aforementioned event holds for both $\xi_{t}$ and $\tilde{\xi}_{\tilde{t}}$, the probability that the coupling does 
not succeed (i.e., the two trajectories do not intersect as described earlier) is at most
\begin{equation}
\left[O\big(e^{n^{-2/3}}\big)\right]^{e^{(\Gamma^{\star}-\delta) n}}.
\end{equation}
Therefore, the probability that the coupling does not succeed before either of $\{\xi_{t}\}_{t \geq 0}$ or 
$\{\tilde{\xi}_{\tilde{t}}\}_{\tilde{t} \geq 0}$ reaches $A_{\mathbf{S}}$ is at most $e^{-\delta n + O(n^{2/3})}$. 
\end{proof}

%%%%%%%%%%%%% SECTION 8 %%%%%%%%%%%%%%%%%%%%%

\section{Proof of the main metastability theorem}
\label{s:proofmettheorem}

In this section we prove Theorem~\ref{thm:metER}. 

\begin{proof}
The key is to show that with $\mathbb{P}_{\ER_n(p)}$-probability tending to $1$ as $n\to\infty$, 
for any $\xi_{0}^{u} \in A_{\mathbf{M}^{u}}$, $\xi_{0}\in A_{\mathbf{M}}$ and $\xi_{0}^{l}\in A_{\mathbf{M}^{l}}$,
\begin{equation}
\mathbb{E}_{\xi_{0}^{u}}\left[\tau_{\mathbf{S}^{u}}\right]
\leq \mathbb{E}_{\xi_{0}}\left[\tau_{\mathbf{S}}\right]
\leq \mathbb{E}_{\xi_{0}^{l}}\left[\tau_{\mathbf{S}^{l}}\right].
\end{equation}
Note that $\mathbb{E}_{\xi_{0}}\left[\tau_{\mathbf{S}}\right]$ is the same for all $\xi_{0} \in A_{\mathbf{M}}$ 
up to a multiplicative factor of $1+o_{n}(1)$, as was shown in Sections~\ref{ss:proofuht} and 
\ref{sec:couplingscheme}. Therefore it is suffices to find \emph{some} convenient $\xi \in A_{\mathbf{M}}$ 
for which we can prove the aforementioned theorem. 

\medskip\noindent
{\bf 1.}
Our proof follows four steps:
\begin{enumerate}
\item[(1)] 
Recall that for $A\subset S_{n}$, $\mu_{A}$ is the probability distribution $\mu$ conditioned to be in 
the set $A$. Starting from the initial distribution $\mu_{A_{\mathbf{M}}}$ on the set $A_{\mathbf{M}}$, 
the trajectory segment taken by $\xi_{t}$ from $\xi_{0}$ to $\xi_{\tau}$ with $\tau=\min\left(\tau_{\mathbf{M}},
\tau_{\mathbf{S}}\right)$, can be coupled to the analogous trajectory segments taken by $\xi_{t}^{l}$ 
and $\xi_{t}^{u}$, starting in $A_{\mathbf{M}^{l}}$ and $A_{\mathbf{M}^{u}}$, respectively, and this 
coupling can be done in such a way that the following two conditions hold: 
\begin{itemize}
\item[(a)] 
If $\xi_{t}$ reaches $A_{\mathbf{S}}$ before returning to $A_{\mathbf{M}}$ (i.e., $\tau_{\mathbf{S}}
<\tau_{\mathbf{M}}$), then $\xi_{t}^{u}$ reaches $A_{\mathbf{S}^{u}}$ before returning to $A_{\mathbf{M}^{u}}$. 
\item[(b)] 
If $\xi_{t}$ returns to $A_{\mathbf{M}}$ before reaching $A_{\mathbf{S}}$ (i.e., $\tau_{\mathbf{M}}
<\tau_{\mathbf{S}}$), then $\xi_{t}^{l}$ returns to $A_{\mathbf{M}^{l}}$ before reaching $A_{\mathbf{S}^{l}}$.
\end{itemize}
\item[(2)] 
We show that if $\xi_{t}$ has initial distribution $\mu_{A_{\mathbf{M}}}$ and $\tau_{\mathbf{M}}
<\tau_{\mathbf{S}}$, then upon returning to $A_{\mathbf{M}}$ the distribution of $\xi_{t}$ is once 
again given by $\mu_{A_{\mathbf{M}}}$. This implies that the argument in Step (1) can be applied 
repeatedly, and that the number of returns $\xi_{t}$ makes to $A_{\mathbf{M}}$ before reaching 
$A_{\mathbf{S}}$ is bounded from below by the number of returns $\xi_{t}^{u}$ makes to 
$A_{\mathbf{M}^{u}}$ before reaching $A_{\mathbf{S}^{u}}$, and is bounded from above by the 
number of returns $\xi_{t}^{l}$ makes to $A_{\mathbf{M}^{l}}$ before reaching $A_{\mathbf{S}^{l}}$. 
\item[(3)] 
Using Lemma \ref{lem:returntail}, we bound the time between unsuccessful excursions, i.e.,
the expected time it takes for $\xi_{t}$, when starting from $\mu_{A_{\mathbf{M}}}$, to return to 
$A_{\mathbf{M}}$, given that $\tau_{\mathbf{M}}<\tau_{\mathbf{S}}$. This bound is used together 
with the outcome of Step (2) to obtain the bound
\begin{equation}
\label{eq:sandreturn}
\mathbb{E}_{\mu_{A_{\mathbf{M}^{u}}}}^{u}\left[\tau_{\mathbf{S}}^{u}\right]
\leq \mathbb{E}_{\mu_{A_{\mathbf{M}}}}\left[\tau_{\mathbf{S}}\right]
\leq \mathbb{E}_{\mu_{A_{\mathbf{M}^{l}}}}^{l}\bigl[\tau_{\mathbf{S}}^{l}\bigr].
\end{equation}
Here, the fact that the conditional average return time is bounded by some large constant rather
than $1$ does not affect the sandwich in \eqref{eq:sandreturn}, because the errors coming from 
the perturbation of the magnetic field in the Curie-Weiss model are polynomial in $n$ (see below).       
\item[(4)] 
We complete the proof by showing that, for any distribution $\mu_{0}$ restricted to $A_{\mathbf{M}}$,
\begin{equation}
\mathbb{E}_{\mu_{0}}\left[\tau_{\mathbf{S}}\right] 
= \left[1+o_{n}(1)\right]\,\mathbb{E}_{\mu_{A_{\mathbf{M}}}}\left[\tau_{\mathbf{S}}\right].
\end{equation}
\end{enumerate}

\medskip\noindent
{\bf 2.}
Before we turn to the proof of these steps, we explain how the bound on the exponent in the prefactor 
of Theorem~\ref{thm:metER} comes about. Return to \eqref{eq: h perturbation}. The magnetic field 
$h$ is perturbed to $h \pm\, (1+\epsilon) \log (n^{11/6})/n$. We need to show how this affects the 
formulas for the average crossover time in the Curie-Weiss model. For this we use the computations
carried out in \cite[Chapter 13]{BdH15}. According to \cite[Eq.~(13.2.4)]{BdH15} we have, for any
$\xi \in A_{\mathbf{M}_n}$ and any $\epsilon>0$,
\begin{equation}
\label{eq:comp1}
\mathbb{E}_{\xi}\left[\tau_{A_{\mathbf{S}_n}}\right] = [1+o_n(1)]\,\frac{2}{1-\mathbf{t}}\,
e^{\beta n [R_n(\mathbf{t})-R_n(\mathbf{m})]}\,\frac{1}{n}\,\mathbf{S}_n
\end{equation}    
with
\begin{equation}
\mathbf{S}_n = \sum_{ {a,a' \in \Gamma_n} \atop {|a-\mathbf{t}}|<\epsilon,\,|a'-\mathbf{m}|<\epsilon}
e^{\beta n[R_n(\mathbf{a})-R_n(\mathbf{t})] - \beta n[R_n(a')-R_n(\mathbf{m})]},
\end{equation} 
where $R_n$ is the free energy defined by $R'_n= -J_n/2\beta$ (recall \eqref{eq:diffR=J pre}). 
(Here we suppress the dependence on $\beta,h$ and note that \eqref{eq:comp1} carries
an extra factor $\tfrac{1}{n}$ because \cite[Chapter 13]{BdH15} considers a discrete-time 
dynamics where at every unit of time a single spin is drawn uniformly at random and is flipped
with a probability that is given by the right-hand side of \eqref{eq:rate r}.) According 
to \cite[Eq.~(13.2.5)--(13.2.6)]{BdH15} we have
\begin{equation}
\label{eq:Indefext}
I_n(a) - I(a) = [1+o_n(1)] \frac{1}{2n} \log \left(\tfrac12 \pi n(1-a^2)\right), \qquad a \in [-1,1],
\end{equation}
so that 
\begin{equation}
\label{eq:comp2}
e^{\beta n[R_n(a)-R(a)]} = [1+o_n(1)]\,\sqrt{\tfrac12\pi n(1-a^2)}, \qquad a \in [-1,1].
\end{equation}
where $R$ is the limiting free energy defined by $R'= -J/2\beta$ (recall \eqref{eq:diffR=J}).
Inserting \eqref{eq:comp2} into \eqref{eq:comp1}, we get
\begin{equation}
\label{eq:comp3}
\mathbb{E}_{\xi}\left[\tau_{A_{\mathbf{S}_n}}\right] = [1+o_n(1)]\,\frac{2}{1-\mathbf{t}}\,
\sqrt{\frac{1-\mathbf{t}^2}{1-\mathbf{m}^2}}\,
e^{\beta n [R(\mathbf{t})-R(\mathbf{m})]}\,\frac{1}{n}\,\mathbf{S}^*_n
\end{equation}
with
\begin{equation}
\label{eq:comp5}
\mathbf{S}^*_n = \sum_{ {a,a' \in \Gamma_n} \atop {|a-\mathbf{t}}|<\epsilon,\,|a'-\mathbf{m}|<\epsilon}
e^{\beta n[R(\mathbf{a})-R(\mathbf{t})] - \beta n[R(a')-R(\mathbf{m})]},
\end{equation}
Finally, according to \cite[Eq.~(13.2.9)--(13.2.11)]{BdH15} we have, with the help of a Gaussian approximation,   
\begin{equation}
\label{eq:comp4}
\lim_{n\to\infty} \frac{1}{n}\,\mathbf{S}^*_n = \frac{\pi}{2\beta \sqrt{[R''(\mathbf{m})[-R''(\mathbf{t})]}}.
\end{equation}
Putting together \eqref{eq:comp3} and \eqref{eq:comp4}, we see how Theorem~\ref{thm: class CW result} 
arises as the correct formula for the Curie-Weiss model.

\medskip\noindent
{\bf 3.}
The above computations are for $\beta,h$ fixed and $p=1$. We need to investigate what changes when
$p \in (0,1)$, $\beta$ is fixed, but $h$ depends on $n$:
\begin{equation}
h_n = h \pm (1+\epsilon)\,\frac{\log (n^{11/6})}{n}.     
\end{equation}
We write $R_n^n$ to denote $R_n$ when $h$ is replaced by $h_n$. In the argument in \cite{BdH15} leading 
up to \eqref{eq:comp1}, the approximation only enters through the prefactor. But since $h_n \to h$ as $n\to\infty$, 
the perturbation affects the prefactor only by a factor $1+o_n(1)$. Since $h$ plays no role in \eqref{eq:Indefext} 
and $R^n_n(a)-R^n(a) = \frac{1}{\beta}[I_n(a)-I(a)]$ (recall \eqref{eq:potential pre} and \eqref{eq:potential}), we get \eqref{eq:comp3} with exponent $\beta n[R^n(\mathbf{t})-R^n(\mathbf{m})]$ and \eqref{eq:comp5} with exponent 
$\beta n[R^n(\mathbf{a})-R^n(\mathbf{t})] - \beta n[R^n(a')-R^n(\mathbf{m})]$. The latter affects the Gaussian 
approximation behind \eqref{eq:comp4} only by a factor $1+o_n(1)$. However, the former leads to an error term
in the exponent, compared to the Curie Weiss model, that equals  
\begin{equation}
\begin{aligned}
&\beta n[R^n(\mathbf{t})-R^n(\mathbf{m})]-\beta n[R(\mathbf{t})-R(\mathbf{m})]
= \beta n \int_{\mathbf{m}}^{\mathbf{t}} d a\, [(R^n)'(a)-R'(a)]\\
&\qquad = \beta n \int_{\mathbf{m}}^{\mathbf{t}} d a\, [-(h_n-h)]
= \beta (\mathbf{t}-\mathbf{m})\,n (h-h_n) = \mp \beta (\mathbf{t}-\mathbf{m})\,(1+\epsilon) \log(n^{11/6}).
\end{aligned}
\end{equation}
The exponential of this equals $n^{\mp \beta (\mathbf{t}-\mathbf{m})\,(1+\epsilon)(11/6)}$, which proves 
Theorem~\ref{thm:metER} with the bound in \eqref{eq:expbd} because $\epsilon$ is arbitrary.

\medskip
\paragraph{\bf Proof of Step (1):} 
This step is a direct application of Lemma \ref{lem: Hitting order}.

\medskip
\paragraph{\bf Proof of Step (2):} 
Let $\xi_{0}\overset{d}{=}\mu_{A_{\mathbf{M}}}$, and recall that $\tau_{\mathbf{M}}$ is the first 
return time of $\xi_t$ to $A_{\mathbf{M}}$ once the initial state $\xi_{0}$ has been left. We want 
to show that $\xi_{\tau_{\mathbf{M}}}\overset{d}{=}\mu_{A_{\mathbf{M}}}$ or, in other words, that 
$\mathbb{P}_{\mu_{A_{\mathbf{M}}}}[\xi_{\tau_{\mathbf{M}}}=\sigma] = \mu_{A_{\mathbf{M}}}
(\sigma)$ for any $\sigma\in A_{\mathbf{M}}$. To facilitate the argument, we begin by 
defining the set of all finite permissible trajectories $\mathscr{T}$, i.e., 
\begin{equation}
\mathscr{T}=\bigcup_{N\in\mathbb{N}}\left\{ \gamma=\left\{ \gamma_{i}\right\} _{i=0}^{N}
\in S_{n}^{N}\colon\left|\left|\gamma_{i}\right|-\left|\gamma_{i+1}\right|\right|=1\,\,
\forall\,0 \leq i \leq N-1\right\}.
\end{equation}
Let $\gamma\in\mathscr{T}$ be any finite trajectory beginning at $\gamma_{0}\in A_{\mathbf{M}}$, 
ending at $\gamma_{|\gamma|-1}=\sigma\in A_{\mathbf{M}}$, and satisfying $\gamma_{i}
\notin A_{\mathbf{M}}$ for $0<i<|\gamma|-1$. Then the probability that $\xi_{t}$ follows 
the trajectory $\gamma$ is given by 
\begin{equation}
\begin{aligned}
\mathbb{P}\left[\xi_{t}\mbox{\text{ follows }}\gamma\right] 
&= \mu_{A_{\mathbf{M}}}\left(\gamma_{0}\right)P\left(\gamma_{0},\gamma_{1}\right)
\times \cdots \times P\left(\gamma_{\left|\gamma\right|-2},\sigma\right)\\
&= \frac{1}{\mu\left(A_{\mathbf{M}}\right)}\mu\left(\gamma_{0}\right)P\left(\gamma_{0},\gamma_{1}\right)
\times \cdots \times P\left(\gamma_{\left|\gamma\right|-2},\sigma\right)\\
&= \frac{1}{\mu\left(A_{\mathbf{M}}\right)}\mu\left(\sigma\right)P\left(\sigma,\gamma_{\left|\gamma\right|-2}\right)
\times \cdots \times P\left(\gamma_{1},\gamma_{0}\right)\\
&= \mu_{A_{\mathbf{M}}}\left(\sigma\right)P\left(\sigma,\gamma_{\left|\gamma\right|-2}\right)
\times \cdots \times P\left(\gamma_{1},\gamma_{0}\right),
\end{aligned}
\end{equation}
where the third line follows from reversibility. Thus, if we let $\mathscr{T}(\sigma)$ be the set of all 
trajectories in $\mathscr{T}$ that begin in $A_{\mathbf{M}}$, end at $\sigma$, and do not visit $A_{\mathbf{M}}$ 
in between, then we get 
\begin{equation}
\begin{aligned}
\mathbb{P}_{\mu_{A_{\mathbf{M}}}}\left[\xi_{\tau_{\mathbf{M}}}=\sigma\right] 
&= \sum_{\gamma\in\mathscr{T}\left(\sigma\right)}\mu_{A_{\mathbf{M}}}\left(\sigma\right)
P\left(\sigma,\gamma_{\left|\gamma\right|-2}\right) \times \cdots \times P\left(\gamma_{1},\gamma_{0}\right)\\
&= \mu_{A_{\mathbf{M}}}\left(\sigma\right)\mathbb{P}_{\sigma}\left[\tau_{\mathbf{M}}<\infty\right]\\
&= \mu_{A_{\mathbf{M}}}\left(\sigma\right),
\end{aligned}
\end{equation}
where we use recurrence and the law of total probability, since the trajectories in $\mathscr{T}(\sigma)$, 
when reversed, give all possible trajectories that start at $\sigma\in A_{\mathbf{M}}$ and return to $A_{\mathbf{M}}$ 
in a finite number of steps. This shows that if $\xi_{t}$ has initial distribution $\mu_{A_{\mathbf{M}}}$, then it also 
has this distribution upon every return to $A_{\mathbf{M}}$. 

We can now define a segment-wise coupling of the trajectory taken by $\xi_{t}$ with the trajectories taken by 
$\xi_{t}^{u}$ and $\xi_{t}^{l}$. First, we define the subsets of trajectories that start and end in particular regions 
of the state space: (i) $\mathscr{T}_{\sigma,L,K}$ is the set of trajectories that start at a particular configuration 
$\sigma$ and end in $A_{K}$ without ever visiting $A_{K}$ or $A_{L}$ in between, for some $K<L$; (ii)
$\mathscr{T}_{\sigma,L,L}$ is the set of trajectories that start at some $\sigma$ and end in $A_{L}$ without
ever visiting $A_{K}$ or $A_{L}$ in between; (iii) $\mathscr{T}_{\sigma,L}$ is the union of the two aforementioned 
sets. In explicit form, 
\begin{equation}
\begin{aligned}
\mathscr{T}_{\sigma,L,K} 
&= \left\{ \gamma\in\mathscr{T}\colon\gamma_{0}=\sigma,\gamma_{\left|\gamma\right|-1}
\in A_{K},K<\left|\gamma_{j}\right|<L\,\,\forall\,0<j<\left|\gamma\right|-1\right\},\\[0.2cm]
\mathscr{T}_{\sigma,L,L} 
&= \left\{ \gamma\in\mathscr{T}\colon\gamma_{0}=\sigma,\gamma_{\left|\gamma\right|-1}
\in A_{L},K<\left|\gamma_{j}\right|<L\,\,\forall\,0<j<\left|\gamma\right|-1\right\},\\[0.2cm]
\mathscr{T}_{\sigma,L} 
& = \mathscr{T}_{\sigma,L,K}\cup\mathscr{T}_{\sigma,L,L}.
\end{aligned}
\end{equation}
By Step (1), for any $\xi_{0}^{l}\in A_{\mathbf{M}^{l}}$ and $\xi_{0}^{u}\in A_{\mathbf{M}^{u}}$,
\begin{equation}
\mathbb{P}_{\xi_{0}^{l}}^{l}\bigl[\mathscr{T}_{\xi_{0}^{l},\mathbf{S}^{l},\mathbf{S}^{l}}\bigr]
\leq\mathbb{P}_{\xi_{0}}\left[\mathscr{T}_{\xi_{0},\mathbf{S},\mathbf{S}}\right]
\leq \mathbb{P}_{\xi_{0}^{u}}^{u}\left[\mathscr{T}_{\xi_{0}^{u},\mathbf{S}^{u},\mathbf{S}^{u}}\right].
\label{eq:comp3probs}
\end{equation}
It is clear that the two probabilities at either end of \eqref{eq:comp3probs} are independent of the 
starting points $\xi_{0}^{l}$ and $\xi_{0}^{u}$. By the argument given above, if for the probability in the 
middle $\xi_{0}\overset{d}{=}\mu_{A_{\mathbf{M}}}$, then each subsequent return to $A_{\mathbf{M}}$ 
also has this distribution. For this reason, we may define a coupling of the trajectories as follows. 

Sample a trajectory segment $\gamma^{l}$ from $\mathscr{T}_{\xi_{0}^{l},\mathbf{S}^{l}}$ for the 
process $\xi_{t}^{l}$. If $\gamma^{l}$ happens to be in $\mathscr{T}_{\xi_{0}^{l},\mathbf{S}^{l},
\mathbf{S}^{l}}$, then by \eqref{eq:comp3probs} we may sample a trajectory segment $\gamma$
from $\mathscr{T}_{\xi_{0},\mathbf{S},\mathbf{S}}$ for the process $\xi_{t}$, and a trajectory segment 
$\gamma^{u}$ from $\mathscr{T}_{\xi_{0}^{u},\mathbf{S}^{u},\mathbf{S}^{u}}$ for the process $\xi^{u}$. 
Otherwise, $\gamma^{l}\in\mathscr{T}_{\xi_{0}^{l},\mathbf{S}^{l},\mathbf{M}^{l}}$, and we independently 
take $\gamma\in\mathscr{T}_{\xi_{0},\mathbf{S},\mathbf{S}}$ with probability $\mathbb{P}_{\xi_{0}}
[\mathscr{T}_{\xi_{0},\mathbf{S},\mathbf{S}}]-\mathbb{P}_{\xi_{0}^{l}}^{l}[\mathscr{T}_{\xi_{0}^{l},
\mathbf{S}^{l},\mathbf{S}^{l}}]$, and $\gamma\in\mathscr{T}_{\xi_{0},\mathbf{S},\mathbf{M}}$ 
otherwise. If $\gamma\in\mathscr{T}_{\xi_{0},\mathbf{S},\mathbf{S}}$, then sample $\gamma^{u}$ 
from $\mathscr{T}_{\xi_{0}^{u},\mathbf{S}^{u},\mathbf{S}^{u}}$. Otherwise $\gamma\in\mathscr{T}_{\xi_{0},
\mathbf{S},\mathbf{M}}$, and so take independently $\gamma^{u}\in\mathscr{T}_{\xi_{0}^{u},
\mathbf{S}^{u},\mathbf{S}^{u}}$ with probability $\mathbb{P}_{\xi_{0}^{u}}^{u}[\mathscr{T}_{\xi_{0}^{u},
\mathbf{S}^{u},\mathbf{S}^{u}}]-\mathbb{P}_{\xi_{0}}[\mathscr{T}_{\xi_{0},\mathbf{S},\mathbf{S}}]$,
and $\gamma^{u}\in\mathscr{T}_{\xi_{0}^{u},\mathbf{S}^{u},\mathbf{M}^{u}}$ with the remaining probability. 
We glue together the sampling of segments leaving and returning to $A_{\mathbf{M}^{l}}/A_{\mathbf{M}}
/A_{\mathbf{M}^{u}}$ with the next sampling of such segments. This results in trajectories for $\xi^{u}$, 
$\xi$, and $\xi^{l}$ that reach $A_{\mathbf{S}^{u}}/A_{\mathbf{S}}/A_{\mathbf{S}^{l}}$, in that particular 
order.

\medskip
\paragraph{\bf Proof of Step (3) and Step (4):} 
These two steps are immediate from Lemma \ref{lem:returntail}.

\end{proof}

%%%%%%%% APPENDIX %%%%%%%%%%%%%%%%%%%%%%%

\appendix

%%%%%%%%%%%%%%%%%%%%%%%%%%%%%%%%%%%%%%

\section{Conditional average return time for inhomogeneous random walk}
\label{app}

In this appendix we prove Lemma \ref{lem:returntail}. In Appendices \ref{app1}--\ref{app2} we
compute the harmonic function and the conditional average return time for an arbitrary 
nearest-neighbour random walk on a finite interval. In Appendix \ref{app3} we use these
computations to prove the lemma. 
 
%%%

\subsection{Harmonic function}
\label{app1}

Consider a \emph{nearest-neighbour} random walk on the set $\{0,\ldots,N\}$ with strictly positive 
transition probabilities $p(x,x+1)$ and $p(x,x-1)$, $0<x<N$, and with $0$ and $N$ acting as 
absorbing boundaries. Let $\tau_0$ and $\tau_N$ denote the first hitting times of $0$ 
and $N$. The \emph{harmonic function} is defined as
\begin{equation}
h_N(x) = \mathbb{P}_x(\tau_N<\tau_0), \quad 0 \leq x \leq N,  
\end{equation}
where $\mathbb{P}_x$ is the law of the random walk starting from $x$. This is the unique 
solution of the recursion relation 
\begin{equation}
h_N(x) = p(x,x+1) h_N(x+1) + p(x,x-1) h_N(x-1), \quad 0<x<N,
\end{equation}
with boundary conditions
\begin{equation}
h_N(0)= 0, \qquad h_N(N) = 1. 
\end{equation}
Since $p(x,x+1)+p(x,x-1)=1$, the recursion can be written as
\begin{equation}
p(x,x+1)[h_N(x+1)-h_N(x)] = p(x,x-1)[h_N(x)-h_N(x-1)].
\end{equation}
Define $\Delta h_N(x) = h_N(x+1)-h_N(x)$, $0 \leq x < N$. Iteration gives
\begin{equation}
\Delta h_N(x) = \pi[1,x]\,\Delta h_N(0), \quad 0 \leq x < N,
\end{equation}
where we define
\begin{equation}
\label{eq:pidef}
\pi(I) = \prod_{z \in I} \frac{p(z,z-1)}{p(z,z+1)}, \quad I \subseteq \{1,\ldots,N-1\},
\end{equation}
with the convention that the empty product equals $1$. Since $h_N(0)=0$, we have
\begin{equation}
h_N(x) = \sum_{z=0}^{x-1} \Delta h_N(z) = \left(\sum_{z=0}^{x-1} \pi[1,z]\right) 
\Delta h_N(0), \qquad 0 < x \leq N.
\end{equation}
Put $C=\Delta h_N(0)$, and abbreviate
\begin{equation}
R(x) = \sum_{z=0}^{x-1} \pi[1,z], \quad 0 \leq x \leq N.
\end{equation}
Since $h_N(N)=1$, we have $C=1/R(N)$. Therefore we arrive at
\begin{equation}
h_N(x) = \frac{R(x)}{R(N)}, \quad 0 \leq x \leq N.
\end{equation}

\begin{remark}
{\rm For simple random walk we have $p(x,x \pm 1) = \tfrac12$, hence $\pi[1,x]=1$ and 
$R(x)=x$, and so
\begin{equation}
h_N(x) = \frac{x}{N}, \quad 0 \leq x \leq N,
\end{equation} 
which is the standard gambler's ruin formula.}
\end{remark}

%%%

\subsection{Conditional average hitting time} 
\label{app2}

We are interested in the quantity
\begin{equation}
e_N(x) = \mathbb{E}_x(\tau_N \mid \tau_N<\tau_0), \qquad 0 < x \leq N.
\end{equation} 
The conditioning amounts to taking the \emph{Doob transformed} random walk, which has 
transition probabilities
\begin{equation}
q(x,x \pm 1) =  p(x,x \pm1 )\,\frac{h_N(x \pm 1)}{h_N(x)}. 
\end{equation}
We have the recursion relation
\begin{equation}
e_N(x) = 1 + q(x,x+1) e_N(x+1) + q(x,x-1) e_N(x-1), \quad 0 < x < N, 
\end{equation}
with boundary conditions
\begin{equation}
e_N(N) = 0, \quad e_N(1) = 1 + e_N(2).
\end{equation}
Putting $f_N(x) = h_N(x)e_N(x)$, we get the recursion
\begin{equation}
f_N(x) = h_N(x) + p(x,x+1) f_N(x+1) + p(x,x-1) f_N(x-1),
\quad 0 < x < N,
\end{equation}  
which can be rewritten as 
\begin{equation}
p(x,x+1)[f_N(x+1)-f_N(x)] = p(x,x-1)[f_N(x)-f_N(x-1)] - h_N(x).
\end{equation}
Define $\Delta f_N(x) = f_N(x+1)-f_N(x)$, $0<x<N$. Iteration gives
\begin{equation}
\Delta f_N(x) = \pi(1,x]\,\Delta f_N(1) - \chi(1,x], \quad 0 < x < N,
\end{equation}
with
\begin{equation}
\chi(1,x] = \sum_{y=2}^x \pi(y,x]\,\frac{h_N(y)}{p(y,y+1)}, \quad 0 < x < N, 
\end{equation}
Since $f_N(N)=0$, we have
\begin{equation}
f_N(x) = - \sum_{z=x}^{N-1} \Delta f_N(z) =
\sum_{z=x}^{N-1} \chi(1,z] - \left(\sum_{z=x}^{N-1} \pi(1,z]\right) \Delta f_N(1), 
\quad 0<x<N,
\end{equation}
or
\begin{equation}
e_N(x) = \frac{1}{h_N(x)} \sum_{z=x}^{N-1} \chi(1,z] 
- \frac{1}{h_N(x)} \left(\sum_{z=x}^{N-1} \pi(1,z]\right) \Delta f_N(1),
\quad 0 < x < N.
\end{equation}
Put $C=\Delta f_N(1)$, and abbreviate
\begin{equation}
A(x) = \sum_{z=x}^{N-1} \pi(1,z], \quad B(x) = \sum_{z=x}^{N-1} \chi(1,z],
\quad 0 < x  \leq N.
\end{equation}
Then
\begin{equation}
e_N(x) = \frac{1}{h_N(x)} \bigl[B(x)-C A(x)\bigr].
\end{equation}
Since $e_N(1) = 1 + e_N(2)$, we have
\begin{equation}
C = \frac{[h_N(2)B(1)-h_N(1)B(2)]-h_N(1)h_N(2)}{h_N(2)A(1)-h_N(1)A(2)}.
\end{equation}
Abbreviate
\begin{equation}
\label{eq:form0}
\bar{R}(x) = \sum_{z=0}^{x-1} \pi(1,z], \quad \bar{S}(x) = \sum_{z=0}^{x-1} \chi(1,z],
\quad 0 < x  \leq N.
\end{equation}
Then
\begin{equation}
A(x) = \bar{R}(N)-\bar{R}(x), \quad B(x) = \bar{S}(N)-\bar{S}(x), \quad 0< x < N. 
\end{equation}
Note that $h_N(x) = R(x)/R(N) = \bar{R}(x)/\bar{R}(N)$, because $\pi[1,z] = \pi(1)\pi(1,z]$
and a common factor $\pi(1)$ drops out.  Note further that $\bar{R}(1) = 1$, $\bar{R}(2) = 2$, 
while $\bar{S}(1)=\bar{S}(2)=0$. Therefore
\begin{equation}
C = \frac{\bar{S}(N)}{\bar{R}(N)} - \frac{2}{\bar{R}(N)^2}. 
\end{equation}
Therefore we arrive at
\begin{equation}
\label{eq:SR}
e_N(x) = \bar{S}(N) - \frac{\bar{R}(N)}{\bar{R}(x)}\,\bar{S}(x) + \frac{2}{\bar{R}(x)} - \frac{2}{\bar{R}(N)},
\quad 0 < x \leq N.
\end{equation}
Abbreviating
\begin{equation}
\label{eq:form1}
\bar{T}(x) = \bar{S}(x)\bar{R}(N) = \sum_{z=0}^{x-1} \sum_{y=2}^z \frac{\pi(y,z]}{p(y,y+1)}\,\bar{R}(y),
\quad \bar{U}(x) = \frac{\bar{T}(x)-2}{\bar{R}(x)}, 
\end{equation} 
we can write
\begin{equation}
\label{eq:form2}
e_N(x) = \bar{U}(N)- \bar{U}(x), \quad 0 < x \leq N.
\end{equation}

\begin{remark}
{\rm For simple random walk we have $p(x,x\pm 1) = \frac12$, $\pi(y,z]=1$, $\bar{R}(x)=x$, $\bar{S}(x) 
= \frac{1}{3N}(x^3-7x+6)$ and $\bar{U}(x) = \tfrac13(x^2-7)$, and so
\begin{equation}
e_N(x) = \tfrac13\,(N^2-x^2), \qquad 0 < x \leq N. 
\end{equation}
This is to be compared with the \emph{unconditional} average hitting time $\mathbb{E}_x(\tau) = x(N-x)$, 
$0 \leq x \leq N$, where $\tau = \tau_0 \wedge \tau_N$ is the first hitting time of $\{0,N\}$.} 
\end{remark}

%%%

\subsection{Application to spin-flip dynamics} 
\label{app3}

We will use the formulas in \eqref{eq:pidef}, \eqref{eq:form0} and \eqref{eq:form1}--\eqref{eq:form2} 
to obtain an upper bound on the conditional return time to the metastable state. This bound will
be sharp enough to prove Lemma \ref{lem:returntail}. We first do the computation for the complete 
graph (Curie-Weiss model). Afterwards we turn to the Erd\H{o}s-R\'enyi Random Graph (our 
spin-flip dynamics).     

%%%

\subsubsection{Complete graph}

We monitor the magnetization of the \emph{continuous-time} Curie-Weiss model by looking at the 
magnetization at the times of the spin-flips. This gives a \emph{discrete-time} random walk on 
the set $\Gamma_n$ defined in \eqref{eq:InJndef}. This set consists of $n+1$ sites. We first consider the 
excursions to the \emph{left} of $\mathbf{m}_n$ (recall \eqref{eq: mts}). After that we consider
the excursions to the \emph{right}.

\medskip\noindent
{\bf 1.} 
For the Curie-Weiss model we have (use the formulas in Lemma \ref{lem: Rate bounds} without
the error terms)
\begin{equation}
\label{eq:CWrates}
\sigma \in A_k\colon \qquad \sum_{\xi \in A_{k+1}} r(\sigma,\xi) = (n-k)\,e^{-2\beta [\theta_k]_+},
\quad \sum_{\xi \in A_{k-1}} r(\sigma,\xi) = k\, e^{-2\beta [-\theta_k]_+},
\end{equation}
where $\theta_k = p(1-\frac{2k}{n})-h$. Hence, the quotient of the rate to move downwards, 
respectively, upwards in magnetization equals   
\begin{equation}
\label{eq:ratiors}
Q(k) = \frac{\sum_{\xi \in A_{k-1}} r(\sigma,\xi) }{\sum_{\xi \in A_{k+1}} r(\sigma,\xi)}
= \frac{k}{n-k}\,e^{2\beta ([\theta_k]_+ - [-\theta_k]_+)}.  
\end{equation}
It is convenient to change variables by writing $k = \frac{n}{2}(a_k+1)$, so that $\theta_k = -pa_k-h$. 
The metastable state corresponds to $k = \mathbf{M}_n = \frac{n}{2}(\mathbf{m}_n+1)$, i.e., 
$a_k=\mathbf{m}_n$. We know from \eqref{eq: mts}--\eqref{eq:InJndef} that $\mathbf{m}_n$ is the 
smallest solution of the equation $J_n(\mathbf{m}_n)=0$ (rounded off by $1/n$ to fall in $\Gamma_n$). 
Hence $\mathbf{m}_n = \mathbf{m} + O(1/n)$ with $\mathbf{m}$ the smallest solution of the 
equation $J_{p,\beta,h}(\mathbf{m})=0$, satisfying $\frac{1-\mathbf{m}}{1+\mathbf{m}} =
e^{-2\beta(p\mathbf{m}+h)}$ (recall \eqref{Jadef}). Hence we can write (for ease of notation we 
henceforth ignore the error $O(1/n)$) 
\begin{equation}
Q(k) = \frac{F(\mathbf{m}_n)}{F(a_k)}, \qquad F(a) = \frac{1-a}{1+a}\,e^{2\beta p a}. 
\end{equation}    
Here, we use that $[\theta_k]_+-[-\theta_k]_+ = \theta_k$, which holds because $0 = R'_{p,\beta,h}
(\mathbf{m}) = -p\mathbf{m}-h+\beta^{-1}I'(\mathbf{m})$ with $I'(\mathbf{m})<0$ because $\mathbf{m}
<0$ (recall \eqref{eq:diffR=J}), so that $-p\mathbf{m}_n-h>0$ for $n$ large enough, which implies that 
also $-pa-h>0$ for all $a<\mathbf{m}_n$ for $n$ large enough. We next note that (recall \eqref{eq:diffR=J} 
and \eqref{eq:Jprime})
\begin{equation}
\label{eq:ratioslope}
\frac{d}{da} \log \left[\frac{F(\mathbf{m}_n)}{F(a)}\right] = -2\left(\beta p-\frac{1}{1-a^2}\right) =
- J_{p,\beta,h}'(a) = 2\beta R''_{p,\beta,h}(a) \geq \delta \quad  \text{for some } \delta>0, 
\end{equation}
where the inequality comes from the fact that $a \mapsto R_{p,\beta,h}(a)$ has a positive curvature that 
is bounded away from zero on $[-1,\mathbf{m}]$ (recall Figure \ref{fig:CWfe}).   

\medskip\noindent
{\bf 2.}
We view the excursions to the left of $\mathbf{m}_n$ as starting from site $N$ in the set $\{0,\ldots,N\}$
with $N = \mathbf{M}_n = \frac{n}{2}(\mathbf{m}_n+1)$. From \eqref{eq:form1}--\eqref{eq:form2}, we get
\begin{equation}
\label{eq:form5}
\begin{aligned}
e_N(x) 
&= \sum_{z=0}^{N-1} \sum_{y=2}^z \frac{\pi(y,z]}{p(y,y+1)}\,\frac{\bar{R}(y)}{\bar{R}(N)}
- \sum_{z=0}^{x-1} \sum_{y=2}^z \frac{\pi(y,z]}{p(y,y+1)}\,\frac{\bar{R}(y)}{\bar{R}(x)}
+ \frac{2}{\bar{R}(N)\bar{R}(x)} [\bar{R}(N)-\bar{R}(x)]\\
&\leq \sum_{z=x}^{N-1} \sum_{y=2}^z \frac{\pi(y,z]}{p(y,y+1)}\,\frac{\bar{R}(y)}{\bar{R}(N)}
+ \frac{2}{\bar{R}(x)}\\
&\leq 2 \sum_{z=1}^{N-1} \sum_{y=2}^z \pi(y,z] + 2.
\end{aligned}
\end{equation}
Here, we use that $p(y,y+1) \geq \tfrac12$ and $1 = \bar{R}(0) \leq \bar{R}(y) \leq \bar{R}(N)$ for 
all $0<y<N$ (recall \eqref{eq:form0} and note that $x \mapsto \bar{R}(x)$ is non-decreasing). The 
bound is independent of $x$. Using the estimate
\begin{equation}
\label{eq:ratioslopealt}
Q(x) = \frac{p(x,x-1)}{p(x,x+1)} \leq e^{-\epsilon(N-x)/N}, \quad 0 < x < N, \quad \text{ for some } 
\epsilon = \epsilon(\delta)>0,
\end{equation}
which comes from \eqref{eq:ratioslope}, we can estimate
\begin{equation}
\pi(y,z] \leq \prod_{x=y+1}^z e^{-\epsilon (N-x)/N}
= \exp\left[ -\epsilon \sum_{x=y+1}^z (N-x)/N \right], \qquad 0 \leq y \leq z <N,
\end{equation} 
from which it follows that
\begin{equation}
\label{eq:sumpi}
\sum_{z=1}^{N-1} \sum_{y=2}^z \pi(y,z]  = O(N/\epsilon), \qquad N\to\infty.
\end{equation}
Thus we arrive at 
\begin{equation}
\label{eq:mean}
e_N(x) = O(N), \qquad N \to \infty, \quad \text{ uniformly in } 0<x<N.
\end{equation}
To turn \eqref{eq:mean} into a tail estimate, we use the Chebyshev inequality: \eqref{eq:mean} 
implies that every $N$ time units there is a probability at least $c$ to hit $N$, for some $c>0$ 
and uniformly in $0 < x < N$. Hence
\begin{equation}
\label{eq:exptail}
\mathbb{P}_x(\tau_N \geq kN \mid \tau_N < \tau_0) \leq (1-c)^k \qquad \forall\, k \in \N_0.
\end{equation}

\medskip\noindent
{\bf 3.} 
For excursions to the right of $\mathbf{m}_n$ the argument is similar. Now $N = \mathbf{T}_n 
- \mathbf{M}_n = \frac{n}{2}(\mathbf{t}_n-\mathbf{m}_n)$ (recall \eqref{eq:MTSdef}), and the 
role of $0$ and $N$ is interchanged. Both near $0$ and near $N$ the drift towards $\mathbf{M}_n$ 
vanishes \emph{linearly} (because of the non-zero curvature). If we condition the random walk 
not to hit $N$, then the average hitting time of $0$ starting from $x$ is again $O(N)$, uniformly 
in $x$.   

\medskip\noindent
{\bf 4.}
Returning from the discrete-time random walk to the continuous-time Curie-Weiss model, we note 
that order $n$ spin-flips occur per unit of time. Since $N \asymp n$ as $n\to\infty$, \eqref{eq:exptail} 
and its analogue for excursions to the right give that, uniformly in $\xi \in A_{\mathbf{M}_n}$,
\begin{equation}
\label{eq:exptailcont}
\mathbb{P}_\xi\left[\tau_{A_{\mathbf{M}_n}} \geq k \mid \tau_{A_{\mathbf{M}_n}} 
< \tau_{A_{\mathbf{T}_n}} \right] \leq e^{-Ck} \qquad \forall\, k \in \N_0.
\end{equation}
for some $C>0$, which is the bound in \eqref{eq:exptailER}.

%%%

\subsubsection{Erd\H{o}s-R\'enyi Random Graph}

We next argue that the above argument can be extended to our spin-flip dynamics after taking into 
account that the rates to move downwards and upwards in magnetization are \emph{perturbed by 
small errors}. In what follows we will write $p^{\mathrm{CW}}(x,x \pm 1)$ for the transition probabilities 
in the Curie-Weiss model and $p^{\mathrm{ER}}(x,x \pm 1)$ for the transition probabilities that serve
as \emph{uniform upper and lower bounds} for the transition probabilities in our spin-flip model. 
Recall that the latter actually depend on the configuration and not just on the magnetization, but 
Lemma \ref{lem: Rate bounds} provides us with uniform bounds that allow us to \emph{sandwich} 
the magnetization between the magnetizations of \emph{two perturbed Curie-Weiss models}.      

\medskip\noindent
{\bf 1.}
Suppose that 
\begin{equation}
\label{eq:form6}
\frac{p^{\mathrm{ER}}(x,x-1)}{p^{\mathrm{ER}}(x,x+1)}
= \frac{p^{\mathrm{CW}}(x,x-1)}{p^{\mathrm{CW}}(x,x+1)} \big[1+O(N^{-1/2})].
\end{equation}
Then there exists a $C>0$ large enough such that 
\begin{equation}
\label{eq:form7}
\pi^{\mathrm{ER}}(y,z] \leq C \pi^{\mathrm{CW}}(y,z], \qquad 0 \leq y \leq z <N. 
\end{equation}
Indeed, as long as $z-y \leq C_1 N^{1/2}$ we have the bound in \eqref{eq:form7} (with $C$ depending
on $C_1$). On the other hand, if $z-y > C_1 N^{1/2}$ with $C_1$ large enough, then \emph{the drift of 
the Curie-Weiss model sets in and overrules the error}: recall from \eqref{eq:ratioslopealt} that the drift 
at distance $N^{1/2}$ from $N$ is of order $N^{1/2}/N = N^{-1/2}$. It follows from \eqref{eq:form7} 
that \eqref{eq:sumpi}--\eqref{eq:exptail} carry over, with suitably adapted constants, and hence so
does \eqref{eq:exptailcont}.

\medskip\noindent
{\bf 2.}
To prove \eqref{eq:form6}, we must show that \eqref{eq:ratiors} holds up to a multiplicative error 
$1+O(n^{-1/2})$. In the argument that follows we assume that $k$ is such that $\theta_k \geq 
\delta$ for some fixed $\delta>0$. We comment later on how to extend the argument to other
$k$ values. Recall that $\theta_k = -pa_k-h$ and that $\theta_k \geq \delta>0$ for all $a_k \in 
[-1,\mathbf{m}]$ for $n$ large enough.

\medskip\noindent
{\bf 3.}
Let $\sigma \in A_k$ and $\sigma^{v}\in A_{k-1}$, where $\sigma^v$ is obtained from 
$\sigma$ by flipping the sign at vertex $v \in \sigma$ from $+1$ to $-1$. Write the transition 
rate from $\sigma$ to $\sigma^v$ as
\begin{equation}
\begin{aligned}
r(\sigma,\sigma^v) &= \exp\left( -\beta \left[ 2p(\tfrac{2k}{n}-1)+2h
+ \tfrac{2}{n}\big(\epsilon(\sigma,v)-\epsilon(\overline{\sigma},v)\big)\right]_{+}\right)\\
&= \exp\left( -2\beta \left[ -\theta_k
+ \tfrac{1}{n}\big(\epsilon(\sigma,v)-\epsilon(\overline{\sigma},v)\big)\right]_{+}\right).
\end{aligned}
\end{equation}
Here, $2p(\tfrac{2k}{n}-1)=\frac{2}{n}p[k-(n-k)]$ equals $\frac{2}{n}$ times the average under 
$\mathbb{P}_{\ER_n(p)}$ of $E(\sigma,v) - E(\overline{\sigma},v)$, with $E(\sigma,v)$ the 
number of edges between the support of $\sigma$ and $v$ and $E(\overline{\sigma},v)$ the
number of edges between the support of $\overline{\sigma}$ and $v$ (recall the notation in 
Definition \ref{def:notsets}), and $\epsilon(\sigma,v)-\epsilon(\overline{\sigma},v)$ is an 
\emph{error term} that arises from deviations of this average. Since $-\theta_k \leq -\delta$, 
the error terms are \emph{not seen except} when they represent a large deviation of size at least 
$\delta n$. A union bound over all the vertices and all the configurations, in combination with 
Hoeffding's inequality, guarantees that, with $\mathbb{P}_{\ER_n(p)}$-probability tending 
to $1$ as $n\to\infty$, for any $\sigma$ there are at most $(\log 2)/2\delta^2 = O(1)$ many 
vertices that can lead to a large deviation of size at least $\delta n$. Since $r(\sigma,\sigma^v) 
\leq 1$, we obtain
\begin{equation}
\sum_{v\in\sigma} r(\sigma,\sigma^v) =  O(1) + [n-k-O(1)]\,e^{-2\beta [-\theta_k]_+}. 
\end{equation} 
This is a refinement of \eqref{eq:frate}.

\medskip\noindent
{\bf 4.}
Similarly, let $\sigma \in A_k$ and $\sigma^v \in A_{k+1}$, where $\sigma^v$ is obtained from 
$\sigma$ by flipping the sign at vertex $v \notin \sigma$ from $-1$ to $+1$. Write the transition 
rate from $\sigma$ to $\sigma^v$ as
\begin{equation}
\begin{aligned}
r(\sigma,\sigma^v) 
&= \exp\left( -\beta \left[2p(1-\tfrac{2k}{n}) - 2h 
+ \tfrac{2}{n}\big(\epsilon(\overline{\sigma},v)-\epsilon(\sigma,v)\big)\right]_{+}\right)\\
&= \exp\left( -2\beta \left[\theta_k 
+ \tfrac{1}{n}\big(\epsilon(\overline{\sigma},v)-\epsilon(\sigma,v)\big)\right]_{+}\right).
\end{aligned}
\end{equation}
We \emph{cannot remove} $[\cdot]_{+}$ when the error terms represent a large deviation of order 
$\delta n$. By the same argument as above, this happens for all but $(\log 2)/2\delta^2 = O(1)$ 
many vertices $v$. For all other vertices, we \emph{can remove} $[\cdot]_{+}$ and write 
\begin{equation}
\label{eq:rsplit}
r(\sigma,\sigma^v) = e^{-2\beta \theta_k} 
\,\exp\left( \tfrac{1}{n} \big(\epsilon(\overline{\sigma},v)-\epsilon(\sigma,v)\big) \right).
\end{equation}
Next, we sum over $v$ and use the inequality, valid for $\delta$ small enough, 
\begin{equation}
e^{-(1+\delta) \tfrac{1}{M} |\sum_{i=1}^M a_i|} \leq \tfrac1M \sum_{i=1}^M e^{a_i} 
\leq e^{(1+\delta) \tfrac{1}{M} |\sum_{i=1}^M a_i|} 
\qquad \forall\,0 \leq |a_i| \leq \delta,\,\, 1 \leq i \leq M.
\end{equation}
This gives
\begin{equation}
\sum_{v \notin \sigma} r(\sigma,\sigma^v) = O(1) + [k-O(1)]\,e^{-2\beta \theta_k}\,e^{O(|S_n|)},
\qquad S_n = \frac{1}{[k-O(1)]} \frac{1}{n} \sum_{v \notin \sigma} 
\big(\epsilon(\overline{\sigma},v)-\epsilon(\sigma,v)\big).   
\end{equation}
We know from Lemma \ref{lem: degree bound} that, with $\mathbb{P}_{\ER_n(p)}$-probability tending 
to $1$ as $n\to\infty$,
\begin{equation}
|S_n| \leq \frac{c n^{3/2}}{[k-O(1)]n} \qquad \forall\,c > \sqrt{\tfrac{1}{8}\log 2}.
\end{equation}
Since we may take $k \geq \frac{n}{3}(p-h)$ (recall \eqref{eq:frate at end alt}), we obtain
\begin{equation}
\sum_{v \notin \sigma} r(\sigma,\sigma^v) = O(1) + [k-O(1)]\,e^{-2\beta \theta_k}\,e^{O(n^{-1/2})}.
\end{equation}
This is a refinement of \eqref{eq: brate}.

\medskip\noindent
{\bf 5.}
The same argument works when we assume that $k$ is such that $\theta_k \leq -\delta$ for some 
fixed $\delta>0$: simply reverse the arguments in Steps 3 and 4. It therefore remains to explain
what happens when $\theta_k \approx 0$, i.e., $a_k \approx -\frac{h}{p}$. We then see from 
\eqref{eq:diffR=J} that $R'_{p,\beta,h}(a_k) \approx \beta^{-1} I'(a_k) < 0$, so that $a_k$ lies
in the interval $[\mathbf{t},0]$, which is \emph{beyond} the top state (recall Fig.~\ref{fig:CWfe}).

%%%%%%%%%%% REFERENCES %%%%%%%%%%%%%%%%%%%%%%%%%


\begin{thebibliography}{99}

\bibitem{BBI09}
A.\ Bianchi, A.\ Bovier and D.\ Ioffe,
Sharp asymptotics for metastability in the random field Curie-Weiss model,
Electron.\ J.\ Probab.\ 14 (2009) 1541--1603.
 
\bibitem{BBI12}
A.\ Bianchi, A.\ Bovier and D. Ioffe,
Pointwise estimates and exponential laws in metastable systems via coupling methods,
Ann.\ Probab.\ 40 (2012) 339--379.
 
\bibitem{BEGK01}
A.\ Bovier, M.\ Eckhoff, V.\ Gayrard and M.\ Klein,
Metastability in stochastic dynamics of disordered mean-field models, 
Probab.\ Theory Relat.\ Fields 119 (2001) 99--161.
 
\bibitem{BdH15} 
A.\ Bovier and F.\ den Hollander, 
\emph{Metastability -- A Potential-Theoretic Approach}, 
Grundlehren der mathematischen Wissenschaften 351, Springer, 2015.

\bibitem{BMPpr}
A.\ Bovier, S.\ Marello and E.\ Pulvirenti, 
Metastability for the dilute Curie-Weiss model with Glauber dynamics,
preprint 23 December 2019.
 
\bibitem{D17}
S.\ Dommers,
Metastability of the Ising model on random regular graphs at zero temperature,
Probab.\ Theory Relat.\ Fields 167 (2017) 305--324.

\bibitem{DGvdH10}
S.\ Dommers, C.\ Giardin\`a and R.\ van der Hofstad,
Ising models on power-law random graphs,
J.\ Stat.\ Phys.\ 141 (2010) 638--660. 

\bibitem{DGvdH14}
S.\ Dommers, C.\ Giardin\`a and R.\ van der Hofstad,
Ising critical exponents on random trees and graphs,
Commun.\ Math.\ Phys.\ 328 (2014) 355--395. 
 
\bibitem{DdHJN17}
S.\ Dommers, F.\ den Hollander, O.\ Jovanovski and F.R.\ Nardi, 
Metastability for Glauber dynamics on random graphs,
Ann.\ Appl.\ Probab.\ 27 (2017) 2130--2158.

\bibitem{vdH17}
R.\ van der Hofstad,
\emph{Random Graphs and Complex Networks, Volume I},
Cambridge University Press, 2017.

\bibitem{dHJ17}
F.\ den Hollander and O.\ Jovanovski,
Metastability on the hierarchical lattice,
J.\ Phys.\ A: Math.\ Theor.\ 50 (2017) 305001.

\bibitem{J17}
O.\ Jovanovski,
Metastability for the Ising model on the hypercube,
J.\ Stat.\ Phys.\ 167 (2017) 135--159.

\bibitem{OV05}
E.\ Olivieri and M.E.\ Vares,
\emph{Large Deviations and Metastability},
Encyclopedia of Mathematics and its Applications 100,
Cambridge University Press, Cambridge, 2005.

\end{thebibliography}
\end{document}